\documentclass[11pt, reqno]{amsart}
\usepackage{amsmath,amsthm,amsfonts,amssymb,mathrsfs,bm,graphicx,stmaryrd}
\usepackage{subcaption}
\usepackage{mathtools}
\usepackage{dsfont}
\usepackage{multicol}
\usepackage[colorlinks=true,linkcolor=blue,citecolor=blue]{hyperref}
\usepackage{bbm}

\newcommand{\wcV}{\widetilde{\mathcal{V}}}
\newcommand{\Fav}{\mathbf{Fav}}
\newcommand{\ubarr}{\underline{\mathrm{bar}}}
\newcommand{\Inside}{\mathbf{Inside}}
\newcommand{\boxx}{\mathrm{Box}}
\newcommand{\Barr}{\mathbf{Bar}}
\newcommand{\obarr}{\overline{\mathrm{bar}}}
\newcommand{\QQ}{\mathbb{Q}}
\newcommand{\Across}{\mathbf{Across}}

\newcommand{\wcL}{\widetilde{\mathcal{L}}}
\newcommand{\corr}{\mathrm{Corridor}}

\newcommand{\In}{\mathrm{In}}
\newcommand{\cI}{\mathcal{I}}
\newcommand{\RR}{\mathbb{R}}

\newcommand{\Reg}{\mathbf{Reg}}

\newcommand{\PP}{\mathbb{P}}
\newcommand{\EE}{\mathbb{E}}

\newcommand{\good}{\mathrm{good}}

\newcommand{\Z}{\mathbb{Z}}

\newcommand{\NN}{\mathbb{N}}
\newcommand{\NU}{\mathrm{NU}}
\usepackage[letterpaper,hmargin=1in,vmargin=1in]{geometry}
\parskip 	\smallskipamount

\newtheorem{theorem}{Theorem}
\newtheorem{lemma}[theorem]{Lemma}

\newtheorem{definition}[theorem]{Definition}

\newtheorem{proposition}[theorem]{Proposition}

\theoremstyle{definition}
\newtheorem{remark}[theorem]{Remark}

\def\Var{{\rm Var}}

\newcommand{\cR}{\mathcal{R}}
\newcommand{\cP}{\mathcal{P}}

\newcommand{\cA}{\mathcal{A}}
\newcommand{\cB}{\mathcal{B}}

\newcommand{\cE}{\mathcal{E}}
\newcommand{\cN}{\mathcal{N}}

\newcommand{\cL}{\mathcal{L}}

\newcommand{\cS}{\mathcal{S}}

\newcommand{\inte}{\mathrm{int}}

\newcommand{\ZZ}{\mathbb{Z}}
\newcommand{\dis}{\mathrm{dis}}

\newcommand{\cH}{\mathcal{H}}
\newcommand{\cT}{\mathcal{T}}
\newcommand{\cor}{\mathrm{Cor}}
\newcommand{\leb}{\mathrm{Leb}}
\usepackage[backend=biber,doi=false,style=alphabetic,url=false,maxalphanames=10,maxnames=50]{biblatex}
\AtBeginBibliography{\small}

\begin{document}
\title[]{A Peano curve from mated geodesic trees in the directed landscape}
\author[]{Riddhipratim Basu}
\address{Riddhipratim Basu, International Centre for Theoretical Sciences, Tata Institute of Fundamental Research, Bangalore, India} 
\email{rbasu@icts.res.in}
\author[]{Manan Bhatia}
\address{Manan Bhatia, Department of Mathematics, Massachusetts Institute of Technology, Cambridge, MA, USA}
\email{mananb@mit.edu}
\date{}
\maketitle
\begin{abstract}
For the directed landscape, the putative universal space-time scaling limit object in the (1+1) dimensional Kardar-Parisi-Zhang (KPZ) universality class, consider the geodesic tree-- the tree formed by the coalescing semi-infinite geodesics in a given direction. As shown in \cite{Bha23}, this tree comes interlocked with a dual tree, which (up to a reflection) has the same marginal law as the geodesic tree. Analogous examples of one ended planar trees formed by coalescent semi-infinite random paths and their duals are objects of interest in various other probability models, a classical example being the Brownian web, which is constructed as a scaling limit of coalescent random walks. In this paper, we continue the study of the geodesic tree and its dual in the directed landscape and exhibit a new space-filling curve traversing between the two trees that is naturally parametrized by the area it covers and encodes the geometry of the two trees; this parallels the construction of the T\'oth-Werner curve between the Brownian web and its dual. We study the regularity and fractal properties of this Peano curve, exploiting simultaneously the symmetries of the directed landscape and probabilistic estimates obtained in planar exponential last passage percolation, which is known to converge to the directed landscape in the scaling limit. On the way, we develop a novel coalescence estimate for geodesics, and this has recently found application in other work. 
\end{abstract}
\begin{figure}
  \centering
    \includegraphics[width=0.50\linewidth]{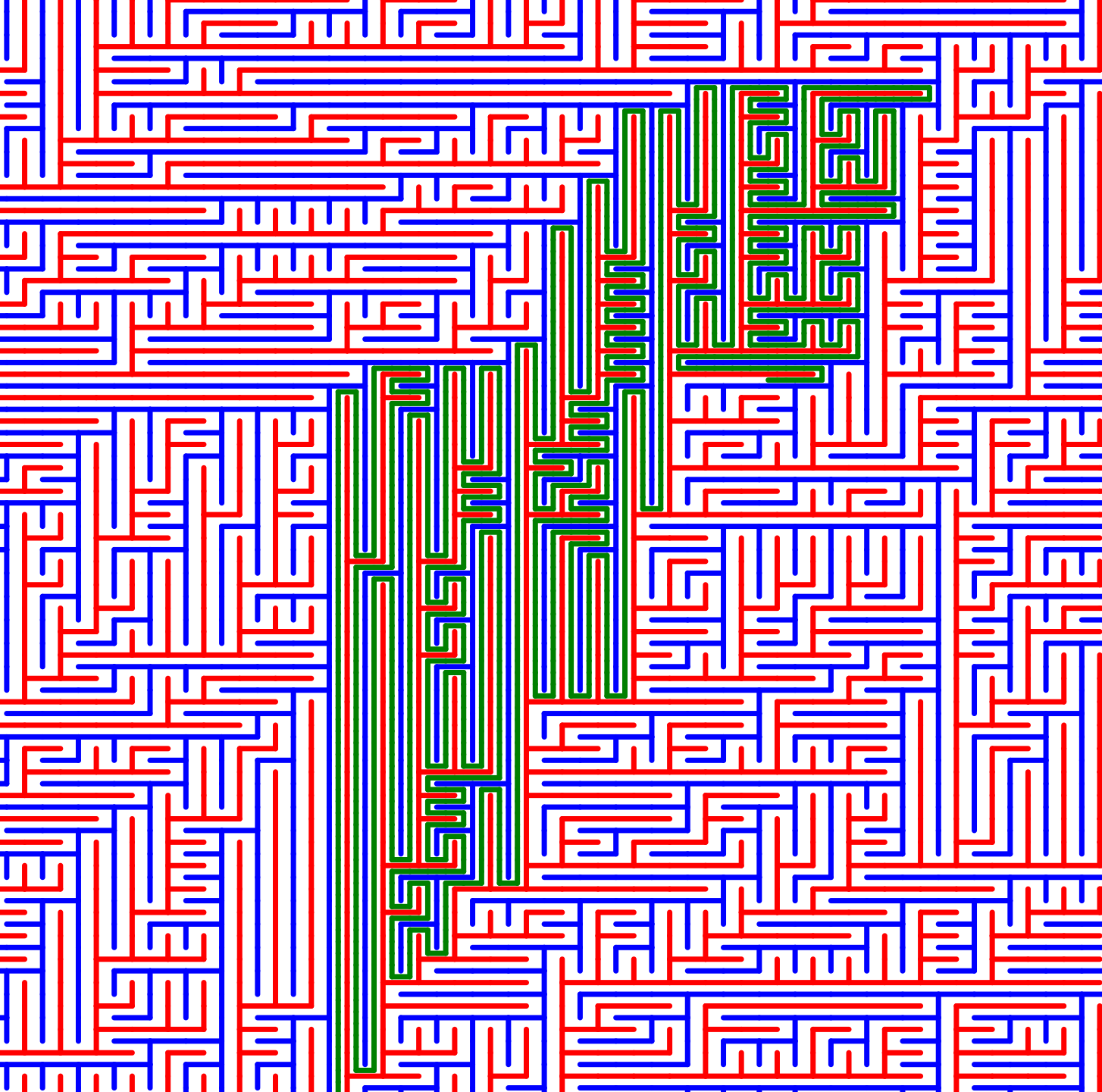}
  \caption{\textbf{A simulation of the Peano curve for exponential LPP}: The blue paths form the tree of semi-infinite geodesics in the $(1,1)$ direction, while the red paths form the dual tree, which we note has the same law as the geodesic tree up to a reflection \cite{Pim16}. The green curve drawn in between the two mated trees is a portion of the Peano curve, and the corresponding curve in the directed landscape is space-filling.}
  \label{fig:simul}
\end{figure}

\newpage
 \tableofcontents

\section{Introduction}
\label{sec:introduction}
First conceived in the doctoral thesis of Arratia \cite{Arr79,Arr81}, the Brownian web is an object consisting of the tree formed by coalescing one-dimensional Brownian motions starting from every point in $\RR^2$, with the vertical coordinate being viewed as the time coordinate and the horizontal axis being the space axis. Remarkably, the Brownian web enjoys a duality, in the sense that the above-mentioned tree can be interlaced with a tree of dual paths which completely avoid the primal paths, and it turns out that this dual tree is also marginally a Brownian web. A couple decades after the work of Arratia, T\'oth and Werner \cite{TW98} revived the study of the Brownian web by constructing a space-filling curve exploring the interface between the Brownian web and its dual, a curve which is now known as the T\'oth-Werner curve. In fact, after the above work, the theme of a space-filling curve {interleaving} between two ``mated'' trees has occurred a number of times in probability, with prominent examples being the Peano curve corresponding to the uniform spanning tree \cite{LSW04} on $\ZZ^2$ and the Schramm-Loewner evolution curves arising in the mating of trees \cite{DMS14} construction of Liouville quantum gravity \cite{She22}.

Apart from constructing the Peano curve corresponding to the Brownian web (the T\'oth-Werner curve), the work \cite{TW98} found an intriguing connection between the so-called ``true'' self-avoiding walk and the corresponding Peano curve for a family of coalescing simple random walks on the lattice. Using this connection, they were able to interpret the T\'oth-Werner curve as the true self-repelling motion (TSRM), the scaling limit of true self-avoiding walks. Moreover, by investigating the precise continuity and variation properties of the T\'oth-Werner curve, they uncovered fine information about the corresponding TSRM.

From the perspective of the above results and the connection to true self-avoiding walks, an important property of the Brownian web which makes it tractable is the independence present within. Indeed, if we reveal the Brownian path emanating upwards from a point $p\in \RR^2$, then the Brownian path from another point $q\neq p$ evolves independently till it encounters the above path and merges thereafter. In this paper, we shall give a construction of a Peano curve in the setting of the so-called (1+1) dimensional Kardar-Parisi-Zhang (KPZ) universality class, where the above independence structure is no longer present. Indeed, we shall work with the directed landscape, a model of planar random geometry constructed in \cite{DOV18} and believed to be the canonical scaling limit of a large class of growth models in (1+1) dimensional Kardar-Parisi-Zhang (KPZ) universality class. The goal of this work is to define a natural space-filling curve in the directed landscape-- the Peano curve between the geodesic tree and its dual, and to investigate its precise regularity and fractal properties. The stark difference in this setting when compared to the Brownian web is that if one reveals one of the strands comprising the geodesic tree, then this provides non-trivial information about distant portions of the geodesic tree. In particular, due to the rich correlation structure in this setting, there is no direct analogue of the iterative sampling of the Brownian web via coalescing Brownian motions.

We now introduce the directed landscape in some detail.
The directed landscape $\cL$ is a random real valued function on
\begin{equation}
  \label{eq:55}
  \RR_\uparrow^4=\{(x,s;y,t)\in \RR^4 :s<t\},
\end{equation}
with the value $\cL(x,s;y,t)$ to be interpreted as the ``distance" between two points $(x,s)$ and $(y,t)$ satisfying $s<t$. One can go a step further and define the weight $\ell(\psi)$ of a path $\psi\colon [s,t]\rightarrow \RR$ by
\begin{equation}
  \label{eq:56}
  \ell(\psi)=\inf_{k\in \NN}\inf_{s=t_0<t_1<\cdots<t_k=t}\sum_{i=1}^k \cL(\psi(t_{i-1}),t_{i-1};\psi(t_i),t_i).
\end{equation}
Now, a natural definition of a geodesic $\gamma_{(x,s)}^{(y,t)}$ between points $(x,s),(y,t)$ is a path $\psi$ which maximizes $\ell(\psi)$ while satisfying $\psi(s)=x$ and $\psi(t)=y$. As expected, it can be shown that geodesics exist almost surely between any two ordered points, and further, every fixed pair of points a.s.\ have a unique geodesic between them.

In fact, it can be shown that a.s.\ all points $(x,s)$ have an upward semi-infinite geodesic emanating out, in the sense that there is a path $\Gamma_{(x,s)}\colon [s,\infty)\rightarrow \RR$ with $\Gamma_{(x,s)}(s)=x$ and $\lim_{t\rightarrow \infty}\Gamma_{(x,s)}(t)/t\rightarrow 0$ along with the property that any finite segment of $\Gamma_{(x,s)}$ is a geodesic between its endpoints; we note that as in the above, we will always use $\Gamma_{(x,s)}$ to denote both a function of time and the graph of this function. For any fixed $(x,s)$, it can be shown that $\Gamma_{(x,s)}$ is a.s.\ unique, and in fact, the union of the interiors of $\Gamma_p$ over all $p\in \RR^2$ a.s\ forms a one-ended tree which we call the geodesic tree and denote by $\cT_\uparrow$. The above geodesic tree turns out to be interlaced with another tree \cite{Bha23, Pim16}, the interface portrait $\cI_\downarrow$, which is in fact equal to the set of points $p$ with multiple `downward' geodesics $\Gamma_{p,\downarrow}$ emanating out of them, where the latter are defined analogously to upward semi-infinite geodesics. With the motivations described earlier in mind, it is thus natural to consider the Peano curve $\eta$ snaking in between the above two mated trees $\cT_\uparrow$ and $\cI_\downarrow$, and we normalize $\eta$ such that $\eta(0)=0$ and parametrize it so that $\leb( \eta[v,w])=w-v$ for every $v<w\in \RR$, where $\leb$ denotes the Lebesgue measure on the plane. %

Having defined the Peano curve $\eta$ (the formal construction appears in Theorem \ref{thm:1} below), we seek to understand its precise smoothness, and in fact establish that when $\RR^2$ is endowed with the `intrinsic' metric (see \eqref{eq:54}) attuned to the symmetries inherent in the directed landscape, then $\eta$ is a.s.\ locally H\"older $1/5-$ continuous. The proof of the above is delicate and involves ruling out geodesics which have very few other geodesics merging into them throughout their length. This is difficult since it involves showing some degree of decorrelation in the environment around a geodesic, and the existing techniques \cite{DSV22,MSZ21} for doing so are not quantitative. To this end, we go back to the pre-limiting exponential LPP model (that converges to the directed landscape \cite{DV21}) and harness the understanding of the environment conditional on a geodesic using the FKG inequality and a barrier-style argument %
to show that with a \emph{superpolynomially} high probability, at least at some place along the length of the geodesic, the geodesics from nearby points merge quickly into the aforementioned geodesic. The crucial novelty and central difficulty here is that the above estimate is required to hold with \emph{superpolynomially} high probability since, until now, all coalescence estimates in the directed landscape or exponential LPP yield polynomial bounds.

We also show that the $\eta$ is not H\"older $1/5$ continuous with respect to the intrinsic metric, and in fact, we show that it has finite $5$th variation in a certain natural sense, and this turns out to be deterministic just as the case of quadratic variation for Brownian motion. Finally, we investigate some fractal aspects of the Peano curve-- we show that under mild conditions, the intrinsic Hausdorff dimension of a set $\cA$ directly determines the dimension of the set $\dim \eta^{-1} (\cA)$.

The techniques used in this paper form a juxtaposition of two approaches. The first among these attempts to harness the rich class of symmetries present in the directed landscape, and we use these extensively-- as can especially be noticed in the proof of Theorem \ref{thm:3}, where a Poisson process is specifically introduced to leverage these symmetries. The second approach is percolation theoretic and much more hands-on and makes a crucial appearance in the proof of the $1/5-$ H\"older continuity, where we first need to argue that the geodesic looks regular at many scales and then subsequently place barriers at a positive density of locations along its whole length.

We note that the results of this work have recently found application in another work of the second author \cite{Bha24+}, and we refer the reader to Section \ref{s:application} for a discussion of this.
\subsection{The main results}
\label{sec:main-result}
To set up the stage for defining the Peano curve, we first discuss infinite geodesics along with the corresponding geodesic tree and interface portrait. In the directed landscape, for every direction $\theta\in \RR$ and point $p=(x,s)$, there exists \cite{RV21,GZ22,BSS22} a path $\Gamma^\theta_p\colon[s,\infty)\rightarrow \RR$ satisfying $\Gamma^\theta_p(s)=x$, $\lim_{t\rightarrow \infty}\Gamma^\theta_p(t)/t=\theta$ along with the property that $\Gamma^\theta_p\lvert_{[s,t]}$ is a finite geodesic for every interval $[s,t]$. We refer to $\Gamma^\theta_p$ as a $\theta$-directed upward semi-infinite geodesic emanating from $p$, and note that for a fixed direction $\theta$ and a fixed point $p$, there is a.s.\ a unique \cite{RV21} such geodesic. We will always work with $\theta=0$ and in this case, we drop $\theta$ from the notation, and just call such a geodesic $\Gamma_p$ as an upward semi-infinite geodesic. Similarly, for $p=(x,s)$, we can analogously define downward $\theta$-directed semi-infinite geodesics $\Gamma_{p,\downarrow}^\theta\colon (-\infty,s]\rightarrow \RR$ which have $\Gamma_{p,\downarrow}^\theta(s)=x$ and satisfy $\Gamma_{p,\downarrow}^\theta(t)/t\rightarrow \theta$ as $t\rightarrow -\infty$, and we again drop $\theta$ from the notation if $\theta=0$. %
We define the upward and downward geodesic trees by $\cT_\uparrow=\bigcup_{p\in \RR^2}\inte \Gamma_p$, $\cT_\downarrow=\bigcup_{p\in \RR^2}\inte \Gamma_{p,\downarrow}$, where the union in both cases is over all points $p$ and all upward/downward geodesics emanating out of $p$, and by the interior of a semi-infinite path, we mean the path with its endpoint removed.

Even for the fixed direction $0$, there still do exist exceptional points $p\in \RR^2$ which admit multiple semi-infinite geodesics $\Gamma_p$, and we use $\NU_\uparrow^k(\cL)$ to denote the collection of points admitting at least $k$ such geodesics. In case of downward geodesics we denote the non-uniqueness set corresponding to $\NU_\uparrow^k(\cL)$ by $\NU_\downarrow^k(\cL)$.
The following result from \cite{Bha23} establishes a duality for the landscape by using the well-known pre-limiting duality between geodesics and competition interfaces from \cite{Pim16} and, in particular, shows that the set $\NU_\uparrow^2(\cL)$ itself has a natural tree structure.
\begin{proposition}[{\cite{Bha23}}]
  \label{prop:4}
  Almost surely, $\cT_\uparrow, \cT_\downarrow$ form one ended trees. Further, there exists a coupling of $\cL$ with a dual landscape $\wcL\stackrel{d}{=}\cL$ such that $\NU_\downarrow^2(\wcL)=\cT_\uparrow(\cL)$ and $\NU_\uparrow^2(\cL)=\cT_\downarrow(\wcL)$. %
 With $\cI_\downarrow=\cI_\downarrow(\cL)$ denoting the latter object, which we call the interface portrait, the sets $\cT_\uparrow$ and $\cI_\downarrow$ are a.s.\ disjoint. 
\end{proposition}

We note that in \cite{Bha23}, the interface portrait $\cI_\downarrow$ has its own separate definition, and the relation $\NU_\uparrow^2(\cL)=\cI_\downarrow$ is obtained as a theorem. However, for the purpose of this paper, it is more efficient to take the latter as a definition and this is what we do. The trees $\cT_\uparrow,\cI_\downarrow$ should be thought of as mated trees (see Figures \ref{fig:simul}, \ref{fig:simul1} for simulations in exponential LPP). As in the above proposition, we will use $\wcL$ to denote the dual landscape throughout the paper and will further use $\Upsilon_p$ to denote a geodesic $\Gamma_{p,\downarrow}(\wcL)$, where we note that there will exist points $p$ for which $\Upsilon_p$ is not unique. The paths $\Upsilon_p$ will usually be called interfaces. We now state a theorem introducing the Peano curve $\eta$ and describing some of its fundamental properties.

\begin{theorem}
  \label{thm:1}
  There is a random space-filling curve $\eta=(\eta_h,\eta_u)\colon \RR\rightarrow \RR^2$ measurable with respect to $\cL$ and satisfying $\eta(0)=0$ with the following additional properties.
  \begin{enumerate}
  \item Almost surely, for every $v<w\in \RR$, $\leb (\eta([v,w]))=w-v$, where $\leb$ denotes the Lebesgue measure on the plane.
  \item Almost surely, for (Lebesgue) a.e.\ point $p$ in $\RR^2$, there is a unique $v$ for which $\eta(v)=p$ and moreover, the set of points $p$ for which the above $v$ is non-unique is equal to the set $\cT_\uparrow\cup \cI_\downarrow$ and a.s.\ has Euclidean Hausdorff dimension $4/3$.
  \item Almost surely, the set of points $p$ for which there exist three values of $v$ with $\eta(v)=p$ is equal to $\NU_\uparrow^3(\cL)\cup \NU_\downarrow^3(\wcL)$. Moreover, there is a.s.\ no point $p\in \RR^2$ admitting four values of $v$ for which $\eta(v)=p$.
  \item The curve $\eta$ exhibits translation invariance in the sense that for any fixed $v_0\in \RR$,
     \begin{displaymath}
      (\eta(\cdot+v_0)-\eta(v_0),\cL(\cdot+\eta(v_0);\cdot+\eta(v_0)))\stackrel{d}{=} (\eta(\cdot),\cL(\cdot;\cdot)).
    \end{displaymath}
  \item The curve $\eta$ respects $\mathrm{KPZ}$ scaling in the sense that for any fixed $\beta>0$, 
    \begin{displaymath}
      \eta(\beta v)\stackrel{d}{=}(\beta^{2/5}\eta_h(v),\beta^{3/5}\eta_u(v)),
    \end{displaymath}
    where the distributional equality holds as parametrized curves.
  \item The curve $\eta$ has time reversal symmetry in the sense that $-\eta(-v)\stackrel{d}{=}\eta(v)$.
  \end{enumerate}
\end{theorem}
It should be possible to isolate a minimal list of properties which determine the curve $\eta$ up to reflection/ time-reversal symmetries. However, we do not carry out the above and instead give a detailed hands-on construction of the Peano curve $\eta$ in Section \ref{sec:constr}, and the rest of the results pertain to this particular curve $\eta$. 

On encountering a random curve, a natural question is to investigate its regularity properties and this is what we turn to next. In order to present the most natural formulation of the result, we introduce the intrinsic metric for the directed landscape-- such a notion has been used in the works \cite{CHHM21, Dau23+}. For points $(x,s), (y,t)\in \RR^2$, we will often work with the metric
\begin{equation}
  \label{eq:54}
  d_\mathrm{in}((x,s),(y,t))=|x-y|^{1/2}+|s-t|^{1/3},
\end{equation}
and we will simply refer to the above as the intrinsic metric and use $\dim_\mathrm{in}$ to denote Hausdorff dimensions with respect to it. %
Throughout the paper, we will often work with different fractals $\cA\subseteq \RR$ will use $\dim \cA$ to denote the Hausdorff dimension with respect to the Euclidean metric on $\RR$. We now state our results on the the H\"older continuity of $\eta$.

\begin{theorem}
  \label{thm:2}
  The curve $\eta$ is a.s.\ locally $1/5-$ H\"older continuous with respect to the intrinsic metric, by which we mean that, almost surely, for any $\epsilon>0$ and $M>0$, $\eta\lvert_{[-M,M]}$ is $1/5-\epsilon$ H\"older continuous with respect to the metric $d_{\mathrm{in}}$ on $\RR^2$. 
\end{theorem}

\begin{theorem}
  \label{thm:5}
  Almost surely, for all intervals $I\subseteq \RR$, $\eta\lvert_I$ is not $1/5$ H\"older continuous with respect to the intrinsic metric.
\end{theorem}
Though $\eta$ fails to be $1/5$ H\"older, one might hope that just as in the parallel case of Brownian motion, $\eta$ still has a well defined $5$th ``variation'' which is also deterministic, and we will show a version of this. For any $\alpha>0$, a given $f\colon\RR\rightarrow \RR^2$ with coordinates $f_h,f_u$, define $\In^\alpha_v(f,\epsilon)$ as the $\alpha$th ``variation'' increment
  \begin{equation}
    \label{eq:38}
    \In^\alpha_{v}(f,\epsilon)=|f_u(v+\epsilon)-f_u(v)|^{\alpha/3}+|f_h(v+\epsilon)-f_h(v)|^{\alpha/2}.
  \end{equation} %
With the above notation at hand, we have the following result.

\begin{theorem}
  \label{thm:3}
  Let $\Pi_n^1$ denote a Poisson point process on $\RR$ with rate $n$ independent of $\cL$. For any fixed interval $I$, let $\Pi^1_n\cap I=\{v^n_1,\dots,v^n_{\tau^n}\}$ %
  for a random integer $\tau^n$, with the $v^n_i$ being increasing in $i$. Then the sum
  \begin{equation}
    \label{eq:3}
    \sum_{i=1}^{\tau^n-1} \In^5_{v^n_i}(\eta,v^n_{i+1}-v^n_i)
  \end{equation}
 converges to $\EE \In^5_0(\eta,1)\leb(I)$ in probability as $n\rightarrow \infty$.
\end{theorem}
We note that a more natural formulation of the above result on the existence of $5$th variation appears as Proposition \ref{prop:2}. Also, we remark that by using Theorem \ref{thm:2} and Theorem \ref{thm:5}, it can be obtained that if the $5$ in \eqref{eq:3} is replaced by $\alpha$, then the sum converges in probability to $0$ or $\infty$ as $n\rightarrow \infty$ in the cases $\alpha>5$ and $\alpha< 5$ respectively.%

We now come to the final main result of the paper, and this deals with the question of how the fractal properties for sets $\cA\subseteq \RR^2$ transform when pulled back by $\eta$ to a subset of $\RR$. Morally, it is true that for a large class of sets $\cA$, the Hausdorff dimension $\dim \eta^{-1}(\cA)=(\dim_\mathrm{in} \cA)/5$, but the precise statement is more complicated and we state it now.
\begin{theorem}
  \label{thm:4}
  Almost surely, simultaneously for every set $\cA\subseteq \RR^2$, we have $\dim\eta^{-1}(\cA)\geq (\dim_{\mathrm{in}} \cA)/5$. On the other hand, for a fixed $M>0$, let $\cA\subseteq [-M,M]^2$ be a random closed set which is measurable with respect to $\cL$ with the additional property that there exists a $d>0$ and a constant $C>0$ such that $\EE\leb(B^\mathrm{in}_\epsilon(\cA))\leq C\epsilon^{5-d}$ for all $\epsilon \in (0,1)$, where $B^\mathrm{in}_\epsilon(\cA)$ denotes the intrinsic $\epsilon$ neighbourhood of $\cA$. Then we have the a.s.\ inequality $\dim \eta^{-1}(\cA)\leq d/5$.
\end{theorem}
Though the upper bound part of the above statement a priori seems weaker than the lower bound, it is in practice good enough to generate matching upper bounds for a large class of fractals. We refer the reader to Remark \ref{rem:1} for a detailed discussion of this point. We also mention that Theorem \ref{thm:4} is analogous to the KPZ dimension change formulae \cite{RV08,DS09} which arise in Liouville quantum gravity, with the work \cite{GHM20} being an exact analogue.
\begin{figure}
  \centering
    \includegraphics[width=0.8\linewidth]{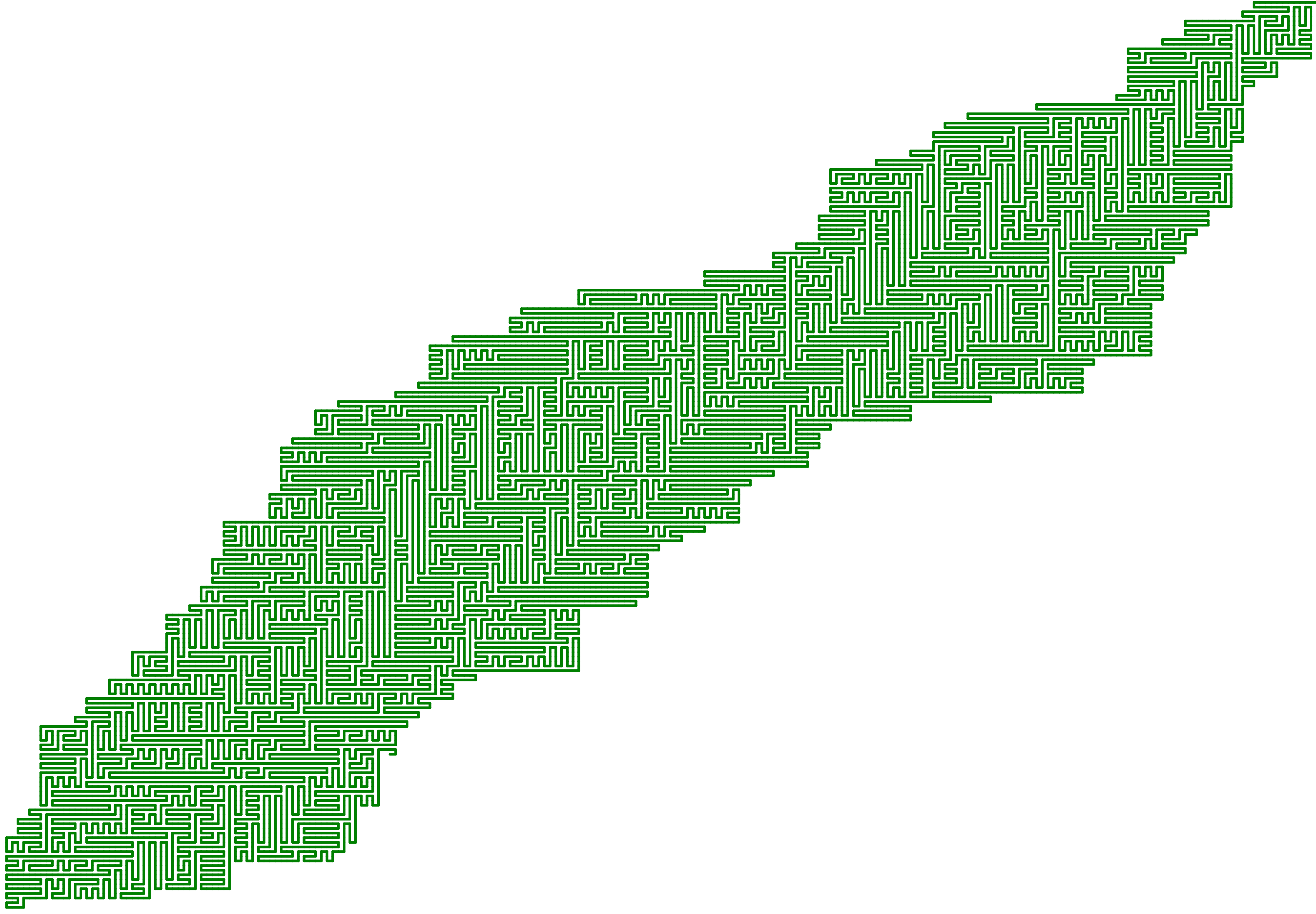}
  \caption{A large simulation showing a portion of the Peano curve corresponding to the tree formed by $(1,1)$ directional semi-infinite geodesics in exponential LPP}
  \label{fig:simul1}
\end{figure}

Having stated the main results, we now state two auxiliary results, which are the core technical ingredients in the proofs of Theorem \ref{thm:5} and Theorem \ref{thm:2} respectively, but are of independent interest as well. The following, together with a result on the decorrelation of well-separated parts of the Peano curve (Lemma \ref{lem:19}), will be used to prove Theorem \ref{thm:5}.

\begin{proposition}
  \label{lem:9}
  The distribution of the random variable $\eta_u(1)$ has unbounded support.
\end{proposition}
In fact, a result similar to the above can be obtained for $\eta_h(1)$ as well, and this is discussed in Remark \ref{rem:unb}.

Next, for the proof of Theorem \ref{thm:2}, we will obtain stretched exponential tail estimates for $\eta_{u}(1)$ and $\eta_{h}(1)$ (Propositions \ref{p:utail} and \ref{p:htail}). For this we shall need to argue that it is stretched exponentially rare for geodesics to have very few other geodesics merging with them on either side, and the next result is the core estimate which addresses this. Let $\wcV_R(\Gamma_0;1)$ denote the set of points $q=(y,t)$ which are to the right of $\Gamma_0$ in the sense that $t\in [0,1]$ and $y>\Gamma_0(t)$, and further, admit a geodesic $\Gamma_q$ satisfying $\Gamma_q(s)=\Gamma_0(s)$ for all $s\geq 1$. Using $\widetilde{V}_R(\Gamma_0;1)$ to denote the value $\leb(\wcV_R(\Gamma_0;1))$, we have the following.

\begin{proposition}
 \label{t:va}
 There exist constants $c,C>0$ such that for all $\epsilon>0$, we have $\PP(\widetilde{V}_{R}(\Gamma_0;1)\le \epsilon)\le C\exp(-c\epsilon^{-3/11})$. 
\end{proposition}

At this point we would like to note that the exponent $3/11$ in the above proposition is not expected to be optimal, but it suffices for our purposes. Note also that that the probability in Proposition \ref{t:va} can be easily shown to be upper bounded by a polynomial in $\epsilon$, but we need a super-polynomial upper bound for our proof of Theorem \ref{thm:2} as mentioned above (see the proof of Proposition \ref{t:holderu}).

  \subsection{Applications of this work}
  \label{s:application}
  Techniques from this paper have recently found application in the work \cite{Bha24+} and we now briefly discuss this. Often in last passage percolation, `point-to-point' estimates are tractable, but corresponding `segment-to-segment' estimates are much harder.

  For instance, let us consider exponential LPP (see Section \ref{s:lppnotation} for important notation including space-time coordinates). Locally, for $n\in \ZZ$, we shall use $\mathbf{n}$ to denote $(n,n)$ and $\ell_n$ to denote the line segment $\{|x|\leq |n|^{2/3}, t=n\}$, where here we are using space-time coordinates. Now, it is trivial to see that the geodesic $\gamma_{-\mathbf{n},\mathbf{n}}^{\dis}$ intersects the line $\{t=0\}$ exactly once, where here, we are using space-time coordinates. However, obtaining control on the cardinality of the set
  \begin{equation}
    \label{eq:17}
    \{\gamma_{p,q}^{\dis}\cap \{t=0\}: p\in \ell_{-n}, q\in \ell_{n}\}
  \end{equation}
is not easy. Indeed, the work \cite[Theorem 3.10]{BHS22} obtained a stretched exponential tail estimate via a BK inequality argument; in particular, they established that with superpolynomially high probability in $n$, the above cardinality is at most $\mathrm{polylog}(n)$.

Using the volume accumulation result Proposition \ref{t:vaexp}, \cite[Section 4.2.3]{Bha24+} provides a general recipe using which, in many scenarios involving integrable last passage percolation models, one can directly upgrade a point-to-point tail estimate to a corresponding segment-to-segment one. Very roughly, the idea is that with some work, Proposition \ref{t:vaexp} can be used to show the following-- that for any fixed points $p\in \ell_{-n},q\in \ell_{n}$, with stretched exponentially high probability in $\epsilon$, the volume of the set of points $(p',q')\in (\ZZ^2)^2$ such that
\begin{equation}
  \label{eq:18}
  \gamma_{p',q'}^{\dis}\cap \{|t|\leq n/2\}=\gamma_{p,q}^{\dis}\cap \{|t|\leq n/2\}
\end{equation}
is at least $\epsilon n^{10/3}$. Thus, for an appropriate choice $\epsilon = n^{-o(1)}$, we can do a union bound over the polynomially many possible choices of $p\in \ell_{-n}, q\in \ell_{n}$ to obtain a uniform version of \eqref{eq:18}, where each corresponding volume is at least $n^{10/3-o(1)}$ with superpolynomially high probability in $n$. As a result, if we want a certain event to hold for all $\gamma_{p,q}^{\dis}\cap\{|t|\leq n/2\}$ for $p\in \ell_{-n},q\in \ell_{n}$, we can first demand that the event holds for all points $(p',q')$ sprinkled according to an independent Poisson process $\cP$ with a suitably chosen $\Theta(n^{-10/3+o(1)})$ intensity, then argue that it is very likely that for any $p\in \ell_{-n},q\in \ell_{n}$, there is at least one $(p',q')\in \cP$ satisfying \eqref{eq:18}. As a result, if we start with a high probability lower bound for the above event for a geodesic between fixed points, then we obtain a corresponding lower bound on the probability that the event holds simultaneously for all geodesics between $p\in \ell_{-n},q\in \ell_{n}$. We refer the reader to \cite[Section 4.2.3]{Bha24+} for a detailed description of the above approach.

  Besides the above-mentioned application, results of this work shall find application in a series of upcoming works of the second author, at the end of which it will be shown that the directed landscape $\cL$ is determined by the parametrised geodesic tree $\cT_\uparrow$, by which we mean the set $\cT_\uparrow$ along with the lengths of all geodesics contained in it. We now give a short description of this project which is a work-in-progress.
\begin{enumerate}
\item The upcoming work \cite{Bha25+} shall yield a conditional independence result for the bubbles created by the Peano curve $\eta$. To be specific, it will be shown that for any fixed $v_1<\dots <v_n\in \RR$, the objects $\cL_{\eta[v_i,v_{i+1}]}$ defined as the directed landscape restricted to the Peano bubbles $\eta[v_i,v_{i+1}]$, are mutually independent conditional on the Busemann function $\cB$ restricted to the union $\bigcup_{i=1}^n\partial_{v_i}\eta$, where $\partial_{v_i}\eta$ is the topological boundary of the set $\eta(-\infty,v_i]$. Here, $\cB$ refers to the Busemann function corresponding to geodesic tree $\cT_\uparrow$ and can be defined as $\cB(p,q)=\cL(p,z)-\cL(q,z)$, where $z$ is any point satisfying $z\in\Gamma_p\cap\Gamma_q$. 

\item Subsequently, in another upcoming work \cite{Bha25++}, it shall be established that the conditional variance of any directed landscape length $\cL(p;q)$ given the parametrised geodesic tree $\cT_\uparrow$ is almost surely equal to zero, thereby establishing that the directed landscape can a.s.\ be reconstructed from the parametrised geodesic tree. The proof of this shall use the stretched exponential tail estimates on the tails of $\eta_u(1), \eta_h(1)$ developed in this paper (Propositions \ref{p:utail}, \ref{p:htail}) along with the conditional independence result from the above-mentioned upcoming work \cite{Bha25+}. Further, we shall also require results from the work \cite{Bha23} on atypical stars on geodesics and the result from \cite{Bha24} on the metric removability of interfaces in the directed landscape.
\end{enumerate}

\paragraph{\textbf{Notational comments}} The cardinality of a finite set $A$ is often denoted by $\# A$. For $a<b\in \RR$, we define $[\![a,b]\!]=[a,b]\cap \ZZ$. Throughout the paper, we will use $\leb$ to denote the Lebesgue measure, both on $\RR$ and on $\RR^2$. Also, we will simply use $0$ to denote both the real number $0$ and the point $(0,0)\subseteq \RR^2$.

\paragraph{\textbf{Organization of the paper}}
The rest of the paper is organized as follows. In Section \ref{sec:outline}, we give an overview of the proofs of the main results. In Section \ref{sec:preliminaries}, we construct the Peano curve and prove Theorem \ref{thm:1}. Here, we also recall a number of results from the directed landscape and exponential LPP literature that will be useful for us in the remainder of the paper. Section \ref{sec:holder} completes the proof of Theorem \ref{thm:2} assuming Proposition \ref{t:va} and the proof of Proposition \ref{t:va} is completed in Section \ref{sec:lpp}. Section \ref{sec:frac} proves Theorem \ref{thm:4}. Section \ref{sec:var} contains the proofs of Theorems \ref{thm:5} and \ref{thm:3} and also the proof of Proposition \ref{lem:9}.

\paragraph{\textbf{Acknowledgments}} We thank B\'alint Vir\'ag for bringing to our attention the similarity with the T\'oth-Werner curve and Milind Hegde for  help with simulations. MB thanks Sky Cao, Scott Sheffield and Catherine Wolfram for the discussions. We also thank the International Centre for Theoretical Sciences, Bangalore for the hospitality during the conference `Topics in High Dimensional Probability' in January 2023 during which this work was initiated. RB is partially supported by a MATRICS grant (MTR/2021/000093) from SERB, Govt.\ of India, DAE project no.\ RTI4001
via ICTS, and the Infosys Foundation via the Infosys-Chandrasekharan Virtual Centre for Random Geometry of TIFR. MB acknowledges the partial support of the Institute for Advanced Study and the NSF grants DMS-1712862 and DMS-1953945.

\section{Outline of the proofs}
\label{sec:outline}

We now give a brief sketch of the proof of the main results while also pointing out the technically difficult and/or novel elements in the paper. 

\subsection{Construction of the Peano curve $\eta$}
\label{sec:constr-outline}

 The construction of the Peano curve is not very complicated. In the discrete set-up of the geodesic tree and its dual in exponential LPP, it is easy to see that one can construct the interface between the two trees in the following way. Starting from a given point, say 0, in the positive direction, we move along edges while always keeping the edges of the primal tree to the left (the choice of left and right here is arbitrary) and the edges of the dual tree to the right. The directions are reversed while moving in the negative direction. It is easy to see that coalescence of geodesics imply that this procedure leads to a ``space-filling" path. To make a similar construction in the continuum, we do the following. We define a corner to be a point $p$ in $\mathbb{R}^2$ together with a pair of primal and dual geodesics emanating from $p$. Owing to the existence of exceptional points with non-unique geodesics, we note that a point $p$ can correspond to multiple corners. The idea is show that any the pair of paths corresponding to any two distinct corners are ordered, i.e., one of them always lie to the left of the other (see Lemma \ref{lem:main:1}), and the Peano curve is defined to be a curve that traces out the points of $\mathbb{R}^2$ according to this total ordering (notice again that a point can be visited multiple times by the Peano curve if it corresponds to multiple corners). Since each point corresponds to at least one corner, such a curve is necessarily space-filling, and it remains to show the continuity and existence of volume parametrisation. For the latter result, it suffices to show that for any two distinct corners $\zeta_1$ and $\zeta_2$, the set of all points in $\mathbb{R}^2$ corresponding to corners in between $\zeta_1$ and $\zeta_2$ has positive Lebesgue measure; this is done in Lemma \ref{lem:1}. The continuity of the Peano curve follows from the convergence of geodesics (see Proposition \ref{prop:7}). This argument establishes part (1) of Theorem \ref{thm:1}. Parts (2) and (3) of this theorem are consequences of results in \cite{Bha23, Bha22, Dau23+} while parts (4)-(6) follow from well known symmetries of the directed landscape; we shall not elaborate for on these here. 

\subsection{The $1/5-$ H\"{o}lder regularity of $\eta$}
\label{sec:holder-outline}
By the scaling symmetry of the directed landscape and the definition of intrinsic metric, if suffices to show that $\eta_u$ is almost surely $3/5-$ H\"{o}lder continuous (Proposition \ref{t:holderu})  and $\eta_h$ is almost surely $2/5-$  H\"{o}lder continuous (Proposition \ref{t:holderh}) on any compact set. For this outline section, let us only focus on the first result. In view of the translation invariance of $\eta$ ((4) in Theorem \ref{thm:1}), the proof of Proposition \ref{t:holderu} is reminiscent of the standard proof of $1/2-$ H\"{o}lder continuity of Brownian motion and requires only a one point estimate which says that $\max_{v\in [0,w]}|\eta_u(v)|$ has a stretched exponential tail at the scale $w^{3/5}$ (Proposition \ref{p:utail}). By the scaling symmetry of the directed landscape, it suffices to show that $\PP(\max_{v\in [0,1]}\eta_u(v)>t)$ is stretched exponentially small in $t$. 

To understand why the above probability should decay fast, consider the following. By transversal fluctuation estimates, one would expect that when $\eta_u$ reaches the value $t$, $\eta_h$ should have approximately reached $t^{2/3}$, and hence the volume parameter of the Peano curve $\eta$ should be at least $t^{5/3}$ at this point. The event described above implies that this happens before volume $1$, which should be unlikely for $t$ large. 

To make this precise, recall the notion $\widetilde{V}_R(\Gamma_0;1)$ (we call it the volume accumulated to the right by a geodesic up to time $1$; this easily generalizes to volume accumulated to the right by a geodesic up to time $t$; see Definition \ref{d:va}) which measures the volume of the set of all points to the right of the  geodesic $\Gamma_0$ that have coalesced with it by time (height) $1$. It is not hard to see (Lemma \ref{l:vv}) that on the event $\{\max_{v\in [0,1]}\eta_u(v)>t\}$, one has 
$\widetilde{V}_R(\Gamma_0;t)\le 1$. By the scaling symmetry of the directed landscape $\widetilde{V}_R(\Gamma_0;t)$ has the same law as $t^{5/3}\widetilde{V}_R(\Gamma_0;1)$ and our required tail estimate follows from Proposition \ref{t:va}. 

We now provide a brief outline of the argument for establishing Proposition \ref{t:va}, which is one of the main technical estimates of this paper. For this, we work with the exponential LPP model on $\Z^2$ (Theorem \ref{t:vaexp}) which enables us to use tools like the FKG inequality. For the semi-infinite geodesic $\Gamma$ from the origin in the vertically upward direction (in the usual exponential LPP coordinates this corresponds to the direction $(1,1)$; see Section \ref{sec:preliminaries} for more details) we look at the set of all vertices to the right of $\Gamma$ such that the semi-infinite geodesics starting from these vertices coalesce with $\Gamma$ before time $n$. Denoting the number of such vertices by $V_R(\Gamma;n)$ we show in Proposition \ref{t:vaexp} that the probability that $V_R(\Gamma;n)$ is smaller than $\delta n^{5/3}$ is stretched exponentially small in $\delta^{-1}$ for small $\delta$ and invoke the convergence of exponential LPP to the directed landscape (Proposition \ref{prop:6}) to establish Proposition \ref{t:va}. 

The proof of Proposition is by what is often referred to as a \emph{percolation argument}. We show that a constant fraction of the locations along the geodesic (a location refers to a time interval of length $\epsilon n$ for small $\epsilon$) are \emph{good} with large probability (i.e., the probability of the complement is exponentially small in $\epsilon^{-1}$) where a location being good refers to certain local events occurring which attracts geodesics from nearby points (i.e., within $(\epsilon n)^{2/3}$ distance of $\Gamma$) to coalesce with $\Gamma$. On the event that there exists a constant fraction of good locations, it follows that $V_R(\Gamma;n)$ is at least of the order $\approx \epsilon^{2/3} n^{5/3}$ and hence the claimed result follows. To show that there is a positive density of good locations with the required probability; the argument goes roughly as follows. The good local events are divided into two parts, ones the are typical, and others that are hold only with probability bounded away from $0$ at a given location (typically these events asks for existence of a region which is costly for the geodesic to pass through, often referred to as a \emph{barrier}). A standard percolation argument shows that, with large probability, the likely events occur at most locations. One needs to define the atypical events carefully so that they are monotone decreasing. By the FKG inequality, conditioning on the geodesic makes the regions not on the geodesic stochastically smaller, and this allows us to show that the atypical events hold at a constant fraction of locations along the geodesics as well.

\subsection{The 5th variation of the Peano curve}
\label{sec:variation-outline}
We now discuss the proof of Theorem \ref{thm:3}-- the result on the $5$th variation of the Peano curve $\eta$. The statement of Theorem \ref{thm:3} uses a Poisson process $\Pi_n^1\subseteq \RR$ that is independent $\cL$, and in order to explain the idea and utility of introducing this, let us first consider a simpler case where this Poisson process is replaced by a deterministic mesh, as is, for example, typically done when considering the quadratic variation of Brownian motion.

Indeed, suppose that, for an interval $I\subseteq \RR$, we defined a deterministic mesh of points $\{w_i^n\}_i\subseteq I$ at spacing $1/n$ and considered the quantity $\sum_{i}\In^5_{w^n_i}(\eta, w_{i+1}^n-w_i^n)$ with the aim of showing that as $n\rightarrow \infty$, the above quantity converges in probability to $\EE \In_0^5(\eta,1)\leb(I)$. Intuitively, in order to obtain the above, we need two necessary ingredients. Firstly, we require that the law of the increment $\In^5_{w_i}(\eta, w_{i+1}^n-w_i^n)$ be the same for all $i$, and this is in fact true as the increment $\eta(w_{i+1}^n)-\eta(w_i^n)$ itself has the same law for all $i$ (Lemma \ref{lem:6}). However, apart from the above, one also requires sufficient decorrelation between the quantities $\In_0^5(\eta,1)\leb(I)$ as $i$ varies in order for a law of large numbers based heuristic to apply, and this is more challenging as the Peano curve $\eta$ is a non-trivial object which takes the global geometry of the landscape $\cL$ into account.

The utility of the Poisson process based construction in Theorem \ref{thm:3} is that it very naturally provides the above decorrelation. The central point here is that since the curve $\eta$ is equipped with the volume parametrization, even if we condition on the curve $\eta$, the point process $\Pi_n^2=\eta(\Pi_n^1)$ is a Poisson point process on $\RR^2$ with rate $n$. In other words the process $\Pi_n^2$ is independent (Lemma \ref{lem:8}) of the landscape $\cL$. Now, if for $p=(x,s)\in \RR^2$, we define the point $\tilde p_n=(\tilde x_n, \tilde s_n)\in \Pi_n^2$ as the one in $\Pi_n^2$ visited by $\eta$ next after visiting $p$, then it can be verified that instead of having to consider an interval $I\subseteq \RR$ and arguing ``sufficient'' decorrelation of the increments $\In_{v_i^n}(\eta, v_{i+1}^n-v_i^n)$ with $v_i^n\in \Pi_n^1\cap I$, we can instead work with a rectangle $R\subset \RR^2$ and show the corresponding decorrelation for the sum $\sum_{p\in \Pi_n^2\cap R}\nu_n(\{p\})$, where $\nu_n(\{p\})=|\tilde s_n-s|^{5/3}+ |\tilde x_n -x|^{5/2}$ is the variation increment accumulated as $\eta$ traverses between $p,\tilde p_n$.

Now, since $\Pi_n^2$ is just a Poisson point process of rate $n$ independent of $\cL$, conditional on the number of points $|\Pi_n^2\cap R|$, the set $\Pi_n^2\cap R$ is just a family of i.i.d.\ uniform points in $R$. Thus, with some work utilizing the Palm theory of Poisson processes, one can reduce the task to obtaining a weak decorrelation estimate for $\nu_n(\{p\})$ as the point $p$ is varied in the set $R$. The reason why this holds is that since $\Pi_n^2$ is of rate $n$, the Peano curve with high probability stays close (roughly at a $d_{\mathrm{in}}$ distance $\sim n^{-1/5}$) to $p$ as it journeys from $p$ to $\tilde{p}_n$. Thus, with high probability, the increment $\nu_n(\{p\})$ is determined by the geometry of the geodesic tree $\cT_\uparrow$ in a small (that is, $o_n(1)$) rectangle $R_{p,n}$ around $p$. Now, by exploiting the coalescence of geodesics in the directed landscape, it can be shown (Lemma \ref{lem:19}) that with high probability, the above geometry depends on the directed landscape $\cL$ restricted to a slightly larger but still small rectangle $R'_{p,n}$. Since the landscape $\cL$ restricted to disjoint boxes is independent, the above yields a decorrelation estimate for the quantity $\nu_n(\{p\})$ as $p$ varies, and this turns out to be sufficient to yield Theorem \ref{thm:3}.
\subsection{The curve $\eta$ is a.s.\ not $1/5$ H\"older regular}
\label{sec:not-holder-outline}
We now give a short discussion of the proof of Theorem \ref{thm:5}. As opposed to Theorem \ref{thm:2}, the task now is to show that $\eta$ is not H\"older 1/5. The strategy for doing so again utilizes the local dependency argument from the previous paragraph using which it can be shown that for $\delta>0$ and far away points $p,q\in \RR^2$, the increments $\In^5_{\eta^{-1}(p)}(\eta,\delta)$ and $\In^5_{\eta^{-1}(q)}(\eta,\delta)$ are approximately independent with the independence strengthening as $\delta\rightarrow 0$ and for larger values of $d_{\mathrm{in}}(p,q)$.

This approximate independence is then combined with a barrier argument, which we now very roughly describe. For a fixed point $p\in \RR^2$, we consider an event on which the point $p$ is sandwiched between two interfaces $\Upsilon_1, \Upsilon_2\subseteq \cI_\downarrow$ which do not coalesce with each other for a ``long time'' and which also stay ``atypically close'' to each other -- see Figure \ref{fig:11}. We show that for any specification of the phrases ``long time'' and ``atypically close'' above, the event in question still has a positive, albeit small probability-- this is done by first using the duality between interface portraits and geodesic trees (Proposition \ref{prop:4}) to convert to a question purely about geodesics and then placing appropriate barriers in terms of passage times to force these geodesics to take the desired routes.

The utility of the event discussed above is that since the curve $\eta$ cannot cross interfaces, on the above event, it is forced to undertake a rapid increase in height while covering a small amount of volume. In effect, we show that for any $M>0$ and $\delta>0$, with positive probability (say $c_M$), the increment $\In^5_{\eta^{-1}(p)}(\eta,\delta)$ is at least $M\delta$. Now, a consequence of this and the approximate independence from earlier is that that for any fixed collection of $K$ distinct points $p_1,\dots, p_K\in \RR^2$, we have
\begin{equation}
  \label{eq:15}
  \limsup_{\delta\rightarrow 0}\PP(\exists i\in [\![1,K]\!]: \In^5_{\eta^{-1}(p_i)}(\eta,\delta)\geq M\delta)\geq 1-(1-c_M)^K,
\end{equation}
and by simply sending $K\rightarrow \infty$, it can be concluded that a.s.\ there exists an infinite sequence of points $q_i$ lying in a compact set such that $\In^5_{\eta^{-1}(q_i)}(\eta,\delta)\geq M\delta$ for all $i$. Since $M$ is arbitrary, this shows that $\eta$ is a.s.\ not $1/5$ H\"older continuous.
\section{Preliminaries}
\label{sec:preliminaries}

\subsection{Construction of the Peano curve}
\label{sec:constr}
In this section, we construct the Peano curve $\eta$. To begin, define the set of corners $\cor$ by
\begin{displaymath}
  \cor= 
  \left\{
    (p,\Upsilon_p,\Gamma_{p}): p\in \RR^2
  \right\},
\end{displaymath}

Note that we have a natural  map $\pi\colon \cor\rightarrow \RR^2$ projecting onto the first coordinate and observe that the potential non-uniqueness of infinite geodesics and interfaces leads to some values of $p$ appearing in multiple corners. Further, we note that the set $\cor$ is naturally totally ordered because of the following lemma. 
\begin{lemma}
  \label{lem:main:1}
  Almost surely, for any two points $p_1, p_2$, one of $(\Upsilon_{p_1},\Gamma_{p_1}),(\Upsilon_{p_2},\Gamma_{p_2})$ lies to the left of the other.
\end{lemma}

\begin{proof}
  The statement follows by using the ordering present in the geodesic tree (and the interface portrait) for the case $p_1\neq p_2$. In the case $p_1=p_2=p$, without loss of generality, we need to rule out the case when there are two pairs $(\Upsilon_p,\Gamma_p),(\Upsilon_p',\Gamma_p')$ such that $\Upsilon_p$ is strictly to the left of $\Upsilon_p'$ but $\Gamma_p$ is strictly to the right of $\Gamma_p'$. However, by Proposition \ref{prop:4}, this would in particular imply that $p\in \cT_\uparrow\cap \cI_\downarrow$ but the latter set is almost surely empty. This completes the proof.
\end{proof}
\noindent The following definition uses the above lemma to formally define the ordering between corners.
\begin{definition}
  \label{def:order}
 For points $p_1,p_2\in \RR^2$, if $(\Upsilon_{p_1},\Gamma_{p_1})$ is to the left of $(\Upsilon_{p_2},\Gamma_{p_2})$, then for $\zeta_1=(p_1,\Upsilon_{p_1},\Gamma_{p_1})$ and $\zeta_2=(p_2,\Upsilon_{p_2},\Gamma_{p_2})$, we define $\zeta_1\leq \zeta_2$. As usual, if $\zeta_1\leq \zeta_2$ but $\zeta_1\neq \zeta_2$, we write $\zeta_1<\zeta_2$.%
\end{definition}
\noindent The following lemma will allows us to finally parametrize the Peano curve according to the total Lebesgue area traversed. %
\begin{figure}
  \centering
  \begin{subfigure}[b]{0.25\linewidth}
 \centering
    \includegraphics[width=\linewidth]{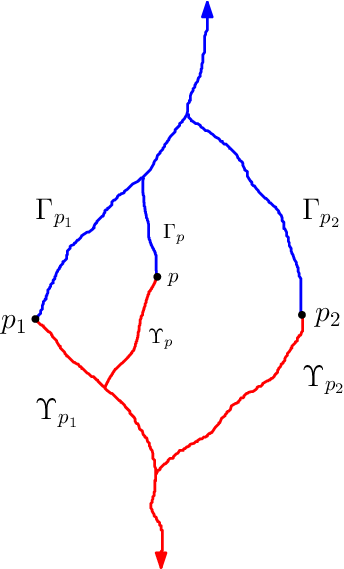}
  \end{subfigure}
  \hspace{3cm}
  \begin{subfigure}[b]{0.2\linewidth}
   \centering
    \includegraphics[width=\linewidth]{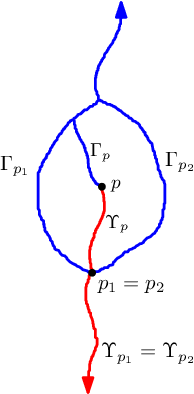}
  \end{subfigure}
  \caption{\textbf{Proof of Lemma \ref{lem:1}}: In the first panel, we have the case when $p_1\neq p_2$ and we see that for points $p$ enclosed between $(\Upsilon_{p_1},\Gamma_{p_1})$ and $(\Upsilon_{p_2},\Gamma_{p_2})$, the corner $\zeta=(p,\Upsilon_p,\Gamma_p)$ satisfies $\zeta_1\leq \zeta\leq \zeta_2$. In the panel to the right, we show an example of the case when $p_1=p_2$ but $\zeta_1\neq \zeta_2$. Again, we see that for points $p$ enclosed between $(\Upsilon_{p_1},\Gamma_{p_1})$ and $(\Upsilon_{p_2},\Gamma_{p_2})$ (or equivalently, between $\Gamma_{p_1},\Gamma_{p_2}$), we have $\zeta_1\leq \zeta\leq \zeta_2$.}
  \label{fig:enc}
\end{figure}
\begin{lemma}
  \label{lem:1}
  For any corners $\zeta_1\leq \zeta_2\in \cor$, define
  \begin{displaymath}
    [\zeta_1,\zeta_2]=
    \left\{
      \zeta\in \cor: \zeta_1\leq \zeta\leq \zeta_2
    \right\}.
  \end{displaymath}
  Almost surely, for any $\zeta_1<\zeta_2$, we have $\leb([\zeta_1,\zeta_2])\coloneqq\leb(\pi([\zeta_1,\zeta_2]))>0$. %
\end{lemma}
\begin{proof}
Let $\zeta_1=(p_1,\Upsilon_{p_1},\Gamma_{p_1})$ and $\zeta_2=(p_2,\Upsilon_{p_2},\Gamma_{p_2})$ be such that $\zeta_1< \zeta_2$. Thus, by definition, $(\Upsilon_{p_1},\Gamma_{p_1})$ lies strictly to the left of $(\Upsilon_{p_2},\Gamma_{p_2})$. Now, the region in between $(\Upsilon_{p_1},\Gamma_{p_1}),(\Upsilon_{p_2},\Gamma_{p_2})$ can be seen to be an open set and thus must have positive Lebesgue area, and it is not difficult to see that for any $p$ lying in this region (see Figure \ref{fig:enc}), and any $\zeta=(p,\Upsilon_p,\Gamma_p)$, we have $\zeta_1\leq \zeta\leq \zeta_2$.  %
\end{proof}
For ease of notation, in case $\zeta_2<\zeta_1$, we define $\leb([\zeta_1,\zeta_2])\coloneqq -\leb([\zeta_2,\zeta_1])$. We now define the augmented Peano curve $\widetilde{\eta}\colon \RR\rightarrow \cor$ by $\eta(0)=(0,\Upsilon_{0},\Gamma_{0})$ and
\begin{equation}
  \label{eq:57}
  \leb([\widetilde{\eta}(0),\widetilde{\eta}(v)])=v
\end{equation}
for all $v\in \RR$. Having defined the above, we define the Peano curve $\eta\colon \RR\rightarrow \RR^2$ by $\eta=\pi \circ \widetilde{\eta}$. Note that Lemma \ref{lem:1} ensures that $\widetilde{\eta},\eta$ are uniquely defined. We now show that as one would expect, $\eta$ is continuous and is space-filling.

\begin{lemma}
  \label{lem:2}
  Almost surely, $\eta$ is a continuous space-filling curve.
\end{lemma}

\begin{proof}
  If $\eta$ were not continuous, then there would exist an $\epsilon>0$ and a corner $\zeta=(p,\Upsilon_p,\Gamma_{p})$ and corners $\zeta_n=(q_n,\Upsilon_{q_n},\Gamma_{q_n})$ such that $|q_n-p|\geq \epsilon$ and $\leb([\zeta_n,\zeta])\rightarrow 0$ as $n\rightarrow \infty$. In fact, it is not difficult to see that the set $\pi([\zeta_n,\zeta])$ must be connected for every $n$, and thus we can further assume that $|q_n-p|=\epsilon$ for all $n$. By the compactness of the circle of radius $\epsilon$ around $p$ along with the compactness property of geodesics, we know that the geodesics $\Gamma_{q_n}$ converge locally uniformly and in the overlap sense (see Proposition \ref{prop:7}) to a geodesic $\Gamma_q$ for some $q$ satisfying $|q-p|= \epsilon$. Also, by the duality between the interface portrait and the geodesic tree along with the same reasoning as above, we also know that the interfaces $\Upsilon_{q_n}$ converge subsequentially locally uniformly and in the overlap sense to an interface $\Upsilon_q$. Thus we have a point $q$ with $|q-p|=\epsilon$ and a  subsequence $\{n_i\}$, for which $\Gamma_{q_{n_i}},\Upsilon_{q_{n_i}}$ converge locally uniformly and in the overlap sense to $\Gamma_{q},\Upsilon_q$ respectively, and we define $\widetilde{\zeta}=(q,\Upsilon_q,\Gamma_{q})$.

  We now note that $\leb([\zeta_{n_i},\zeta])\rightarrow \leb([\widetilde{\zeta},\zeta])$ as $i\rightarrow\infty$. To see this, we first note that $|\leb([\zeta_{n_i},\zeta])-\leb([\widetilde{\zeta},\zeta])|\leq |\leb([\zeta_{n_i},\widetilde{\zeta}])|$. Now, since $\Gamma_{q_{n_i}},\Upsilon_{q_{n_i}}$ converge locally uniformly and in the overlap sense to $\Gamma_q,\Upsilon_q$ respectively, it is not difficult to see that the Lebesgue area enclosed between $(\Upsilon_{q_{n_i}},\Gamma_{q_{n_i}})$ and $(\Upsilon_q,\Gamma_q)$ must converge to zero as well, and this means that $\lim_{i\rightarrow \infty} \leb([\zeta_{n_i},\zeta])= \leb([\widetilde{\zeta},\zeta])$.

  Thus since $\lim_{i\rightarrow \infty}\leb([\zeta_{n_i},\zeta])= 0$, we obtain that $\leb ([\widetilde{\zeta},\zeta])=0$ but this contradicts Lemma \ref{lem:1} since $\widetilde{\zeta}\neq \zeta$. Thus the assumption that $\eta$ is not continuous is false. Finally, since every point $p$ has at least one choice of $\Gamma_p$ and $\Upsilon_p$, the curve $\eta$ has to be space-filling. 
\end{proof}

 The following lemma characterizes the set of points in $\RR^2$ which $\eta$ hits multiple times.
\begin{lemma}
  \label{lem:21}
  The set of points $p\in \RR^2$ for which there exist multiple $v$ with $\eta(v)=p$ is equal to $\cT_\uparrow\cup \cI_\downarrow$ and a.s.\ has Euclidean Hausdorff dimension $4/3$ and Lebesgue measure zero.
\end{lemma}
\begin{proof}
   By Proposition \ref{prop:4}, the set of points $p$ having multiple such choices of $\Gamma_p$ or $\Upsilon_p$ is equal to $\NU_\uparrow^2(\cL)\cup \NU_\downarrow^2(\wcL)=\cI_\downarrow\cup\cT_\uparrow$. The latter a.s.\ has Euclidean Hausdorff dimension $4/3$ as a consequence of \cite[Theorem 5]{Bha23} and thus in particular has measure zero.
\end{proof}
As mentioned earlier, we often write $\eta=(\eta_h,\eta_u)$, where we use $h$ for `horizontal' and $u$ for `upwards'. We now show that our definition of $\eta$ respects the KPZ scaling symmetry inherent in the landscape.
\begin{lemma}
  \label{lem:3}
  The Peano curve $\eta$ satisfies the following scale invariance. For any fixed $\beta>0$, we have 
  \begin{displaymath}
    \eta(\beta v)\stackrel{d}{=}(\beta^{2/5}\eta_h(v),\beta^{3/5}\eta_u(v)),
  \end{displaymath}
  where the distributional equality is as parametrized curves.
\end{lemma}
\begin{proof}
  This is a consequence of the KPZ scaling invariance of the directed landscape \cite[(5) in Lemma 10.2]{DOV18}. Indeed, we know that for any $\alpha>0$, the object
  \begin{equation}
    \label{eq:67}
    \cL_\alpha(x,s;y,t)=\alpha^{-1/3}\cL(\alpha^{2/3} x,\alpha s;\alpha^{2/3}y,\alpha t)
  \end{equation}
  itself has the law of a directed landscape. Also, it is not difficult to see that on transforming $\cL$ as above, the geodesic tree and interface portrait transform analogously as well and thus as a consequence, the trace of the curve $v\mapsto (\alpha^{2/3}\eta_h(v),\alpha \eta_u(v))$ has the same distribution (as a random closed set) as the trace of $\eta$, though the parametrization might be different. To correct for the parametrization, we note that the Jacobian of the linear map $(x,s)\mapsto (\alpha^{2/3}x,\alpha s)$ is simply $\alpha^{5/3}$ and we thus obtain that the curve $v\mapsto (\alpha^{2/3}\eta_h(\alpha^{-5/3}v),\alpha \eta_u(\alpha^{-5/3}v))$ in fact has the same law as $\eta$. We now replace $v$ by $\alpha^{5/3} v$ and $\alpha$ by $\beta^{3/5}$ to obtain the desired result.%
\end{proof}
In the definition of the Peano curve, we demanded that $\eta(0)=0$, but the choice of $0$ here is of course arbitrary. The following lemma makes this formal.
\begin{lemma}
  \label{lem:5}
  For any $p=(x_0,s_0)\in \RR^2$, let $v_p$ be the a.s.\ unique value for which $\eta(v_p)=p$. Then we have
  \begin{displaymath}
    (\eta(v+v_p)-p,\cL(x_0+x,s_0+s;x_0+y,s_0+t))\stackrel{d}{=}(\eta(v),\cL(x,s;y,t)),
  \end{displaymath}
   where the distributional equality is as processes with $v\in \RR, (x,s;y,t)\in \RR^4_\uparrow$.
\end{lemma}

\begin{proof}
 This is an immediate consequence of the translation invariance of the directed landscape.
\end{proof}

In fact, using an abstract measure theoretic argument, we can convert the above symmetry into a translation invariance purely in terms of the volume parametrization.
\begin{lemma}
  \label{lem:6}
  For any fixed $v_0\in \RR$, we have
  \begin{displaymath}
    (\eta(v+v_0)-\eta(v_0),\cL( x+ \eta_h(v_0),s +\eta_u(v_0); y+ \eta_h(v_0),t +\eta_u(v_0) ) \stackrel{d}{=}(\eta(v), \cL(x,s;y,t) ),
  \end{displaymath}
  where the distributional equality is as processes with $v\in \RR, (x,s;y,t)\in \RR^4_\uparrow$. As a consequence, $\Gamma_{\eta(v_0)},\Upsilon_{\eta(v_0)}$ are a.s.\ unique for any fixed $v_0\in \RR$.
\end{lemma}

\begin{proof}
In this proof, we will often use the Lebesgue measure, both on $\RR$ and $\RR^2$, and we locally use $\leb_\RR$ and $\leb_{\RR^2}$ to denote these respectively. Let $\cS$ denote the space of curves from $\RR$ to $\RR^2$ which are $0$ at $0$, with the locally uniform topology and the corresponding Borel $\sigma$-algebra and use $\cH$ to denote the space in which $\cL$ takes it values. We use $\mu_{\eta\times \cL}$ to denote the law of $(\eta,\cL)$, where $\mu_{\eta\times \cL}$ is a measure on $\cS\times \cH$. %
 Consider the measure on $\RR\times \cS\times \cH$ given by the product measure $\leb_\RR\times \mu_{\eta\times \cL}$, and consider the map $f\colon \RR\times \cS\times \cH \rightarrow \RR^2\times \cS\times \cH$ given by
  \begin{equation}
    \label{eq:20}
    f(v,\psi(\cdot),h(\cdot;\cdot))=(\psi(v),\psi(v+\cdot)-\psi(v),h(\psi(v)+\cdot;\psi(v)+\cdot)).
  \end{equation}
We now endow the space $\RR\times \cS\times \cH$ by the measure $\leb_{\RR}\times \mu_{\eta\times \cL}$ and consider the push-forward $f^*(\leb_\RR\times \mu_{\eta\times \cL})$. Since $\eta$ is volume parametrized, we note that the marginal of the above push-forward on $\RR^2$ is just $\leb_{\RR^2}$. Further, as a consequence of Lemma \ref{lem:5}, we obtain that
  \begin{equation}
    \label{eq:21}
    f^*(\leb_\RR\times \mu_{\eta\times \cL})=\leb_{\RR^2}\times \mu_{\eta\times \cL}.
  \end{equation}
 Indeed, by using the disintegration theorem for measures, there are measures $\lambda_p$ on $\cS\times \cH$ defined for a.e.\ $p\in \RR^2$ such that for any measurable set $E\times F\subseteq \RR^2\times (\cS\times \cH)$, we have
  \begin{equation}
    \label{eq:78}
    f^*(\leb_\RR\times \mu_{\eta\times \cL}) (E\times F)=\int_{E}\lambda_p(F) d\leb_{\RR^2}(p).
  \end{equation}
  However, we note that Lemma \ref{lem:5} is equivalent to the statement that the measures $\lambda_p$, which are a-priori only defined for a.e.\ $p$, in fact makes sense for all $p\in \RR^2$ and are equal to $\mu_{\eta\times \cL}$ for all $p$. This establishes \eqref{eq:21}.

  We now note that $f$ is in fact a.e.\ invertible, in the sense that $f^{-1}$ is uniquely defined except on a set of measure zero for $\leb_{\RR^2}\times \mu_{\eta\times \cL}$. Indeed, this is a consequence of the fact that $\eta$ is a.s.\ space-filling and that the set of $p\in \RR^2$ admitting multiple $v$ satisfying $\eta(v)=p$ almost surely has measure zero (Lemma \ref{lem:21}). As a consequence of this and \eqref{eq:21}, we can write $\leb_\RR\times \mu_{\eta\times \cL}=(f^{-1})^*(\leb_{\RR^2} \times \mu_{\eta \times \cL})$.

  We now consider the measure $\mu_{\eta\times \cL}^{v_0}$ on $\cS\times \cH$ which is defined as the law of $(\eta(v_0+\cdot)-\eta(v_0),\cL( \eta(v_0)+\cdot;\eta(v_0)+\cdot))$. By using that $\leb_\RR$ when translated by $v_0$ stays invariant, we again obtain that
  \begin{equation}
    \label{eq:77}
  f^*(\leb_\RR\times \mu_{\eta\times \cL}^{v_0})=\leb_{\RR^2} \times \mu_{\eta\times \cL}.
\end{equation}
Further, if we run the same a.e.\ invertibility argument from above, we obtain that $\leb_\RR\times \mu_{\eta\times \cL}^{v_0}=(f^{-1})^*(\leb_{\RR^2}\times \mu_{\eta\times \cL})$. Thus we have in fact obtained that
\begin{equation}
  \label{eq:76}
  \leb_\RR\times \mu_{\eta\times \cL}^{v_0}=(f^{-1})^*(\leb_{\RR^2}\times \mu_{\eta\times \cL})=\leb_\RR\times \mu_{\eta\times \cL},
\end{equation}
and as a consequence, we obtain that $\mu^{v_0}_{\eta\times \cL}=\mu_{\eta\times \cL}$, and this completes the proof of the distributional equality. To obtain the final statement regarding $\Gamma_{\eta(v_0)},\Upsilon_{\eta(v_0)}$, we just use the distributional equality along with the fact that $\Gamma_0,\Upsilon_0$ are a.s.\ unique.
\end{proof}

\begin{proof}[Proof of Theorem \ref{thm:1}]
  The space filling nature and the continuity of $\eta$ were established in Lemma \ref{lem:2}. Item (1) is a consequence of \eqref{eq:57} while (2) was shown in Lemma \ref{lem:21}. %
  
  For (3), we first note that there a.s.\ does not exist any point $p\in \RR^2$ having both multiple choices of $\Gamma_p$ and multiple choices of $\Upsilon_p$ since we know that $\cT_\uparrow$ and $\cI_\downarrow$ are disjoint. As a consequence, a point $p\in \RR^2$ being hit by $\eta$ at least $k$ times is equivalent to the condition $p\in \NU_\uparrow^k(\cL)\cup \NU_\downarrow^k(\wcL)$. For the case $k=3$, we know that the latter set is non-empty and is countable by \cite[Theorem 5]{Bha23}. Further, for any $k\geq 4$, we can show that the above set is a.s.\ empty. Indeed, to show that $\NU_\uparrow^k(\cL)\cup \NU_\downarrow^k(\wcL)$ is a.s.\ empty, it suffices to show that the set $\NU_\uparrow^k(\cL)$ is a.s.\ empty. We now note that every $p\in \NU_\uparrow^k(\cL)$ is a forward $k$-star in the language of \cite{Dau23+,Bha22} and we know that these do not exist for $k\geq 4$ by \cite[Theorem 1.5]{Dau23+}. Next, we note that (4), (5) were shown in Lemma \ref{lem:6} and Lemma \ref{lem:3} respectively. For the final statement, we observe that $v\mapsto -\eta(-v)$ is the Peano curve for the reflected dual landscape defined by $\wcL^\mathrm{ref}(x,s;y,t)=\wcL(-y,-t; -x,-s)$, where we are using the reflection symmetry of the directed landscape (see (3) in \cite[Lemma 10.2]{DOV18}). This completes the proof.
\end{proof}
Before ending this section, we comment that the Peano curve $\eta$, defined using the trees $\cI_\downarrow,\cT_\uparrow$, can in fact be used to recover the above trees. Indeed, it is not difficult to see that for all $v\in \RR$, $\partial \eta(( -\infty, v])=\Upsilon_p\cup \Gamma_p$, where $\widetilde{\eta}(v)=(p,\Upsilon_p,\Gamma_p)$.
\subsection{Required results from the directed landscape literature}
\label{sec:import}
In this section, we go over some results from the literature which we shall require. The following result from \cite{GZ22} uniformly controls the transversal fluctuation of geodesics in the directed landscape and will be useful to us.
\begin{proposition}[{\cite[Lemma 3.11]{GZ22}}]
  \label{prop:5}
  There exist positive constants $C,c$ and a random variable $S>0$ with the tail bound $\PP(S>M)<Ce^{-cM^{9/4} (\log M)^{-4}}$ for which we have the following. Simultaneously for all $u=(x,s;y,t)\in \RR^4_\uparrow$ and all geodesics $\gamma_{(x,s)}^{(y,t)}$ and $r$ satisfying $(s+t)/2\leq r< t$, 
  \begin{displaymath}
    \left|
      \gamma_{(x,s)}^{(y,t)}(r)-\frac{x(t-r)+y(r-s)}{t-s}
    \right|< S(t-r)^{2/3}\log^3
    \left(
      1+\frac{\|u\|}{t-r}
    \right),
  \end{displaymath}
where $\|u\|$ denotes the usual $L^2$ norm.  A similar bound holds when $s<r<(s+t)/2$ by symmetry.
\end{proposition}
We will also often use the following compactness property of geodesics.
\begin{proposition}[{\cite[Lemma 3.1]{DSV22},\cite[Lemma B.12]{BSS22},\cite[Lemma 13]{Bha23}}]
  \label{prop:7}
  Almost surely, for any $u=(x,s;y,t)$ and any sequence of points $u_n=(x_n,s_n;y_n,t_n)$ converging to u, any sequence of geodesics $\gamma_{(x_n,s_n)}^{(y_n,t_n)}$ admits subsequential limits in the uniform topology, and every such subsequential limit $\gamma_{(x,s)}^{(y,t)}$ is a geodesic from $(x,s)$ to $(y,t)$. Similarly, for any sequence of points $p_n\rightarrow p\in \RR^2$, any sequence $\{\Gamma_{p_n}\}$ is precompact in the locally uniform topology and any subsequential limit is a geodesic $\Gamma_p$.
  Further, the converge above happens in the overlap sense, which means that the set
  \begin{displaymath}
    \{s'\in [s_n,t_n]\cap [s,t]:\gamma_{(x_n,s_n)}^{(y_n,t_n)}(s')=\gamma_{(x,s)}^{(y,t)}(s')\}
  \end{displaymath}
  is an interval whose endpoints converge to $s$ and $t$ along the subsequence from above. A corresponding overlap sense convergence holds for the case of infinite geodesics as well. 
\end{proposition}
For the proofs of Theorems \ref{thm:3} and \ref{thm:5}, we will need to use that the restrictions of the landscape to neighbourhoods of distinct points are roughly independent as long as the neighbourhoods are small. For this, we use the following lemma from \cite{Dau22}.

\begin{proposition}[{\cite[Proposition 2.6]{Dau22}}]
  \label{prop:3}
  For a fixed $k\in \NN$ and pairwise distinct points $z_1,z_2,\dots z_k$, there is a coupling of directed landscapes $\cL,\cL^\mathrm{re}_{z_1},\dots,\cL^\mathrm{re}_{z_k}$ such that $\{\cL^\mathrm{re}_{z_1},\dots,\cL^\mathrm{re}_{z_k}\}$ are mutually independent %
and almost surely, for all small enough $\epsilon$, we have $\cL\lvert_{z_i+[-\epsilon,\epsilon]^2}=\cL^\mathrm{re}_i\lvert_{z_i+[-\epsilon,\epsilon]^2}$ for all $i\in [\![1,k]\!]$.
\end{proposition}
We note that in \cite{Dau22}, the above result is stated for the case when all the points $z_i$ lie on the same horizontal line, but the argument therein does go forth to yield the above result. In order to obtain the dimension of the topological boundaries of segments of $\eta$ (Lemma \ref{lem:12}), we will need to know the Euclidean dimension of infinite geodesics, and we now import this from \cite{Bha23}.
\begin{proposition}[{\cite[Lemma 41]{Bha23}}]
  \label{prop:8}
  Almost surely, simultaneously for all $p\in \RR^2$ and all geodesics $\Gamma_p$, any segment of $\Gamma_p$ has Euclidean Hausdorff dimension $4/3$.
\end{proposition}
We note that the proof of \cite[Lemma 41]{Bha23} only shows the above for a fixed $p$, but this in conjunction with \cite[Lemma 20]{Bha23} immediately implies the full result.

At multiple points in the paper, we will crucially use the FKG inequality. While the availability of the FKG inequality is clear for pre-limiting models, it is a priori not so direct in the continuum. For this reason, we will instead execute the FKG arguments in the pre-limiting model of exponential LPP and will then use the recently established \cite{DV21} convergence of exponential LPP to the directed landscape to obtain the corresponding results in the continuum. We now introduce exponential LPP and state the convergence result.

Throughout this discussion, will use $\mathbf{v},\mathbf{w}$ to denote the vectors $(1,1)$ and $(1,-1)$ respectively. Let $X^n=X^n_{(i,j)}$ denote a family of i.i.d.\ $\mathrm{exp}(1)$ random variables indexed by $(i,j)\in \ZZ^2$, whose coupling across different $n$ will be specified later in Proposition \ref{prop:6}. %
The variables $X^n_{(i,j)}$ above will play the role of the weights of the pre-limiting exponential LPP model which converges to the directed landscape. For any two points $p=(x_1,y_1)\leq q=(x_2,y_2)\in \ZZ^2$ in the sense that $x_1\leq x_2$ and $y_1\leq y_2$, we define the passage time
\begin{equation}
  \label{eq:58}
  T^n(p,q)=\max_{\pi: p \rightarrow q}\ell(\pi),
\end{equation}
where the maximum above is over all up-right lattice paths $\pi$ going from $p$ to $q$ and the weight of such a path $\pi$ is simply defined as $\ell(\pi)=\sum_{r\in \pi\setminus \{p\}}X^n_r$ \footnote{Here, we do not add the weight of the first vertex $X^n_p$ for notational simplicity later on. Indeed, this ensures that $T^n(p,q), T^n(q,r)$ are independent for every $p<q<r$. We note that this convention is different from the one usually used, which is $\ell(\pi)=\sum_{r\in \pi} X^n_r$. However, in practice, both these conventions are equivalent and in particular, all the upcoming results in this section hold for both these conventions.}. The a.s.\ unique path achieving the above maximum is called the geodesic from $p$ to $q$ and is denoted by $\gamma_{p,q}\{X^n\}$. In fact, each point $p\in \ZZ^2$ also has \cite{FMP09,FP05}  an a.s.\ unique $\mathbf{v}$-directed semi-infinite geodesic $\Gamma^{n,\dis}_p$ which is an up-right path starting at $p$ whose every finite segment is a geodesic, and further, any two $\Gamma^{n,\dis}_{p_1},\Gamma^{n,\dis}_{p_2}$ with $p_1\neq p_2$ a.s. coalesce \cite{FP05}. We will often think of $\Gamma^{n,\dis}_p$ as a discrete path such that $m\mathbf{v}+\Gamma_{p}^{n,\dis}(m)\mathbf{w}\in \Gamma^{n,\dis}_p$ for all $m\in (1/2)\ZZ$ satisfying $m\geq (x_1+y_1)/2$, where we recall that $p=(x_1,y_1)$. In this notation, the geodesic $\Gamma_p^{n,\dis}$ being $\mathbf{v}$-directed just means that $\Gamma ^{n,\dis}_p(m)/m\rightarrow 0$ a.s.\ as $m\rightarrow \infty$.

To introduce the prelimiting directed landscapes $\cL^n$, we need some additional notation. For a given $p\in \ZZ^2$, we define the set $\boxx(p)$ by
\begin{equation}
  \label{eq:8}
  \boxx(p)=p+\{s\mathbf{v}+x\mathbf{w}: s\in (-1/4,1/4], x\in (-1/2,1/2]\},
\end{equation}
and note that $\RR^2$ is equal to the disjoint union $\bigsqcup_{p\in \ZZ^2}\boxx(p)$. For $q\in \RR^2$, we use the notation $\mathfrak{r}(q)$ to denote the unique $p\in \ZZ^2$ for which $q\in \boxx(p)$. We now define the pre-limiting directed landscape $\cL^n$ for $n\in \NN$ by
\begin{equation}
  \label{eq:59}
  \cL^n(x,s;y,t)=2^{-4/3}n^{-1/3}\left(T^n(\mathfrak{r}(sn\mathbf{v}+ 2^{2/3}xn^{2/3} \mathbf{w}), \mathfrak{r}( tn\mathbf{v}+ 2^{2/3}yn^{2/3} \mathbf{w}))-4(t-s)n\right),
\end{equation}
where $(x,s;y,t)\in \RR^4_\uparrow$. Similarly, we define the rescaled upward semi-infinite geodesics $\Gamma^n_p$ by
\begin{equation}
  \label{eq:60}
  \Gamma^n_p(t)=2^{-2/3}n^{-2/3}\Gamma^{n,\dis}_{\mathfrak{r}( 2^{2/3}xn^{2/3}\mathbf{w}+sn\mathbf{v})} (nt)
\end{equation}
for $nt\in (1/2)\ZZ$ and interpolate linearly in between.
We can now finally state the convergence result of prelimiting exponential LPP to the directed landscape.

\begin{proposition}[{\cite[Theorem 1.7, Remark 1.10]{DV21},\cite[Proposition 35]{Bha23}}]
  \label{prop:6}
  There is a coupling between the $X^n$ such that $\cL^n\rightarrow \cL$ almost surely with respect to the locally uniform topology on $\RR^4_\uparrow$. Further, in this coupling, almost surely, for any points $p_n\rightarrow p\in \RR^2$, the geodesics $\Gamma^n_{p_n}$ are precompact in the locally uniform topology, and any subsequential limit is an upward semi-infinite geodesic $\Gamma_p$.
\end{proposition}
An important part of the barrier arguments used to prove Propositions \ref{lem:9} and \ref{t:va} is the usage of induced passage times, which we now define. Given points $p\leq q\in \ZZ^2$ and a set $S\subseteq \RR^2$ containing $p,q$, we define the induced passage time
\begin{equation}
  \label{eq:63}
  T^n(p,q\vert S)= \max_{\pi: p\rightarrow q, \pi\subseteq S}\ell(\pi).
\end{equation}
Now for a rectangle $R=[x_1,y_1]\times [s_1,t_1]$, we define a corresponding rectangle $R^\dagger$ by
\begin{equation}
  \label{eq:65}
R^\dagger= \{\mathfrak{r}( s'n\mathbf{v}+ 2^{2/3} x' n^{2/3} \mathbf{w}): (x',s')\in R\}.  
\end{equation}
Finally, for $(x,s;y,t)\in \RR^4_\uparrow$ and a rectangle $R=[x_1,y_1]\times [s_1,t_1]$ containing  both $(x,s)$ and $(y,t)$, we define the rescaled induced passage time $\cL^n(x,s;y,t\vert R)$ by
\begin{equation}
  \label{eq:64}
  \cL^n(x,s;y,t\vert R)=2^{-4/3}n^{-1/3}\left(T^n(\mathfrak{r}( sn\mathbf{v}+ 2^{2/3}xn^{2/3} \mathbf{w}), \mathfrak{r}( tn\mathbf{v}+ 2^{2/3}yn^{2/3} \mathbf{w})\vert R^\dagger)-4(t-s)n\right).
\end{equation}
Finally, before concluding this section, we quickly introduce Busemann functions, a notion from geometry \cite{Bus12}, originally introduced to first passage percolation in \cite{New95, Hof05}. With $z\in \ZZ^2$ being the point at which the geodesics $\Gamma^{n,\dis}_p$ and $\Gamma^{n,\dis}_q$ first meet, we define the discrete Busemann function
\begin{equation}
  \label{eq:62}
  \cB^{n,\dis}(p,q)=T^n(p,z)-T^n(q,z).
\end{equation}
Similarly, for $p,q\in \RR^2$, we define the rescaled Busemann function $\cB^n$ by
\begin{equation}
  \label{eq:66}
  \cB^n(p,q)= \cL^n(p;z)-\cL^n(q;z).
\end{equation}
\subsubsection{\textbf{Notation relevant for Sections \ref{sec:basic},  \ref{sec:lpp}}}
\label{s:lppnotation}
In Sections \ref{sec:basic}, \ref{sec:lpp}, we will work with LPP without being concerned with the coupling from Proposition \ref{prop:6}. In this case, we will use slightly modified notation from the above and will omit the $n$. That is, we will have a field of i.i.d.\ variables $X_{i,j}$, passage times $T_{p,q}=T(p,q)$, induced passage times $T_{p,q}^S=T(p,q\vert S)$, geodesics $\gamma^\dis_{p,q}$, %
$\mathbf{v}$-directed semi-infinite geodesics $\Gamma^\dis_{p}$, and Busemann functions $\cB^\dis$. We will simply define $\Gamma=\Gamma^\dis_{0}$ for convenience. Often, for $q=(q_1,q_2)\in \ZZ^2$, we shall use the space-time coordinates defined by $t(q)=\frac{q_1+q_2}{2}$ and $x(q)=\frac{q_1-q_2}{2}$.

\subsection{Some basic estimates in the LPP setting}
\label{sec:basic}
We shall record some widely used estimates about the passage times between vertices in a parallelogram here; these will be crucial for the arguments in Section \ref{sec:lpp}. These were first proved for Poissonian LPP in \cite{BSS14}, and we quote the results in the exponential LPP setting from \cite{BGZ19}.

Let $U$ be a parallelogram whose one pair of parallel sides is along the lines $\{q=(q_1,q_2)\in \ZZ^2: q_1+q_2=0\}$ and $\{q=(q_1,q_2)\in \ZZ^2: q_1+q_2=2n\}$, each of length $2n^{2/3}$ with midpoints $(mn^{2/3},-mn^{2/3})$ and $(n,n)$ respectively. Using the space-time notation from Section \ref{s:lppnotation}. Let $U_1$ and $U_2$ denote the parallelograms obtained by taking the intersection of $U$ with the strips $\{t\in [0,n/3]\}$ and $\{t\in [2n/3,n]\}$ respectively. We have the following result. 

\begin{proposition}[{\cite[Theorem 4.2]{BGZ19}}]
    \label{p:para}
    For any fixed $\Psi\in (0,1)$, there exist constants $C,c$ depending only on $\Psi$ such that for all $n\ge 1$ and $|m|<\Psi n^{1/3}$ we have the following for all $s>0$,
    \begin{enumerate}
        \item[(i)] $\PP(\inf_{p\in U_1,q\in U_2} T_{p,q}-\EE T_{p,q}\le-sn^{1/3})\le Ce^{-cs^3}$, 
        \item[(ii)] $\PP(\sup_{p\in U_1,q\in U_2} T_{p,q}-\EE T_{p,q}\ge sn^{1/3})\le Ce^{-c\min\{s^{3/2},sn^{1/3}\}}$,
        \item[(iii)] $\PP(\inf_{p\in U_1,q\in U_2} T^{U}_{p,q}-\EE T_{p,q}\le-sn^{1/3})\le Ce^{-cs}$.
    \end{enumerate}
\end{proposition}

Using Proposition \ref{p:para}, one can control the transversal fluctuation of finite geodesics. With $p_{m}$ denoting the point such that $t(p_{m})=n$ and $x(p_m)=mn^{2/3}$, we have the following result. %

\begin{proposition}[{\cite[Proposition C.9]{BGZ19}}]
    \label{l:tffin}
    For any fixed $\Psi\in (0,1)$, there exist constants $C,c$ depending only on $\Psi$ such that for all $n\ge 1$ and $|m|<\Psi n^{1/3}$, we have for all $h>0$, 
    $$\PP(\sup_{s\in \llbracket 0,2n \rrbracket/2}|\gamma^\dis_{0,p_{m}}(s)|\ge hn^{2/3})\le Ce^{-ch^{3}}.$$
\end{proposition}

We can also similarly bound transversal fluctuations for semi-infinite geodesics. %
See e.g.\ \cite[Theorem 3, Remark 1.3, Proposition 6.2]{BSS19}) for the following result.\footnote{These results only upper bound $\Gamma(n)$, but combining this with Proposition \ref{l:tffin} gives Proposition \ref{p:tf}.} 

\begin{proposition}
    \label{p:tf}
    There exist constants $C,c>0$ such that for $n\ge 1$ and $h>0$, we have 
    $$\PP(\sup_{s\in \llbracket 0,2n \rrbracket/2}|\Gamma(s)|\ge hn^{2/3})\le Ce^{-ch^{3}}.$$
\end{proposition}

In fact, it is easy to see that the following stronger version also holds. Let $I_n$ be the line segment in the diagonal line given by $\{x=0\}\cap \{t\in [-n,0]\}$.

\begin{proposition}
    \label{p:tfstrong}
    There exist constants $C,c>0$ such that for $n\ge 1$ and $h>0$, we have 
    $$\PP(\sup_{p\in I_{n}, s\in \llbracket 0,2n \rrbracket/2}|\Gamma^{\rm{dis}}_{p}(s)|\ge hn^{2/3})\le Ce^{-ch^{3}}.$$
\end{proposition}

Finally, we need the result that barrier events occur with positive probability. Let $U_{\Delta}$ denote the rectangle whose one pair of parallel sides are along the lines $\{t=0\}$ and $\{t=n\}$ with lengths $2\Delta n^{2/3}$ and midpoints  on the line $\{x=0\}$. %

\begin{proposition}[{\cite[Lemma 4.11]{BGZ19}}]
    \label{p:barrier}
    For every $\Delta, L>0$ there exists $\beta=\beta(\Delta,L)>0$ such that for all $n$ sufficiently large, we have 
    $$\PP\left(\sup_{p,q\in U_{\Delta}: t(q)-t(p)\ge L^{-1}n}  T^{U_{\Delta}}_{p,q}-\EE T_{p,q} \le -Ln^{1/3}\right)\ge \beta.$$
\end{proposition}

 \section{Tail Estimates and H\"{o}lder continuity}
 \label{sec:holder}
The purpose of this section is to prove Theorem \ref{thm:2}, the H\"{o}lder continuity result for $\eta$, assuming the volume accumulation result Proposition \ref{t:va}. For this, we shall need tail estimates for the co-ordinate functions $\eta_{u}$ and $\eta_{h}$.

\begin{proposition}
    \label{p:utail}
    There exist constants $c,t_0>0$ such that for all $w>0,t>t_0$,
    $$\PP(\max_{v\in [-w,w]}|\eta_{u}(v)|\ge w^{3/5}t)\le \exp(-ct^{5/11}).$$
\end{proposition}

\begin{proposition}
    \label{p:htail}
    There exist constants $c,y_0>0$ such that for all $w>0,y>y_0$, 
    $$\PP(\max_{v\in [-w,w]}|\eta_{h}(v)|\ge w^{2/5}y)\le \exp(-cy^{30/77}).$$
\end{proposition}
Note that the exponents $5/11, 30/77$ are not expected to be optimal, since the exponent in the key ingredient for the proofs of these results, Proposition \ref{t:va}, is already sub-optimal, as mentioned before. Before moving on to the proofs, we note that the above imply the following estimate on the probability that $\eta$ quickly exits an intrinsic metric ball, 
\begin{proposition}
  \label{prop:tails}
  There exist constants $c,\ell_0$ such that for all $w>0,\ell>\ell_0$,
  \begin{displaymath}
    \PP
    \left(
      \eta([-w,w])\not\subseteq B^\mathrm{in}_{w^{1/5}\ell}(0)
    \right)\leq \exp(-c\ell^{60/77}).
  \end{displaymath}
\end{proposition}
We spend the rest of the section proving Propositions \ref{p:utail}, \ref{p:htail}. We begin by defining the important notion of the volume accumulated to the right.

\begin{definition}
\label{d:va}
Let $p=(x,s)\in \RR^2$ be fixed and let $\Gamma_p$ be the a.s.\ unique upward semi-infinite geodesic emanating from $p$. %
Now define $\wcV_R(\Gamma_p;\ell)$ as the set of points $q=(y,t)$ such that $t\in [s,s+\ell]$ and $y\ge \Gamma_{p}(t)$ (i.e., $q$ lies to the right of $\Gamma_{p}$) and there exists an upward semi-infinite geodesic $\Gamma_q$ such that $\Gamma_q$ and $\Gamma_{p}$ coalesce below $t=s+\ell$. Notice that $\wcV_R(\Gamma_p;\ell)$ is a closed set and hence measurable. We define $\widetilde{V}_{R}(\Gamma_{p}; \ell)=\leb(\wcV_R(\Gamma_p;\ell))$ and refer to it as the volume accumulated to the right by $\Gamma_{p}$ in time $\ell$. %
\end{definition}
We note that the above notation agrees with the one used in Proposition \ref{t:va}. The importance of Definition \ref{d:va} in the proof of Theorem \ref{thm:2} lies in the following deterministic lemma. 

\begin{lemma}
    \label{l:vv} Let $w\in \RR$ be fixed. Consider the a.s.\ unique (by Lemma \ref{lem:6}) upwards infinite geodesic $\Gamma_{\eta(w)}$ emanating from $\eta(w)=(x_w,t_w)$. Then for any $v>w$ with the property that $\eta(v)=(\Gamma_{\eta(w)}(t),t)$ for some $t>t_w$, we have, $v> w+ \widetilde{V}_{R}(\Gamma_{\eta(w)}; t-t_w)$.      
\end{lemma}

This lemma says that starting from $\eta(w)$, when the Peano curve $\eta$ reaches some point on the geodesic $\Gamma_{\eta(w)}$, the volume parametrization must have increased by at least the volume accumulated by $\Gamma_{\eta(w)}$ to the right, up to that height. 

\begin{proof}
   Let $\zeta_1=(\eta(w),\Upsilon_{\eta(w)},\Gamma_{\eta(w)})$ and let $\zeta_2=\widetilde{\eta}(v)=(\eta(v),\Upsilon_{\eta(v)},\Gamma_{\eta(v)})$, where $\widetilde{\eta}$ denotes the augmented Peano curve from \eqref{eq:57}, and the choices $\Upsilon_{\eta(v)},\Gamma_{\eta(v)}$ are made so that the above equality holds. By planarity, we note that any $p\in \wcV_R(\Gamma_{\eta(w)};t-t_w)$, we must have a corner $\zeta_p=(p,\Upsilon_p,\Gamma_p)$ satisfying that $\zeta_1\leq \zeta\leq \zeta_2$, where $\Gamma_p$ coalesces with $\Gamma_{\eta(w)}$ before the time $t$. Thus $\leb ([\zeta_1,\zeta_2])\geq \leb( \wcV_R(\Gamma_{\eta(w)};t-t_w))= \widetilde{V}_{R}(\Gamma_{\eta(w)}; t-t_w)$, and this completes the proof.
\end{proof}

We now provide the proofs of Propositions \ref{p:utail}, \ref{p:htail} assuming Proposition \ref{t:va}.

\begin{proof}[Proof of Proposition \ref{p:utail}]
  First note that since $\eta_u(\beta\cdot)\stackrel{d}{=}\beta^{3/5}\eta_u(\cdot)$ (Theorem \ref{thm:1} (5)), it suffices to obtain the needed estimate after fixing $w=1$. Combining the above with the time reversal symmetry (Theorem \ref{thm:1} (6)) of $\eta$, we need only show that there exist constants $c,t_0>0$ such that for all $t>t_0$, we have
  \begin{equation}
    \label{eq:74}
    \PP(\max_{v\in [0,1]} \eta_u(v)\ge t)\le \exp(-ct^{5/11}).
  \end{equation}
  The aim of the rest of the proof is to show \eqref{eq:74}. We consider the point $p=(\Gamma_0(t),t)$, and note that there are two choices of the interface $\Upsilon_p$ by Proposition \ref{prop:4}, and we use $\Upsilon_p$ to denote the rightmost one. Realizing that $\Gamma_p=\Gamma_0\lvert_{[t,\infty)}$, we consider the corner $\zeta_p=(p,\Upsilon_p,\Gamma_p)$ and note by planarity that $\zeta_p\leq \zeta$ for any corner $\zeta=( (x,s),\Upsilon_{(x,s)},\Gamma_{(x,s)})$ satisfying $\zeta\geq (0,\Upsilon_0,\Gamma_0)$ and $s\geq t$. %
 As a consequence, on the event $\{\max_{v\in [0,1]}\eta_u(v)\geq t\}$, there must exist $v_0\in (0,1)$ satisfying $\eta(v_0)=(\Gamma_0(t),t)$. By Lemma \ref{l:vv}, it follows that $\widetilde{V}_{R}(\Gamma_0;t)\le 1$ on the above event. Now, by the scaling of the directed landscape (Lemma \ref{lem:3}), it follows that $\widetilde{V}_{R}(\Gamma;t)\stackrel{d}{=} t^{5/3}\widetilde{V}_{R}(\Gamma;1)$ and hence we have 
$$\PP(\max_{v\in[0,1]}\eta_{u}(v)\ge t)\le \PP(\widetilde{V}_{R}(\Gamma;1)\le t^{-5/3})\le \exp(-ct^{5/11})$$
by Proposition \ref{t:va}. This completes the proof. 
\end{proof}

\begin{figure}[htbp!]
     \centering
         \centering
         \includegraphics[width=0.7\textwidth]{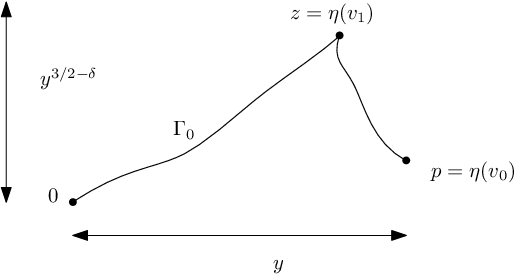}
     \caption{\textbf{Proof of Proposition \ref{p:htail}, case 1}: In this case, there exist $0\le v_1<v_0\le 1$ such that $x(\eta(v_0))=y$, $z=\eta(v_1)\in \Gamma_0$ and $t(z) \le y^{3/2-\delta}, x(z)\ge y/2$. This implies that the geodesic $\Gamma_0$ has atypically large transversal fluctuation at or below height $y^{3/2}-\delta$ and hence this case is unlikely.} 
     \label{fig:volacc}
     \end{figure}

\begin{proof}[Proof of Proposition \ref{p:htail}]
  Again, by using the KPZ scaling  and the time reversal symmetry of $\eta$, we need only show that there exist constants $c,y_0>0$ such that for all $y>y_0$,
  \begin{equation}
    \label{eq:75}
    \PP(\max_{v\in [0,1]}|\eta_{h}(v)|\ge y)\le \exp(-cy^{30/77}).
  \end{equation}
  In the rest of the proof, we prove \eqref{eq:75}. Let $\delta>0$ be fixed, to be chosen appropriately later. We write 
    \begin{equation}
        \label{e:split}
        \PP(\max_{v\in [0,1]} \eta_h(v)\ge y)\le \PP(\max_{v\in [0,1]} \eta_h(v)\ge y, \max_{v\in[0,1]}|\eta_{u}(v)|\le y^{3/2-\delta})+\PP(\max_{v\in[0,1]}|\eta_{u}(v)|\ge y^{3/2-\delta}).
    \end{equation}
   The second term is bounded above by $\exp(-cy^{15/22-3\delta/2})$ by Proposition \ref{p:utail} and thus we need only bound the first term. 
   For this, let $v_0\in [0,1]$ be such that $\eta_h(v_0)= y$ (such a $v_0$ must exist on the first event above by continuity). With $\widetilde{\eta}$ denoting the augmented Peano curve from \eqref{eq:57}, we define $(p,\Upsilon_p,\Gamma_p)\coloneqq \widetilde{\eta}(v_0)$ and let $z$ be denote the coalescence point corresponding to the geodesics $\Gamma_p, \Gamma_0$. By planarity, there must exist $\Upsilon_z,\Gamma_z$ such that $(0,\Upsilon_0,\Gamma_0)\leq (z,\Upsilon_z,\Gamma_z)\leq (p,\Upsilon_p,\Gamma_p)$, and thus there exists a $v_1\in [0,v_0]$ such that $z=\eta(v_1)$. %
 Writing $z=(\Gamma_0(\ell), \ell)$ we consider two cases. Note also that on $\{\max_{v\in[0,1]}|\eta_{u}(v)|\le y^{3/2-\delta})\}$ we also have $\ell \le y^{3/2-\delta}$.

   \textbf{Case 1:} $\Gamma_0(\ell)\ge \frac{y}{2}$. Since $\ell \le y^{3/2-\delta}$, the probability of this case is upper bounded by that (see Figure \ref{fig:volacc}) of the event $\{(y^{3/2-\delta})^{-2/3}\sup_{t\in [0,y^{3/2-\delta}]} |\Gamma_0(t)|\ge \frac{1}{2}y^{2\delta/3}\}$, which by the scaling of the directed landscape (Lemma \ref{lem:3})  has the same probability as 
   the event $\{\sup_{t\in [0,1]} |\Gamma_0(t)|\ge \frac{1}{2}y^{2\delta/3}\}$. By Proposition \ref{p:tf} (which gives the bound for exponential LPP) and Proposition \ref{prop:6} (which gives the convergence of geodesics in exponential LPP to the geodesics in directed landscape), the above probability is upper bounded by $\exp(-cy^{2\delta})$. 

   \textbf{Case 2:} $\Gamma_0(\ell)\le \frac{y}{2}$. In this case, the geodesic $\Gamma_p$ has large transversal fluctuations, and we take a union bound over all possible locations of $p$. Specifically, we divide the line segment  $\{(y,t):|t|\le y^{3/2-\delta}\}$ into $2y^{3/2-\delta}$ many line segments of length $1$ each (see Figure \ref{fig:volacc2}). Call these intervals $I_{i}$. Let $\mathcal{A}_{i}$ denote the event that there exists a point $r\in I_{i}$ and a semi-infinite geodesic $\Gamma_{r}$ started from $r$ such that $\Gamma_{r}(\ell)\le y/2$ for some $\ell \le y^{3/2-\delta}$. It follows from {Proposition \ref{p:tfstrong} and Proposition \ref{prop:6}} that for each $i$, $\PP(\mathcal{A}_{i})\le \exp(-cy^{2\delta})$, so by a union bound the probability of this case is upper bounded by $2y^{3/2-\delta}\exp(-cy^{2\delta})$. 

 \begin{figure}[t]
     \centering
         \centering
         \includegraphics[width=0.6\textwidth]{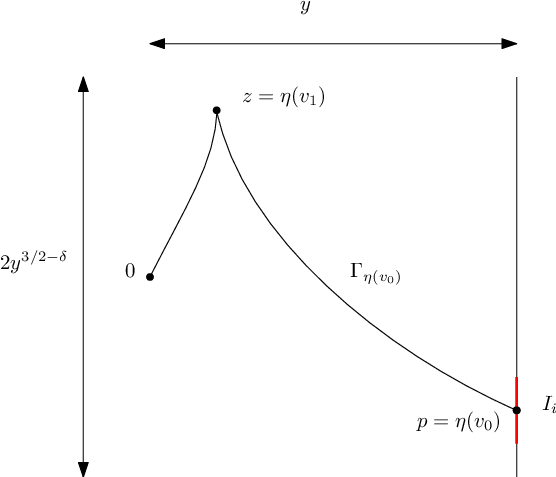}
     \caption{\textbf{Proof of Proposition \ref{p:htail}, case 2}: In this case, there exist $0\le v_1<v_0\le 1$ such that $x(\eta(v_0))=y$, $z=\eta(v_1)\in \Gamma_0$ and $t(z) \le y^{3/2-\delta}, x(z)\le y/2$. 
     This implies that the geodesic $\Gamma_{\eta(v_0)}$ has atypically large 
     transversal fluctuation at or below height $y^{3/2}-\delta$. To handle this we divide the line segment $\{y\}\times [-y^{3/2-\delta},y^{3/2-\delta}]$ into segments of length $1$ each and call these segments $I_i$. For each $I_{i}$, the event that there exists a point $p\in I_{i}$ such that $\Gamma_{p}$ has atypically large transversal fluctuations is unlikely and we get the result by a union bound over $i$.} 
     \label{fig:volacc2}
     \end{figure}

   Let us now choose $\delta=15/77$ such that $15/22-3\delta/2=2\delta=30/77$. For this choice of $\delta$, by taking a union bound over the two cases above and using \eqref{e:split}, we get the desired result. 
\end{proof}

We can now complete the proof of Theorem \ref{thm:2}. Notice that by the definition of intrinsic metric $d_\mathrm{in}$ from \eqref{eq:54}, it suffices to prove the following two results.

\begin{proposition}
    \label{t:holderu}
    For any $\delta>0$ and $M>0$. $\eta_{u}$ is almost surely $3/5-\delta$ H\"{o}lder continuous on $[-M,M]$. 
\end{proposition}

\begin{proposition}
    \label{t:holderh}
    For any $\delta>0$ and $M>0$. $\eta_{h}$ is almost surely $2/5-\delta$ H\"{o}lder continuous on $[-M,M]$. 
\end{proposition}

Using Propositions \ref{p:utail}, \ref{p:htail} together with the fact that for every fixed $v$ and $\epsilon$, the law of $\eta(v+\epsilon)-\eta(v)$ does not depend on $v$ ((4) in Theorem \ref{thm:1}) the proofs of the above theorems are essentially the same as the standard proof that Brownian motion is $1/2-$ H\"{o}lder continuous on compact sets almost surely. We shall give the details for Proposition \ref{t:holderu}. 

\begin{proof}[Proof of Proposition \ref{t:holderu}]
Note first that by the scaling from Theorem \ref{thm:1} (5), it suffices to prove the result for $M=1$. For $N\in \NN$, let $H_{N}:=\{2^{-N}i: i=-2^{N}, -2^{N}+1, \ldots, -1,0,1,\ldots, 2^{N}\}$.
and let $\mathbb{Q}_2=\cup_{N} H_{N}$ denote the set of all dyadic rationals in $[-1,1]$. We shall show that $\eta_{u}$ is almost surely $3/5-\delta$ H\"{o}lder continuous on $\mathbb{Q}_2$. Clearly this suffices since dyadic rationals are dense in $[-1,1]$ and by continuity we have that if $v_n\to v$ and $w_{n}\to w$ in $[-1,1]$ with $v\ne w$ then $\frac{\eta_{u}(v_n)-\eta_{u}(w_n)}{|v_{n}-w_n|^{3/5-\delta}}\to \frac{\eta_u(v)-\eta_{u}(w)}{|v-w|^{3/5-\delta}}$. 

To prove H\"{o}lder continuity on $\mathbb{Q}_2$, we let $G_{N}$ be the event that 
$$ {|\eta_{u}(2^{-N}i)-\eta_{u}(2^{-N}(i-1))|}\le 2^{-(3/5-\delta)N}$$
for all $\in H_{N}$. It follows from Proposition \ref{p:utail} that 
$$\PP(G^{c}_{N})\le 2^{N}\exp(-c2^{5\delta N/11}).$$
By the Borel-Cantelli Lemma, we obtain that $\PP(\liminf_{N\rightarrow \infty} G_N)=1$ and thus there exists a random $N_*$ such that $G_N$ occurs for all $N\geq N^*$. %
It follows from a standard argument that there exists an absolute constant $C$ (depending only on $\delta$) such that we almost surely have %
$$ |\eta_{u}(v)-\eta_{u}(w)| \le C|v-w|^{3/5-\delta}$$
for all $v,w\in \mathbb{Q}_2$ with $|v-w|< 2^{-N_*}$. It further follows by the triangle inequality that almost surely, there exists a $C'$ (random but finite) such that 
$$ |\eta_{u}(v)-\eta_{u}(w)| \le C'|v-w|^{3/5-\delta}$$
for all $v,w\in \mathbb{Q}_2$, as required. This completes the proof. 
\end{proof}

The proof of Proposition \ref{t:holderh} is identical to that of Proposition \ref{t:holderu} except that we use Proposition \ref{p:htail} instead of Proposition \ref{p:utail}; we omit the details. This completes the proof of Theorem \ref{thm:2} assuming Proposition \ref{t:va}, and the goal of the next section is to prove the latter.

\section{One-sided volume accumulation: The proof of Proposition \ref{t:va}}
\label{sec:lpp}

In this section, we provide the proof of Proposition \ref{t:va}. Since the arguments here are naturally suited to the discrete setting, we will first work with LPP and prove an analogous volume accumulation result in LPP (Proposition \ref{t:vaexp}), and then finally use the convergence (Proposition \ref{prop:6}) of exponential LPP to the directed landscape to transfer the result. We recall the discussion on LPP from Section \ref{sec:import}, particularly the notation introduced in Section \ref{s:lppnotation}.
Importantly, \textbf{only in this section}, we shall use space-time coordinates by default, that is, whenever we write $q=(x,t)$, we mean that $t(q)=t$ and $x(q)=x$. Let $\mathbb{L}_{n}$ denote the anti-diagonal line $\{t=n\}$, and for any (up/right) path $\gamma$, define $\gamma(t)=x(q)$ where $q$ is the point (necessarily unique if there is one) where $\gamma$ intersects $\mathbb{L}_{t}$.

Let $\mathcal{V}_R(\Gamma;n)$ denote the set of all $q$ such that $t(q)\ge 0$ and $x(q)\ge \Gamma(t(q))$ such that 
$\Gamma \cap \Gamma^\dis_{q}\cap \mathbb{L}_{n}\neq \emptyset$. We define the \textbf{volume accumulated by $\Gamma$ to the right until time $n$}, denoted $V_{R}(\Gamma;n)$ by the size of $\mathcal{V}_R(\Gamma;n)$, i.e., $V_{R}(\Gamma;n)=\#\mathcal{V}_R(\Gamma;n)$. The purpose of this section is to prove the following discrete analogue of Proposition \ref{t:va}.
\begin{proposition}
    \label{t:vaexp}
    There exist constants $C,c,K'>0$ such that for all $\delta>0,n\in \NN$ satisfying $\delta^{12/11} n \geq K'$, we have $\PP(V_{R}(\Gamma;n)\le \delta n^{5/3})\le C\exp(-c\delta^{-3/11})$.
\end{proposition}

The above result, which we shall often refer to as the `volume accumulation result' is the key result of this section. The argument for the proof of the above is reminiscent of an argument from \cite{BSS14}. The idea is to show that at different locations along the geodesic $\Gamma$, there is a positive probability that the geodesics started from close by points will coalesce with $\Gamma$; and these events are approximately independent. We now move on to the proof of the above. Before moving on, we note that the same proof, with some modifications, can be used to prove slight variants of Proposition \ref{t:vaexp} in slightly different settings e.g.\ for point-to-point geodesics instead of semi-infinite geodesics. At the very end of this section, we shall give a short discussion of this point (see Propositions \ref{prop:9}, \ref{prop:1}).

We now prepare for the proof of Proposition \ref{t:vaexp}. To this end, we shall define a number of parameters. Our first task is to prove certain regularity properties of $\Gamma$ and the environment around it.

\subsection{The regularity of the geodesic}
We shall prove two results in this section: the first one showing that transversal fluctuations are on-scale at most locations along the geodesic $\Gamma$, and the second showing that typical events happen at most locations along the geodesic.

\subsubsection{Transversal fluctuations along the geodesic}
We shall work with $\epsilon_1>0$ which should be thought of as small; let us also assume without loss of generality that $\epsilon_1^{-1}$ is an integer that divides $n$. For $i\in [\![0,\epsilon_{1}^{-1}]\!]$, let $\tau_i=\lfloor(\epsilon_1n)^{-2/3}(\Gamma(i\epsilon_1n))\rfloor$. Then we have the following lemma. 

\begin{proposition}
    \label{l:geodfluc}
    There exist positive absolute constants $T,c,C$ such that for all $\epsilon_1,n$ satisfying $\epsilon_1^2n\geq 1$, we have %
    $$\PP\left(\sum_{i=0}^{\epsilon_1^{-1}-1}|\tau_{i+1}-\tau_{i}|\ge T\epsilon_1^{-1}\right)\le C\exp(-c\epsilon_1^{-1}).$$
\end{proposition}

Note that this proposition says that the average on-scale transversal displacement of the geodesic at scale $\epsilon_1$ is constant (on-scale) with high probability, which in particular implies, by Markov's inequality, that with high probability, most locations along the geodesic show on-scale transversal fluctuations. The proof of this proposition will be by a coarse version of what we call the \emph{percolation argument}, the argument used to show the existence of a sub-critical phase in Bernoulli percolation. This technique will be important for us  later in this section as well. This proof essentially follows the arguments in \cite{BSS14}, where a similar result was shown for point-to-point geodesics in Poissonian LPP. We need a couple of easy lemmas. 

\begin{lemma}
There exist positive constants $c,C$ such that for all $\epsilon_1>0$, we have
  \begin{equation}
    \label{eq:lb1}
    \PP(T_{0, (\Gamma(n),n)}\le 4n-\epsilon_1^{-2/3}n^{1/3})\le C\exp(-c\epsilon_1^{-1}).
\end{equation}
\end{lemma}

\begin{proof}
    Note first that by Proposition \ref{p:tf}, for some constant $C$, we have
    $$\PP(\max_{0\le t\le n}|\Gamma(t)|\ge \frac{1}{10}\epsilon_1^{-1/3}n^{2/3}) \le C\exp(-c\epsilon_1^{-1}).$$
Therefore it suffices to show that for a possibly different constant $C$, we have
$$\PP(\inf_{|x|\le \frac{1}{10}\epsilon_1^{-1/3}n^{2/3}} T_{0,(x,n)} \le 4n-\epsilon_1^{-2/3}n^{1/3})\le C\exp(-c\epsilon_1^{-1}).$$
For $|j|\le \frac{1}{10}\epsilon_1^{-1/3}$, it follows from Proposition \ref{p:para} (and noting that $\EE T_{0,(x,n)}\ge 4n-\frac{1}{2}\epsilon_1^{-2/3}n^{1/3}$ for $|x|\le \frac{1}{10}\epsilon_1^{-1/3}n^{2/3}$)  that 
$$\PP(\inf_{x\in [jn^{2/3}, (j+1)n^{2/3}]} T_{0,(x,n)} \le 4n-\epsilon_1^{-2/3}n^{1/3})\le C\exp(-c\epsilon_1^{-1}).$$
The proof of \eqref{eq:lb1} is completed by a union bound over $j$.
\end{proof}

\begin{lemma}
  \label{l:maxj}
There exist positive constants $c,C$ such that for all $\epsilon_1,n$ satisfying $\epsilon_1 n\geq 1$, we have
    $$\PP(\exists i\in [\![1,\epsilon_1^{-1}]\!]: |\tau_{i}-\tau_{i-1}|\ge \frac{1}{10}\epsilon_1^{-1/3})\le C\exp(-c\epsilon_1^{-1}).$$
\end{lemma}

\begin{proof}
    As in the proof of the previous lemma, it suffices to restrict ourselves to the case $|\tau_i| \le \frac{1}{10}\epsilon_1^{-1}$ for all $i$, as the complement of this has probability at most $C\exp(-c\epsilon_1^{-1})$ by transversal fluctuation estimates.
    Now if $|\tau_{i}-\tau_{i-1}|\ge \frac{1}{10}\epsilon_1^{-1/3}$ it follows that there exists $j\in \Z$, $|j|\le \epsilon_1^{-1}$ such that the geodesic $\Gamma^\dis_{(j(\epsilon_1 n)^{2/3}, (i-1)\epsilon_1n)}$ satisfies 
$$|\Gamma^\dis_{(j(\epsilon_1 n)^{2/3}, (i-1)\epsilon_1n)}(i\epsilon_1 n)-j(\epsilon_1 n)^{2/3}|\ge (\frac{1}{10}\epsilon_1^{-1/3}-2)(\epsilon_1 n)^{2/3}.$$ Now, by Proposition \ref{p:tf}, for fixed $i$ and $j$ this event has probability at most $C\exp(-c\epsilon_1^{-1})$, and here we are using that $\epsilon_1n \geq 1$. The proof is then completed by a union bound over all $i$ and $j$.
\end{proof}

We can now complete the proof of Proposition \ref{l:geodfluc}. 

\begin{proof}[Proof of Proposition \ref{l:geodfluc}]
    Let $K$ be sufficiently large to be fixed later. For $k\ge 1$, let us consider the set $\mathcal{N}_{k,K}$ of all sequences $\{j_0,j_1,\ldots, j_{\epsilon_1^{-1}}\}$ such that $j_0=0$, $\max_{i} |j_{i+1}-j_{i}|\le \frac{1}{10}\epsilon_1^{-1/3}$ and $\sum_{i}|j_{i+1}-j_{i}| \in  [2^{k}K,2^{k+1}K]\epsilon_1^{-1}$. Let us set $\mathscr{T}:=\{\tau_0,\ldots, \tau_{\epsilon_1^{-1}}\}$. By the two preceding lemmas, it suffices to show that for each $k$,
    \begin{equation}
        \label{e:jbound}
        \PP(\mathscr{T}\in \mathcal{N}_{k,K}, T_{0, (\Gamma(n),n)}\ge 4n-\epsilon_1^{-2/3}n^{1/3})\le C\exp(-c(2^{k}K)\epsilon_1^{-1}).
    \end{equation}
    
Fix $k\ge 1$ and $J=\{j_0,j_1,\ldots, j_{\epsilon_1^{-1}}\}\in \mathcal{N}_{k,K}$. Let $I_{i}$ denote the interval $[j_i(\epsilon_1 n)^{2/3}, (j_i+1)(\epsilon_1 n)^{2/3}]\times \{i\epsilon_1 n\}$. Clearly, on the event $\{\mathscr{T}=J\}$, we have 
$$T_{0, (\Gamma(n),n)}\le \sum_{i} \sup_{p\in I_{i-1},q\in I_{i}} T_{p,q}:=X_{J}.$$ 
Clearly, 
$$\PP(\mathscr{T}\in \mathcal{N}_{k,K}, T_{0, (\Gamma(n),n)}\ge 4n-\epsilon_1^{-2/3}n^{1/3})\le \sum_{J\in \mathcal{N}_{k,K}}\PP(X_{J}\ge 4n-\epsilon_1^{-2/3}n^{1/3}).$$

Notice that the random variables $\sup_{p\in I_{i-1},q\in I_{i}} T_{p,q}$ are independent across $i$. Further, notice that our hypothesis that $\max_{i} |j_{i+1}-j_{i}|\le \frac{1}{10}\epsilon_1^{-1/3}$, where we note that the latter does not depend on $n$. Now, this combined with the assumption $\epsilon_1^2 n\geq 1$ implies that the endpoints $p,q$ above satisfy the slope condition in Lemma \ref{p:para}\footnote{When we say certain endpoints/parallelograms satisfy the slope condition in Proposition \ref{p:para}, we shall always mean that there exists a $\Psi$ independent of $n$ and a parallelogram such that the hypothesis of the proposition is valid for the same.}. Indeed, for the slope condition to hold, we require that the quantity $(\frac{1}{10}\epsilon_1^{-1/3})/(\epsilon_1 n)^{1/3}$ stay bounded, and this is true since we have $\epsilon_1^2 n\geq 1$.

Thus using Proposition \ref{p:para} along with the fact that $\EE \sup_{p\in I_{i-1},q\in I_{i}} \EE T_{p,q} \le 4\epsilon_1 n- \frac{1}{4}(j_i-j_{i-1})^{2}(\epsilon_1 n)^{1/3}$ if $|j_i-j_{i-1}|\ge H_0$ for some $H_0>1$ sufficiently large (see (12) in \cite{BGZ19}), we obtain that there exist $c,C>0$ such that
\begin{equation}
  \label{eq:19}
  \PP\left(\dfrac{\sup_{p\in I_{i-1},q\in I_{i}} T_{p,q}-4\epsilon_1 n}{(\epsilon_1 n)^{1/3}}+\frac{1}{4}(j_i-j_{i-1})^21_{|j_i-j_{i-1}|\ge H_0}\ge s\right)\le Ce^{-cs}
\end{equation}
for all large $s$. Notice now that for $J\in \mathcal{N}_{k,K}$,
\begin{align*}
  \frac{1}{4}\sum_{i} (j_i-j_{i-1})^21_{|j_i-j_{i-1}|\ge H_0} &\ge \frac{H_0}{4}\sum_{i} |j_{i}-j_{i-1}|1_{|j_{i}-j_{i-1}|\ge H_0}\nonumber\\
  &\ge \frac{H_02^{k}K\epsilon_1^{-1}}{4}-\frac{H_0}{4}\sum_{i}|j_i-j_{i-1}|1_{|j_{i}-j_{i-1}|\le H_0}\nonumber\\
  &\ge \frac{H_02^{k}K\epsilon_1^{-1}}{4}-\frac{H_0^2\epsilon_1^{-1}}{4}.
\end{align*}
Therefore for $K$ large enough compared to $H_0$ and $H_0>2$, and any $J\in \cN_{k,K}$,
$$\frac{1}{4}\sum_{i} (j_i-j_{i-1})^21_{|j_i-j_{i-1}|\ge H_0} \ge 2^{k-2}K\epsilon_1^{-1}.$$
By using the above along with \eqref{eq:19} and the definition of $X_{J}$, we now obtain that for all $J\in\cN_{k,K}$,
$$X_{J}\le 4n- 2^{k-2}K \epsilon_1^{-2/3}n^{1/3} + (\epsilon_1 n)^{1/3}\sum_{i} Z_{i}$$
where the $Z_{i}$s are independent random variables with uniform exponential tails. 
It therefore follows that for $K$ sufficiently large and any $J\in \cN_{k,K}$,
$$\PP(X_{J}\ge 4n-\epsilon_1^{-2/3}n^{1/3})\le \PP(\sum_{i} Z_{i} \ge 2^{k-3}K\epsilon_1^{-1})\le Ce^{-c2^{k}K\epsilon_1^{-1}}.$$
where the final inequality follows from standard concentration results for sums of independent sub-exponential random variables (see e.g.\ \cite[Theorem 2.8.2]{Ver18}). Some standard counting results show that $|\mathcal{N}_{k,K}|\le (2^{k+3}K)^{\epsilon_1^{-1}}$ and hence \eqref{e:jbound} follows by taking a union bound. This completes the proof of the proposition.  
\end{proof}

\subsubsection{Typical events along the geodesic}
\label{ss:typ}
For $i\in [\![0,\epsilon_1^{-1}-1]\!]$ and $j\in \Z$, let us consider events $G_{i,j}$ with the following properties:

\begin{enumerate}
    \item[(i)] $G_{i,j}$ is measurable with the respect to the weights restricted to the strip $\{t\in [i\epsilon_1n,(i+1)\epsilon_1 n)\}$. In particular, $G_{i,j}$ and $G_{i',j'}$ are independent if $i\ne i'$. 
    \item[(ii)] $G_{i,j}$ is the translate of $G_{0,0}$ under the translation of $\RR^2$ that takes $(0,0)$ to $(j(\epsilon_1 n)^{2/3}, i\epsilon_1 n)$. 
\end{enumerate}
One should interpret $G_{i,j}$ as some local event depending on the weights near the point $(j(\epsilon_1 n)^{2/3}, i\epsilon_1 n)$. Our next result will show that if $G_{i,j}$'s are large probability events, then with high probability, $G_{i,j}$ holds at most locations along the geodesic; i.e., $G_{i,\tau_i}$ holds for most $i$. 

\begin{proposition}
    \label{p:percn}
There exist positive constants $c,C$ such that the following is true: For any $\eta>0$, there exists $\delta>0$ such that if $\PP(G_{i,j})\ge 1-\delta$, then for all $\epsilon_1,n$ satisfying $\epsilon_1^2 n\geq 1$, we have
$$\PP(\#\{i: G_{i,\tau_i} ~\mathrm{occurs}\}\le (1-\eta)\epsilon_1^{-1})\le C\exp(-c\epsilon_1^{-1}).$$
\end{proposition}

\begin{proof}
    Let $T$ be as in the statement of Proposition \ref{l:geodfluc}. Consider the set $\mathscr{N}$ of all sequences $\{j_0,j_1,\ldots, j_{\epsilon_1^{-1}}\}$ such that $j_0=0$ and $\sum_{i}|j_{i+1}-j_{i}|\le T\epsilon_1^{-1}$. Clearly, $|\mathscr{N}|\le (8T)^{\epsilon^{-1}_1}$. %
    Since, by definition, the $\{G_{i_k,j_k}\}_{k\ge 1}$ are mutually independent if the $i_{k}$s are all distinct, it follows that for each 
    $\{j_0,j_1,\ldots, j_{\epsilon_1^{-1}}\}\in \mathscr{N}$, we have    
    $$\PP(\#\{i: G^c_{i,j_i} \textrm{ occurs}\}\ge \eta \epsilon_{1}^{-1}) \le 2^{\epsilon^{-1}} \delta^{\eta \epsilon_{1}^{-1}}.$$
Taking a union bound over all sequences in $\mathscr{N}$, it follows that 
$$\PP(\#\{i: G_{i,\tau_i} \textrm{ occurs}\}\le (1-\eta)\epsilon_1^{-1})\le \PP(\mathscr{T}\notin \mathscr{N})+ (16T)^{\epsilon_1^{-1}}\delta^{\eta \epsilon_1^{-1}}.$$
The result follows by applying Proposition \ref{l:geodfluc} and choosing $\delta$ sufficiently small depending on $\eta$ (and $T$). 
\end{proof}

\subsection{Typical events for volume accumulation}
\label{sec:good}
We shall now define specific events $G_{i,j}$ in the setup described above which will be useful for the proof of the volume accumulation result. We shall first need to make some geometric definitions. In what follows, $M$ will be a parameter whose specific value will be fixed later and sufficiently large (depending on $T$).

\begin{figure}
    \centering
    \includegraphics[width=0.8\textwidth]{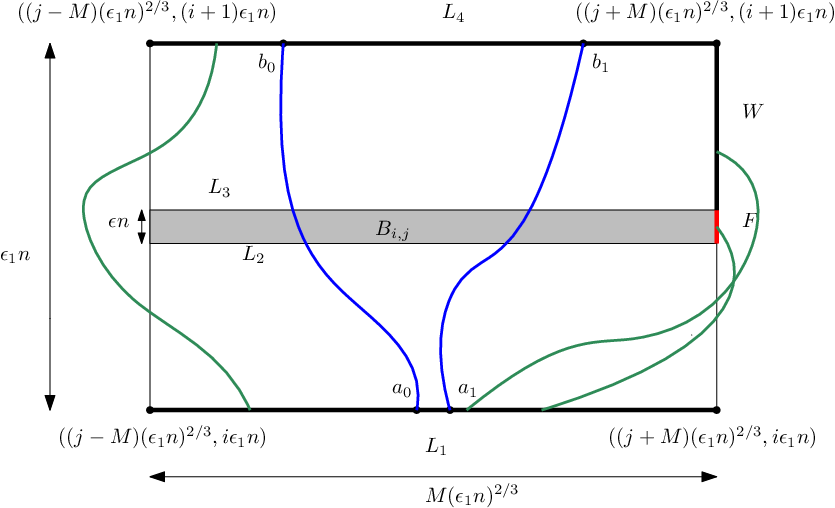}
    \caption{The rectangle $\mbox{Rect}_{i,j}$ and associated geometric constructions in the definition of the event $G_{i,j}$ in Section \ref{sec:good}. We ask for passage times between $L_1$ and $L_2$ and between $L_3$ and $W\cup L_4$ to be typical, and passage times from $L_1$ to $F\cup W\cup L_4$ avoiding $B_{i,j}$ (representative green paths in the figure) to be not too bad compared to typical. We also want the geodesic between $a_0$ and $b_0$ as well as the geodesic between $a_1$ and $b_1$ to have typical transversal fluctuation. %
    }
    \label{fig:good}
\end{figure}

\textbf{Geometric definitions:}
Let $\mbox{Rect}_{i,j}$ denote the rectangle
$$\mbox{Rect}_{i,j}=\{(x,t): t\in [i\epsilon_1n, (i+1)\epsilon_1n], |x-j(\epsilon_1n)^{2/3}|\le M(\epsilon_1 n)^{2/3}\}.$$ For the remainder of the definitions, we shall mostly drop the subscript $i,j$ to reduce notational overhead, but we shall keep the subscript for just $B_{i,j}$ to avoid confusion. Let $L_1$ and $L_4$ denote the bottom and top boundaries of $\mbox{Rect}_{i,j}$ respectively. Let $L_2$ (resp. $L_3$) denote the line segment $\{t=(i+2^{-1})\epsilon_1n\}\cap \mbox{Rect}_{i,j}$ (resp.\ $\{t=(i+2^{-1})\epsilon_1n+\epsilon n\}\cap \mbox{Rect}_{i,j}$) where $\epsilon=L^{-1}\epsilon_1$ and $L$ is some large parameter whose value would be fixed later (depending on all other parameters, but not on $\epsilon_1$). Let the sub-rectangle of $\mbox{Rect}_{i,j}$ enclosed between $L_2$ and $L_3$ be denoted $B_{i,j}$. Let the right boundary of $B_{i,j}$ be denoted by $F$, and the part of the right boundary of $\mbox{Rect}_{i,j}$ above $F$ be denoted $W$. Finally, we define
\begin{align}
  \label{eq:71}
  a_0&= (j(\epsilon_1 n)^{2/3}, i\epsilon_1 n), a_1=((j+1)(\epsilon_1 n)^{2/3}, i\epsilon_1 n),\nonumber\\
   b_0&= ( (j-\frac{M}{4})(\epsilon_1 n)^{2/3}, (i+1)\epsilon_1n), b_1=((j+\frac{M}{4})(\epsilon_1n)^{2/3}, (i+1)\epsilon_1n),
\end{align}
and we refer the reader to Figure \ref{fig:good} for an illustration. We now define $G_{i,j}$ to be the event that the following conditions are satisfied (the parameter $C_0$ will be chosen large depending on $M$). 

\begin{enumerate}
    \item[(i)] For all $p\in L_1$, for all $q\in L_2$, $|T_{p,q}-2\epsilon_1 n|\le C_0(\epsilon_1 n)^{1/3}$. 
    \item[(ii)] For all $p\in L_1$ and all $q\in F\cup W\cup L_4$, 
    $T^{{B_{i,j}^c}}_{p,q}-4(t(q)-t(p)) \geq -C_0(\epsilon_1 n)^{1/3}$;
    recall that $T^{B_{i,j}^{c}}_{p,q}$ denotes the weight of the highest weight path joining $p$ and $q$ that does not intersect the interior of $B_{i,j}$. 
    \item[(iii)] For all $p\in L_3$ and all $q\in L_4\cup W$, $T_{p,q}-4(t(q)-t(p))\le C_0(\epsilon_1 n)^{1/3}$.
    \item[(iv)]  $\sup_{t\in [i\epsilon_1 n, (i+1)\epsilon_1n]} |\gamma^\dis_{a_0,b_0}(t)-j(\epsilon_1 n)^{2/3}|\le \frac{M}{2}(\epsilon_1 n)^{2/3}$,
      $\sup_{t\in [i\epsilon_1 n, (i+1)\epsilon_1 n]} |\gamma^\dis_{a_1,b_1}(t)-j(\epsilon_1 n)^{2/3}|\le \frac{M}{2}(\epsilon_1 n)^{2/3}$, where $\gamma^\dis_{a_i,b_i}$ denotes the geodesic between $a_{i}$ and $b_i$ for $i=0,1$.
\end{enumerate}

It is easy to see that the $G_{i,j}$ defined above satisfy the two conditions mentioned in Section \ref{ss:typ}. Indeed, these events are translates of one another and are determined locally, that is, $G_{i,j}$ depends only on the weights in the strip $\{i\epsilon_1 n \le t \le (i+1)\epsilon_1 n\}$. However, note that $G_{i,j}$ is not measurable with respect to the weights restricted to $\mbox{Rect}_{i,j}$; this will not be important for us. 

The next result says that the $G_{i,j}$s can be made arbitrarily likely by choosing the parameters appropriately. 

\begin{lemma}
    \label{l:gbound}
    Let $\delta>0$ be fixed. There exists $M$ sufficiently large depending on $\delta$ and $C_0$ sufficiently large depending on $\delta$ and $M$ such that for all $\epsilon_1,n$ satisfying $\epsilon_1 n\geq 1$, we have $\PP(G_{i,j})\ge 1-\delta$. 
\end{lemma}

\begin{proof}
 Throughout this proof, $C,c$ shall denote large enough positive constants.  We shall handle each of the four events in the definition of $G_{i,j}$ separately. By Proposition \ref{p:para} and a union bound obtained after discretising $L_1$ and $L_2$ into intervals of size $(\epsilon_1 n)^{2/3}$ and choosing $C_0\gg M^2$, we get that for all $\epsilon_1n\geq 1$,
    $$\PP(\sup_{p\in L_1,q\in L_2} |T_{p,q}-4\epsilon_1 n|\ge C_0(\epsilon_1n)^{1/3})\le CM^2e^{-cC_0^{3/2}},$$
    and we note that the above can be made smaller than $\delta/4$ by choosing $C_0$ sufficiently large depending on $\delta$ and $M$. 

    \begin{figure}
        \centering
        \includegraphics[width=0.8\textwidth]{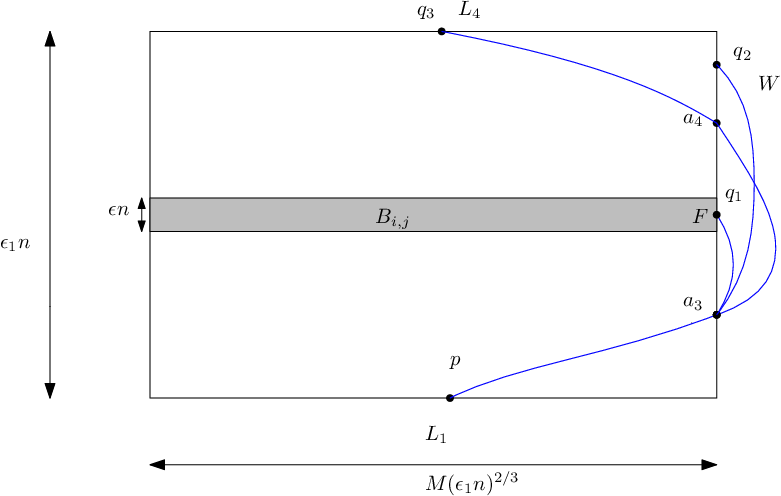}
        \caption{\textbf{Proof of Lemma \ref{l:gbound}, the second event}: We need to show that with good probability, for every $p\in L_1$ and every $q\in F\cup W\cup L_4$ there exist paths from $p$ to $q$ avoiding $B_{i,j}$ such that the weights of these paths are not much smaller than typical. To do this, we construct points $a_3$ and $a_4$ along the right boundary of $\mbox{Rect}_{i,j}$ as shown in the figure and consider paths (marked in blue) going through $a_3$ (when $q\in F\cup W$, representative points $q_1,q_2$ shown in the figure) or both $a_3$ and $a_4$ (if $q\in L_4$, representative point $q_3$ in the figure). By Proposition \ref{p:para} (iii), on a large probability event all the passage times {$T^{B_{i,j}^c}_{a_3,q}$} are not too small for $q\in F\cup W$. Whereas by Proposition \ref{p:para} (i), with good probability, all the passage times $T_{p,a_3}$ and $T_{a_4,q}$ are not too small for $p\in L_1, q\in L_4$. Combining these yields the desired result.}
        \label{fig:a3a4}
    \end{figure}

For the second event (see Figure \ref{fig:a3a4}), consider the points $a_3=((j+M)(\epsilon_1 n)^{2/3},(i+\frac{1}{4})\epsilon_1 n)$ and $a_{4}=((j+M)(\epsilon_1 n)^{2/3},(i+\frac{3}{4})\epsilon_1 n)$. 
   Consider the events
   $$H_1=\{\inf_{p\in L_1} T_{p,a_3}\ge \epsilon_ 1n- \frac{C_0}{4}(\epsilon_1 n)^{1/3}\},$$ 
   $$H_2=\{\inf_{q\in F\cup W} T^{ B_{i,j}^c}_{a_3,q}\ge 4(t(q)-t(a_3))- \frac{C_0}{4}(\epsilon_1 n)^{1/3}\},$$
   $$H_3=\{\inf_{q\in L_4} T_{a_4,q}\ge \epsilon_1 n- \frac{C_0}{4}(\epsilon_1 n)^{1/3}\}.$$
   It is clear that $H_1\cap H_2\cap H_3$ is contained in the second event in the definition of $G_{i,j}$ (indeed, consider the path obtained by concatenating the geodesic from $p$ to $a_3$ with the path attaining $T^{B_{i,j}^c}_{a_3,q}$ if $q\in F\cup W$, or the paths obtained by concatenating the geodesic from $p$ to $a_3$, the path attaining $T^{B_{i,j}^c}_{a_3,a_4}$ and the geodesic between $a_4$ and $q$ if $q\in L_4$) and therefore it suffices to lower bound the probability of the former event. By the same discretization argument as in the previous case we get $\PP(H_1^{c}\cup H_3^{c})\le CMe^{-cC_0^{3/2}}$, and further, by Proposition \ref{p:para} (iii), we obtain $\PP(H_2^{c})\le Ce^{-cC_0}$.
   Therefore, by appropriate choice of parameters, we get $\PP(H_1\cap H_2\cap H_3)\ge 1-\delta/4$.

  The bound for the third event involves a step-back argument. Similar arguments have been used in the literature before \cite{BSS14,BGZ19, BBF22}.  Notice that, for $q\in L_4\cup W$,
  $$\sup_{p\in L_3} T_{p,q}\le T_{a_0,q}-\inf_{p\in L_3} T_{a_0,p}.$$
  Therefore, 
  $$\sup_{p\in L_3, q\in L_4\cup W} T_{p,q}-4(t(q)-t(p))\le 
  \sup_{q\in L_4\cup W}(T_{a_0,q}-4(t(q)-t(a_0))) -\inf_{p\in L_3} (T_{a_0,p}-4(t(p)-t(a_0)).$$
  By Proposition \ref{p:para},
  $$\PP(\sup_{q\in L_4\cup W}(T_{a_0,q}-4(t(q)-t(a_0)))\ge \frac{C_0 (\epsilon_1 n)^{1/3}}{2})\le CMe^{-cC_0^{3/2}}$$
 and 
  $$\PP(\inf_{p\in L_3} (T_{a_0,p}-4(t(p)-t(a_0))\le -\frac{C_0 (\epsilon_1 n)^{1/3}}{2})\le CMe^{-cC_0^{3}}.$$
  Therefore, by choosing the parameters appropriately, we get that the probability that the third condition in the definition of $G_{i,j}$ fails is at most $\delta/4$. 
   
Finally, by Proposition \ref{l:tffin}, it follows that 
$$\PP(\sup_{t\in [i\epsilon_1n, (i+1)\epsilon_1 n]} |\gamma^\dis_{a_0,b_0}(t)-j(\epsilon_1 n)^{2/3}|\ge \frac{M}{2}(\epsilon_1 n)^{2/3})\le Ce^{-cM^3},$$
which can again be made smaller than $\delta/8$ by choosing $M$ sufficiently large depending on $\delta$. The case of $\gamma^\dis_{a_1,b_1}$ is handled similarly, and by taking a union bound over all these events, we get the desired result. 
\end{proof}

Let $S_{i}$ denote the event that $|\Gamma(i\epsilon_1 n)-\Gamma((i+1)\epsilon_1 n)|\le \frac{M}{10}(\epsilon_1 n)^{2/3}$ and recall the constant $T$ as defined in Proposition \ref{l:geodfluc}. For any $\eta>0$, it follows from Proposition \ref{l:geodfluc}, Proposition \ref{p:percn} and Lemma \ref{l:gbound}, that for $M$ sufficiently large depending on $\eta$ and $T$ and $C_0$ sufficiently large depending on $\eta$ and $M$, we have the following for some positive constants $C,c$ and all $\epsilon_1^2n\geq 1$:
$$\PP(\#\{i: G_{i,\tau_i}\cap S_{i}\textrm{ occurs}\}\le (1-\eta)\epsilon_1^{-1})\le \exp(-c\epsilon_1^{-1}).$$

Notice now that $G_{i,\tau_i}\cap S_{i}$ implies $\widetilde{G}_{i,\tau_i}\cap \widetilde{S}_{i}$ where $\widetilde{G}_{i,j}$ is the same as $G_{i,j}$ except the fourth condition (the transversal fluctuation condition) removed and $\widetilde{S_{i}}$ denotes the event that 
$$\sup_{t\in [i\epsilon_1 n, (i+1)\epsilon_1 n]} |\Gamma(t)-\tau_i (\epsilon_1 n)^{2/3}|\le \frac{M}{2}(\epsilon_1 n)^{2/3}.$$ Define $G^*_i=\widetilde{G}_{i,\tau_i}\cap \widetilde{S}_{i}$. It follows from the above discussion that we have the following proposition. 

\begin{proposition}
    \label{p:manygood}
   There exist positive constants $C,c$ such that the following holds: For any $\eta>0$, there exist appropriate choices of parameters $M$ and $C_0$, such that for all $\epsilon_1,n$ satisfying $\epsilon_1^2 n\geq 1$, we have
    $$\PP(\#\{i: G^*_{i}~\mathrm{occurs}\}\le (1-\eta)\epsilon_1^{-1})\le C\exp(-c\epsilon_1^{-1}).$$
\end{proposition}

We now define the region $B^*_{i}$ to be the part of $B_{i,\tau_i}$ to the right of $\Gamma$. The reason we switched to $G^*_{i}$ from $G_{i,\tau_i}\cap S_{i}$ is the following lemma whose proof is clear from the definition. 

\begin{lemma}
    \label{l:mble}
    The events $G^*_{i}$ are measurable with respect to the $\sigma$-algebra generated by $\Gamma$ and the weights in the region $\Z^2\setminus (\cup_{i} B^*_{i})$. 
\end{lemma}

\subsection{Barrier events}
We shall now condition on $\Gamma$ and the weight configuration on $\Z^2\setminus (\cup_{i} B^*_{i})$; denote this conditioning by $\mathcal{F}$. By Lemma \ref{l:mble}, the set 
$N=\{i: G^*_i\textrm{ occurs}\}$ is measurable with respect to $\mathcal{F}$. Let $\mathcal{G}$ denote the $\sigma$-algebra generated by $\mathcal{F}$, $N$, and the weights on the vertices lying in $B^*_{i}$ for $i\notin N$. 

Let us now define events $\cR_{i}$ measurable with respect to the weight configuration in $B^*_{i}$. Let $p_i$ denote the point where $\Gamma$ intersects the bottom boundary of $B^*_{i}$. We shall again mostly drop the subscript $i$ in the following definitions, and for example, we shall use $p$ to refer to the point $p_i$. Similar to the case of $B_{i,j}$s, we always keep the subscript for $B^*_{i}$ in order to avoid confusion.

\begin{figure}
    \centering
    \includegraphics[width=0.7\textwidth]{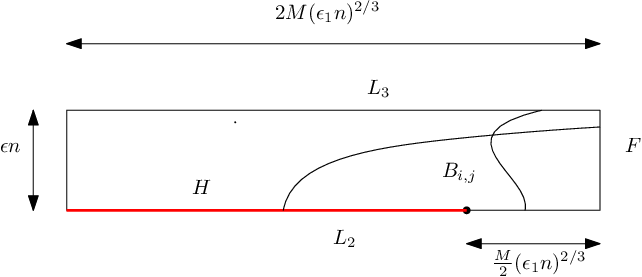}
    \includegraphics[width=0.7\textwidth]{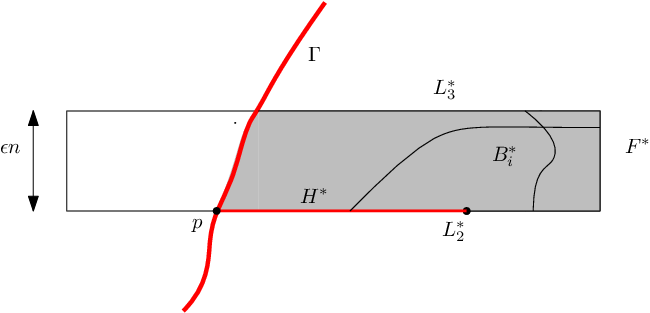}
    \caption{\textbf{The barrier events $\widetilde{\mathcal{R}}_{i,j}$ and $\mathcal{R}_{i}$}: the top panel shows the event $\widetilde{\mathcal{R}}_{i,j}$ in the rectangle $B_{i,j}$ (the indices are all suppressed in the figure). It asks that the paths crossing $B_{i,j}$ or the paths starting at $H$ (i.e., not too close to the right boundary) and exiting via the right boundary $F$ are worse than typical. The bottom panel illustrates the events $\mathcal{R}_{i}$ in the rectangle $B_{i,\tau_i}$ through which the semi-infinite geodesic $\Gamma$ passes, $B_i^*$ denotes the part of $B_{i,\tau_i}$ to the right of $\Gamma$. This event asks that any crossing of $B_i^*$ without hitting $\Gamma$ has weight much worse than typical and the same holds for any path started at $H^*$, exiting via $F^*$ and not meeting $\Gamma$.}
    \label{fig:barrier}
\end{figure}

Let $L^*_2$ (resp.\ $L^*_3$) denote the bottom and the top boundary of $B_i^*$ and let $F^*$ denote its right boundary. Let $H^*$ denote the interval of length $(\epsilon n)^{2/3}$ (recall that $\epsilon=\epsilon_1/L$ is the height of $B_i^*$) of $L^*_2$ whose left end point is $p$.

Let $\cR_i$ denote the event (see Figure \ref{fig:barrier}) that the following two conditions are satisfied:
\begin{enumerate}
    \item[(i)] $\sup_{q\in L^*_2, r\in L_3^*} T^{B^*_i\setminus \Gamma}_{q,r}-4\epsilon n \le -L(\epsilon n)^{1/3}$;
    \item[(ii)] $\sup_{q\in H^*, r\in F^*} T^{B^*_i\setminus \Gamma}_{q,r}-4(t(r)-t(q)) \le -L(\epsilon n)^{1/3}.$
\end{enumerate}
Recall here that $T^{B^*_i\setminus \Gamma}$ denotes the induced passage time where only paths contained in $B^*_i$ and disjoint from $\Gamma$ are considered. Notice also that since we have conditioned on $\mathcal{G}$, and are working on the event $G^*_{i}$, it follows (from the transversal fluctuation condition in the definition of $G^*_{i}$) that $H^*$ is to the left of the vertical line $x=(\tau_{i}+\frac{M}{2}+1)(\epsilon_1 n)^{2/3}$. This implies that paths that are considered in the second condition above all have large transversal fluctuations. 

The event $\cR_i$ is decreasing in the environment in $B_i^*$ and we shall use the FKG inequality to control its probability. For any path $\gamma$, the conditioning $\{\Gamma=\gamma\}$ is a negative conditioning on the weights in $\Z^2\setminus \gamma$. Therefore, by the {FKG inequality}, conditioning on $\Gamma$ makes the weights stochastically smaller than typical on subsets of $\Z^2\setminus \Gamma$, and thus any decreasing event (measurable with respect to the weights in such a region) becomes conditionally more likely than typical. To make things formal, we make the following definitions. 

Recall the rectangle $B_{i,j}$ (see Figure \ref{fig:good}) in the definition of $G_{i,j}$. Using $H$ to denote the sub-interval of $L_2$ to the left of the vertical line $\{x=(j+\frac{M}{2}+1)(\epsilon_1 n)^{2/3}\}$, we consider the event $\widetilde{\cR}_{i,j}$ that the following conditions (see Figure \ref{fig:barrier}) are satisfied. 
\begin{enumerate}
    \item[(i)] $\sup_{q\in L_2, r\in L_3} T^{B_{i,j}}_{q,r} -4\epsilon n \le -L(\epsilon n)^{1/3}$;
    \item[(ii)] $\sup_{q\in H, r\in F} T^{B_{i,j}}_{q,r}-4(t(r)-t(q)) \le -L(\epsilon n)^{1/3}$.
\end{enumerate}
 We note that all of $L_2,L_3,H,F$ implicitly depend on $i,j$ but this dependency has been suppressed. Since the probability $\widetilde{\cR}_{i,j}$ does not depend on $i,j$, we will often just work with $\widetilde{\cR}_{0,0}$. Clearly, the defining conditions above are stronger than those defining $\cR_i$, and therefore, as discussed above, by the FKG inequality, we get that for all $i\in N$, 
 \begin{equation}
   \label{eq:22}
  \PP(\cR_{i}\mid \mathcal{G})\ge \PP(\widetilde{\cR}_{0,0}).
 \end{equation}
 Further, since the regions $B^*_{i}$ are disjoint we get the following lemma. 

\begin{lemma}
    \label{l:dom}
In the above setup, the conditional law of 
$\{1_{\cR_{i}}:i\in N\}$ given $\mathcal{G}$, stochastically dominates the law of $\{X_{i}:i\in N\}$ where $X_{i}$s are independent realizations of $1_{\widetilde{\cR}_{0,0}}$. 
\end{lemma}

We now need to lower bound the probability of $\widetilde{\cR}_{0,0}$, and this is done in the following lemma. 

\begin{lemma}
\label{l:barlb}
  There exists a constant $K'>1$ such that for $L$ sufficiently large depending on $M$ and as long as $\epsilon n\geq K'$, we have $\PP(\widetilde{\cR}_{0,0})=\rho>0$ for a constant $\rho$ depending only on $M$ and $L$.
\end{lemma}

\begin{proof}
    Since the two events defining $\widetilde{\cR}_{0,0}$ are both decreasing in the weights in $B_{0,0}$, they are positively correlated by the FKG inequality. Therefore, it suffices to show separately that each of the events have probability bounded away from $0$. For the first event, this follows from Proposition \ref{p:barrier}. For the second event, we shall discretize and partition $H$ into $ML^{2/3}$ many intervals of size $(\epsilon n)^{2/3}$ each. Fix such an interval $I$. We shall show that 
    \begin{equation}
        \label{eq:111}
        \PP(\sup_{q\in I, r\in F} T^{B_{0,0}}_{q,r}-4(t(r)-t(q)) \le -L(\epsilon n)^{1/3})\ge \rho_*>0
    \end{equation}
    where $\rho_{*}$ depends only on $M$ and $L$ (and not on $I$ or $\epsilon$). For this, we shall use another step-back argument. Let $p'$ denote the left end-point of $I$. Consider the point $q'$ with $x(p')=x(q')$ and $t(q')=t(p')-\epsilon n$. Notice that the pairs of points $(q',r)$ where $r\in F$ satisfies the slope condition %
    in Proposition \ref{p:para}, and therefore with probability at least $0.9$, 
    $$\sup_{r\in F} (T_{q',r}-\EE T_{q',r}) \le C(\epsilon n)^{1/3}$$ and 
    $$\inf_{q\in I} (T_{q',q}-\EE T_{q',q}) \ge -C(\epsilon n)^{1/3}$$
    for some absolute positive constant $C$. Observe that
    \begin{itemize}
    \item $\sup_{q\in I, r\in F} T_{q,r} \le \sup_{r\in F} T_{q',r}- \inf_{q\in I} T_{q',q}$
    \item $\EE T_{q',r}-4(t(r)-t(q')) \le -\frac{1}{10}M^{2}L^{4/3}(\epsilon n)^{1/3}$, where here, we fix $K'$ to be a large enough constant and recall the assumption $\epsilon n\geq K'$.
    \item $|\EE T_{q',q}-4(t(q)-t(q'))|\le C(\epsilon n)^{1/3}$ for some $C$ not depending on $M$ or $L$.
    \end{itemize}
  In view of the above, \eqref{eq:111} follows by choosing $L$ sufficiently large. By the FKG inequality, the probability that the second event in the definition of $\widetilde{\cR}_{0,0}$ holds is therefore at least $\rho_{*}^{ML^{2/3}}>0$. This completes the proof of the lemma. 
\end{proof}

We digress now for a moment to show another application of the above argument which will be used later in Lemma \ref{lem:13}. Recall the rectangle $U_{\Delta}$ from Proposition \ref{p:barrier}; this is a rectangle with height $n$ and width $2\Delta n^{2/3}$. Let $L_{U}$ (resp.\ $R_{U}$) denote its left (resp.\ right) side. We have the following lemma.

\begin{lemma}
    \label{l:bstrong}
    For any $\Delta,M>0$, there exists $\beta=\beta(\Delta,M)>0$ such that for all $n$ large, we have 
    $$\PP\left(\sup_{p\in L_{U},q\in R_{U}} T_{p,q}-4|t(q)-t(p)|\le -Mn^{2/3}\right)\ge \beta.$$
\end{lemma}

\begin{proof}
    Clearly it suffices to prove the result for $M$ sufficiently large and this is what we shall do. Let us fix $K$ such that $K\gg M$. By Proposition \ref{p:barrier} and the FKG inequality it suffices to show that for some $\beta'>0$,
    $$\PP\left(\sup_{p\in L_{U},q\in R_{U}: |t(q)-t(p)|\le K^{-1}n} T_{p,q}-4|t(q)-t(p)|\le -Mn^{2/3}\right)\ge \beta'.$$
    Now divide the line segments $L_{U}$ and $R_{U}$ into segments of length $K^{-1}n$, denote them by $L^{i}_{U}$ (resp.\ $R^{i}_{U}$) for $i=1,2,\ldots, K$ increasing along the time direction. 
    Let $A_{i}$ denote the event 
    $$\bigcup_{j=i,i+1}\left\{ \sup_{p\in L^{i}_{U},q\in R^{j}_{U}:t(q)>t(p)} T_{p,q}-4(t(q)-t(p))\le -Mn^{2/3}\right\}$$
    and let $B_{i}$ denote the event 
    $$\bigcup_{j=i,i+1}\left\{ \sup_{p\in L^{j}_{U},q\in R^{i}_{U}:t(p)>t(q)} T_{q,p}-4(t(p)-t(q))\le -Mn^{2/3}\right\}.$$
    It clearly suffices to show that $\PP(\cap_{i} A_{i}\cap B_{i})$ is bounded away from $0$. Noting that $\PP(A_{i})=\PP(B_{i})$ by symmetry and using the FKG inequality again we are finally reduced to showing $\PP(A_{1})$ is bounded away from $0$. 

    Let $p_0$ denote the bottom endpoint of $L_U^1$ and let $p_1$ be the point $p_0-(0,2K^{-1}n)$ %
    (i.e., the point $2K^{-1}n$ distance below $p_0$ in the time direction). Arguing is the the proof of Lemma \ref{l:barlb}, on the intersection of the  events 
    $\{\sup_{q\in R^{1}_{U}\cup R^{2}_U} T_{p_1,q}-4(t(q)-t(p_1)) \le -2Mn^{1/3}\}$ and $\{\inf_{q\in L^{1}_{U}} T_{p_1,q}-4(t(q)-t(p_1)) \ge -Mn^{1/3}\}$, $A_1$ holds. By noticing that $\sup_{q\in R^{1}_{U}\cup R^{2}_U} \EE T_{p_1,q}-4(t(q)-t(p_1))\le -M'n^{1/3}$ for some $M'\gg M$ if $K\gg M$, we obtain by using Proposition \ref{p:para} that both the above events have probability at least 0.9 if $M$ is sufficiently large. The lower bound for $\PP(A_1)$ follows by a union bound and this completes the proof. 
\end{proof}

We are now ready to prove the main technical result of this section. 

\begin{proposition}
    \label{p:density}
   There exist positive constants $C,c$ and $K>1$ such that the following holds. For appropriate choices of the parameters $M, L$ and $C_0$, there exists $\kappa>0$ such that as long as we have $\epsilon^2 n\geq K$, we have
    $$\PP(\#\{i: G^*_i\cap \cR_{i}~\mathrm{occurs}\}\le \kappa \epsilon_1^{-1}) \le C\exp(-c\epsilon_1^{-1}).$$ 
\end{proposition}

\begin{proof}
  Let the parameters be fixed such that the conclusions of all the previous lemmas hold; also, we choose $K=K'^2$, where $K'$ is as in the statement of Lemma \ref{l:barlb}. Now, since $\epsilon_1\geq \epsilon$ and since $K>1$, the condition $\epsilon^2n \geq K$ implies that $\epsilon_1^2n\geq 1$. As a result, it follows from Proposition \ref{p:manygood} that $\PP(\# N\le (1-\eta)\epsilon_1^{-1}) \le C\exp(-c\epsilon_1^{-1})$.
  
Further, since $\epsilon^2 n \geq K>1$, we also have $\epsilon n \geq \sqrt{K}=K'$ and thus we can invoke Lemma \ref{l:barlb}. Indeed, on the event that $\{\# N\ge (1-\eta)\epsilon_1^{-1}\}$, it follows from Lemma \ref{l:dom} and Lemma \ref{l:barlb} that 
    $\#\{i: G^*_i\cap \cR_{i}~\mathrm{occurs}\}$ is stochastically larger than a $\mbox{Bin}((1-\eta)\epsilon_1^{-1}),\rho)$ random variable. From a standard lower tail large deviation estimate for the binomial random variable, it follows that for $\kappa$ sufficiently small (depending on $\rho$) we get 
    $$\PP(\#\{i: G^*_i\cap \cR_{i}~\mathrm{occurs}\} \le \kappa \epsilon_1^{-1} \big\lvert \mathcal{G}, \# N\ge (1-\eta)\epsilon_1)\le C\exp(-c\epsilon_1^{-1}).$$
    The proof of the proposition is completed by a union bound. 
\end{proof}

\subsection{Volume accumulation on the good event}
Recalling the notation of the previous subsection, let $p_i$ denote the point where $\Gamma$ intersects the line $\{t=(i+2^{-1})\epsilon_1 n\}$, and let $q_i$ denote the point where $\Gamma$ intersects the line $t=i\epsilon_1 n$. Further, let $p'$ be a point on the line $\{t=(i+2^{-1})\epsilon_1 n\}$ to the right of $p_i$ with $|p_i-p'|\le (\epsilon n)^{2/3}$. We shall show that if $p'$ is close to $p_i$, then, on the event $G^*_i\cap \cR_{i}$ it is very likely that the upward semi-infinite geodesic from $p'$ will coalesce with $\Gamma$ below the line $\{t=(i+1)\epsilon_1 n\}$. We now move towards a formal statement. 

Recall from Section \ref{s:lppnotation} that for $p,q\in \Z^2$, the Busemann function $\cB^\dis(p,q)$ is defined by $$\cB^\dis(p,q)=T_{p,r}-T_{q,r},$$ where $r$ is a point in $\Gamma^\dis_{p}\cap \Gamma^\dis_{q}$. Since we are on the probability one set where all upward semi-infinite geodesics are unique and coalesce, this is well defined. Consider the process $\{S^{n}(x)\}_{x\in \Z}$ defined by $S^{n}(x)=\cB^\dis((x,n), (0,n))$, i.e., the process of Busemann increments along the anti-diagonal line $\{t=n\}$. The following result is well known, see e.g.\ \cite[Theorem 4.2]{Sep17}.

\begin{proposition}
\label{p:Busemann}
    In the above notation, $S^{n}(\cdot)$ is a two sided random walk with i.i.d.\ increments with mean $0$ that has the increment distribution same as $X-Y$ where $X$ and $Y$ are independent $\exp(1/2)$ random variables.  
\end{proposition}

We now have the following result for coalescence. Recall that $\epsilon=L^{-1}\epsilon_1$, where $L$ is always taken to be sufficiently large.

\begin{figure}[htbp!]
     \centering
         \centering
         \includegraphics[width=0.5\textwidth]{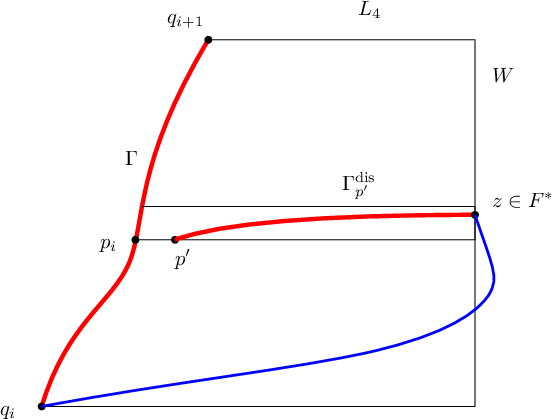}
     \caption{\textbf{Proof of Proposition \ref{l:contra}, case 1}: we show that on $G^{*}_{i}\cap \mathcal{R}_{i}$, if the geodesic $\Gamma^\dis_{p'}$ does not coalesce with $\Gamma$ before $q_{i+1}$, then the Busemann increment $|\cB^{\rm dis}(p_i,p')|$  must be large. This figure depicts Case 1 where the first exit point of $\Gamma^\dis_{p'}$ from $B_i^{*}$ is at the right boundary $F^*$. By the linearity of Busemann functions, it suffices to show that $|\cB^\dis(q_i,p')-2\epsilon_1 n|$ is large and this is achieved by comparing $\Gamma^{\rm dis}_{p'}$ with the path which goes from $q_{i}$ to $z$ along the geodesic $\gamma^\dis_{q_i,z}$ and then follows $\Gamma^\dis_{p'}$. %
     }
     \label{fig:vol1}
     \end{figure}

\begin{proposition}
    \label{l:contra}
    Let $i$ be such that $G^*_i\cap \cR_{i}$ holds. Suppose further that $p'$ is a point as above on the line $\{t=(i+2^{-1})\epsilon_1 n\}$ to the right of $p_i$ with $|p_i-p'|\le (\epsilon n)^{2/3}$. Then either, (a) $\Gamma^\dis_{p'}$ coalesces with $\Gamma$ before $q_{i+1}$, or (b) $|\cB^\dis(p_i,p')|\ge \frac{L}{8}(\epsilon n)^{1/3}$. 
\end{proposition}

 \begin{proof}
We shall assume that $\Gamma^\dis_{p'}$ does not coalesce with $\Gamma$ before $q_{i+1}$ and show that this implies $|\cB^\dis(p_i,p')|\ge \frac{L}{8}(\epsilon n)^{1/3}$. Notice that the Busemann function satisfies the following linear relation 
$$\cB^\dis(p_{i},p')=\cB^\dis(p_{i},q_{i})+\cB^\dis(q_{i},p').$$
Also, since $q_i,p_{i}\in \Gamma$,
$$\cB^\dis(p_{i},q_{i})=-T_{p_{i},q_{i}},$$ and therefore, by (i) in the conditions listed after \eqref{eq:71}, $|\cB^\dis(p_{i},q_{i})-2\epsilon_1 n|\le C_0(\epsilon_1 n)^{1/3}=C_0L^{1/3}(\epsilon n)^{1/3}$. Since $L$ is large enough compared to $C_0$, it suffices to show that $|\cB^\dis(q_i,p')-2\epsilon_1 n|\ge \frac{L}{4}(\epsilon n)^{1/3}$. 

Let $r$ be the point where $\Gamma^\dis_{q_i}$ and $\Gamma^\dis_{p'}$ meet. There are two cases to consider. %

     \textbf{Case 1:} The first exit point of $\Gamma^\dis_{p'}$ from $B_i^*$ is at  $z\in F^*$. In this case (see Figure \ref{fig:vol1}), considering the path from $q_i$ to $r$ obtained by concatenating the best path from $q_{i}$ to $z$ avoiding $B_i^*$ and the restriction of $\Gamma^\dis_{p'}$ between $z$ and $r$, it follows that 
    $$\cB^\dis(q_i,p')\ge T^{(B_i^*)^c}_{q_{i},z}-T^{B^*_{i}\setminus \Gamma}_{p',z} \ge 2\epsilon_1 n+ (-C_0L^{1/3}+L)(\epsilon n)^{1/3} \ge 2\epsilon_1 n+\frac{1}{4}L(\epsilon n)^{1/3},$$
as required, where the penultimate inequality above follows from the definitions of $G^*_i\cap \cR_{i}$ whereas the final inequality is obtained by choosing $L$ sufficiently large compared to $C_0$.

\begin{figure}[htbp!]
     \centering
         \centering
         \includegraphics[width=0.5\textwidth]{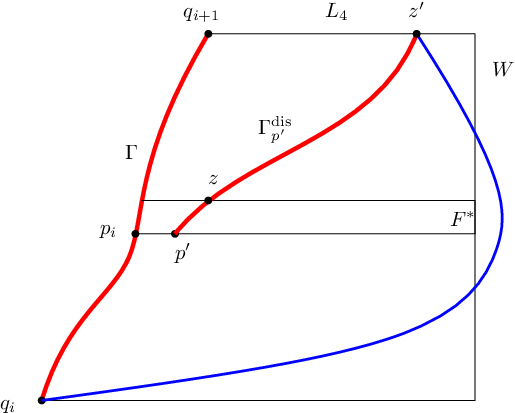}
     \caption{\textbf{Proof of Proposition \ref{l:contra}, case 2}: in this case first exit point $z$ of $\Gamma^\dis_{p'}$ from $B_i^{*}$ is at the top boundary of $B^*$. Denoting by $z'$, the intersection point of $\Gamma^\dis_{p'}$ with $W\cup L_4$, as before we show that $|\cB^\dis(q_i,p')-2\epsilon_1 n|$ is large by comparing $\Gamma^\dis_{p'}$ with the path which goes from $q_{i}$ to $z'$ along the geodesic $\gamma^\dis _{q_{i},z'}$ and then follows the geodesic $\Gamma^\dis_{p'}$.}
     \label{fig:vol2}
     \end{figure}

\textbf{Case 2:}  The first exit point of $\Gamma^\dis_{p'}$ from $B_i^*$ is at  $z\in L^*_3$. In this case (see Figure \ref{fig:vol2}), let $z'$ be the first point where $\Gamma^\dis_{p'}$ intersects $L_4\cup W$, where $L_4$ and $W$ denote the corresponding line segments in the event $G_{i,\tau_i}$. As before, following the best path from $q_i$ to $z'$ avoiding $B_i^*$ and then following $\Gamma^\dis_{p'}$ to $r$, we get
$$\cB^\dis(q_i,p')\ge T^{(B_i^*)^c}_{q_i,z'}-T^{B_i^*\setminus \Gamma}_{p',z}-T_{z,z'}\ge 2\epsilon_1 n+(-2C_0L^{1/3}+L)(\epsilon n)^{1/3}\ge 2\epsilon_1 n+\frac{1}{4}L(\epsilon n)^{1/3},$$
as required, where the penultimate inequality above follows from the definitions of $G^*_i\cap \cR_{i}$, whereas the final inequality is obtained by choosing $L$ sufficiently large compared to $C_0$. This completes the proof of the proposition. 
 \end{proof}   

To ensure that coalescence actually happens at a positive density of locations along the geodesic, our next task is to control the Busemann increments on the line $t=(i+\frac{1}{2})\epsilon_1 n$ around the point $\Gamma((i+\frac{1}{2})\epsilon_1n)$. To this end, we have the following global bound. 

\begin{lemma}
    \label{l:bussebd}
  There exist constants $c,C$ and $K>1$ such that the following holds. Let $\mathcal{A}_{\epsilon}$ denote the event that for all $i\in [\![0,\epsilon_1^{-1}]\!]$ and for all $|x_1|, |x_2| \le \epsilon^{-1/3}n^{2/3}$ with $|x_1-x_2|\le 2\epsilon^{5/3}n^{2/3}$ we have 
    $|\cB^\dis(p_1,p_2)|< \frac{L}{8}(\epsilon n)^{1/3}$ where $p_{j}=(x_j, (i+\frac{1}{2})\epsilon_1n)$. Then for all $\epsilon,n$ satisfying $\epsilon^4 n\geq K$, we have $\PP(\mathcal{A}^{c}_{\epsilon})\le CL^{-1}e^{-c\epsilon^{-1}}$. %
\end{lemma}

\begin{proof}
  Let $S(\cdot)$ denote a two sided random walk as in Proposition \ref{p:Busemann} with $S(0)=0$. %
 Now, as we shall outline how to obtain soon, it follows by basic arguments that for any positive $\delta, T$ such that $T/\sqrt{\delta} n^{1/3}$ is small enough, we have
  \begin{equation}
    \label{eq:1}
    \PP(\sup_{x\in [0,\delta n^{2/3}]} S(x)-\inf_{x\in [0,\delta n^{2/3}]} S(x) \ge T\sqrt{\delta} n^{1/3}) \le 2\PP(\sup_{x\in [0,\delta n^{2/3}]} S(x) \ge \frac{T}{2}\sqrt{\delta}n^{1/3}) \le 2e^{-cT^2}.
  \end{equation}
We now give a quick overview of the above. First, for $m\in \NN$, we write $S(m)=\sum_{j=1}^m Z_j$, where i.i.d.\ random variables $Z_j$ have the explicit description as described in Proposition \ref{p:Busemann}. Now, since $S$ is a martingale with respect to its natural filtration and since the $Z_j$s are sub-exponential random variables, $\exp(\lambda S)$ is a submartingale for all small enough values of $\lambda>0$. Thus by Doob's maximal inequality, to obtain \eqref{eq:1}, we need only show that for a well-chosen choice of $\lambda$, we have $\PP(e^{\lambda S(m)}\geq e^{\lambda T\sqrt{m}})\leq e^{-cT^2}$ as long as $T/\sqrt{m}$ is small enough. However, this can be obtained by a simple calculation-- indeed, $\EE\exp(\lambda S(m))=(\EE \exp(\lambda Z_1))^m=(1-4\lambda^2)^{-m}$, and it can be checked that choosing $\lambda = T/(8\sqrt{m})$ suffices since $T/\sqrt{m}$ is taken to be small enough.
    
    We now use \eqref{eq:1} to obtain the needed result. We choose $\delta=4\epsilon^{5/3}$, $T=\Theta(\epsilon^{-1/2})$ and note that the condition $T/(\sqrt{\delta} n^{1/3})$ being small enough is now equivalent to $\epsilon^4 n$ being large enough and we now choose and fix $K$ to assure that this holds. Now, we cover the interval $[-\epsilon^{-1/3}n^{2/3}, \epsilon^{-1/3}n^{2/3}]$ by $\epsilon^{-2}$ many overlapping intervals of the form $[2j\epsilon^{5/3}n^{2/3}, (2j+4)\epsilon^{5/3}n^{2/3}]$. This ensures that any two points within $2\epsilon^{5/3}$ distance of each other are contained in such an interval. By now invoking \eqref{eq:1} and taking a union bound over these intervals as well as all possible values of $i$ ($\epsilon_1^{-1}=L^{-1}\epsilon^{-1}$ many of them), we obtain
    $$\PP(\mathcal{A}^{c}_{\epsilon})\le 2L^{-1}\epsilon^{-1}e^{-c\epsilon^{-1}},$$
    and this completes the proof.
\end{proof}

The final task is to ensure that the geodesic $\Gamma$ must accumulate some amount of volume if the coalescence described in Proposition \ref{l:contra} happens at many locations along the geodesic. To this end, we need some uniform control on the transversal fluctuations of semi-infinite geodesics started from various points. We need the following lemma where the choice of $\epsilon^{5/3}$ is guided by Lemma \ref{l:bussebd}.

\begin{lemma}
    \label{l:tfglobal}
   There exist constants $c,C>0$ such that the following holds. For $0\le i \le \epsilon_1^{-1}$ and $-\epsilon^{-2}\le j \le \epsilon^{-2}$, %
 consider the point $q_{i,j}:=(j\epsilon^{5/3}n^{2/3}, (i+\frac{1}{2})\epsilon_1n)$. Let $\mathcal{H}_\epsilon$ denote the event that for all $i,j$ as above, $\sup_{s\in [0,1]}|\Gamma^\dis_{q_{i,j}}(((i+\frac{1}{2})\epsilon_1+s\epsilon^{3})n)-j\epsilon^{5/3}n^{2/3}|\le \frac{1}{3}\epsilon^{5/3}n^{2/3}$. Then for all $\epsilon, n $ satisfying $\epsilon^3 n\geq 1$, we have $\PP(\mathcal{H}_\epsilon^{c})\le Ce^{-c\epsilon^{-1}}$.
\end{lemma}

\begin{proof}
    Fix $i$ and $j$. It follows from Proposition \ref{p:tf} that $$\PP(\sup_{s\in [0,1]}|\Gamma^\dis_{q_{i,j}}(((i+\frac{1}{2})\epsilon_1+s\epsilon^{3})n)-j\epsilon^{5/3}n^{2/3}|\ge \frac{1}{3}\epsilon^{5/3}n^{2/3})\le Ce^{-c\epsilon^{-1}}.$$
    The proof follows by a union bound.
\end{proof}

We shall now obtain a volume accumulation lower bound on a high probability event. Recall the definition of $V_{R}(\Gamma;n)$, the \textbf{volume accumulated by $\Gamma$ to the right until time $n$}. Let $\widetilde{\mathcal{H}}_\epsilon$ denote the event that $\sup_{t\in [0,n]}|\Gamma(t)|\le \epsilon^{-1/3}n^{2/3}$. 

\begin{proposition}
    \label{p:va}
    Suppose that we have $\epsilon^4 n\geq K$, where $K$ is the constant from Lemma \ref{l:bussebd}. Then, with $\kappa$ is as in the statement of Proposition \ref{p:density}, on the event
    \begin{displaymath}
    \mathcal{A}_{\epsilon}\cap \mathcal{H}_\epsilon\cap \widetilde{\mathcal{H}}_\epsilon \cap \{\#\{i: G^*_i\cap \cR_{i}~\mathrm{occurs}\}\ge \kappa \epsilon_1^{-1}\},  
    \end{displaymath}
we have $V_R(\Gamma;n)\ge \frac{\kappa L^{-1}}{3}\epsilon^{11/3}n^{5/3}$. 
\end{proposition}

\begin{figure}
    \centering
    \includegraphics[width=0.6\textwidth]{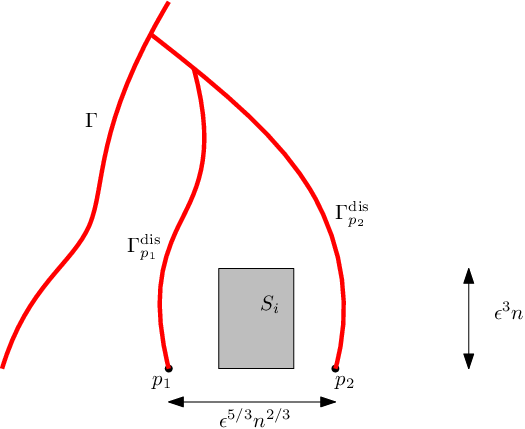}
    \caption{\textbf{Proof of Proposition \ref{p:va}}: on the large probability events, we have that the geodesics started from both $p_1$ and $p_2$ meet with the geodesic $\Gamma$ before time $n$. Now, we construct the rectangle $S_{i}$ such that using transversal fluctuation estimates, on a large probability event, the rectangle $S_i$ lies between the geodesics $\Gamma^{\rm dis}_{p_1}$ and $\Gamma^{\rm dis}_{p_2}$, hence all vertices of $S_{i}$ contribute to $V_{R}(\Gamma;n)$. The rectangles $S_{i}$ are disjoint by construction and have area $\Theta(\epsilon^{14/3}n^{5/3})$; the proof is completed by summing over the indices $i$ where $G^*_{i}\cap \mathcal{R}_{i}$ occur.}
    \label{fig:volume}
\end{figure}

\begin{proof}
We refer the reader to Figure \ref{fig:volume} for an illustration of the proof. Fix $i$ such that $G^*_i\cap \cR_{i}$ holds. Let $j$ be such that $\Gamma((i+\frac{1}{2})\epsilon_1n)< j\epsilon^{5/3} n^{2/3}\le \Gamma((i+\frac{1}{2})\epsilon_1 n)+\epsilon^{5/3}n^{2/3}$. Consider the points $p_1=(j\epsilon^{5/3}n^{2/3}, (i+\frac{1}{2})\epsilon_1n)$ and $p_2=p_1+(\epsilon^{5/3}n^{2/3},0)$. Clearly, by the definitions of $\mathcal{A}_{\epsilon}$, $\mathcal{H}_{\epsilon}$ along with Proposition \ref{l:contra}, the geodesics $\Gamma^\dis_{p_1}$ and $\Gamma^\dis_{p_2}$ coalesce with $\Gamma$ below $t=n$. Let us now consider the rectangle $S_{i}$ with the bottom side having 
    endpoints $p_1+(\frac{1}{3}\epsilon^{5/3}n^{2/3},0)$ and $p_2-(\frac{1}{3}\epsilon^{5/3}n^{2/3},0)$ and with height $\epsilon^{3}n$. By the definition of $\mathcal{H}_\epsilon$, (and also using the fact that we are on $\widetilde{\mathcal{H}}_\epsilon$) it follows that the geodesic $\Gamma^\dis_{p_1}$ lies to the left of $S_{i}$, while the geodesic $\Gamma^\dis_{p_2}$ lies to the right of $S_{i}$. It now follows by planarity that for all $q\in S_i$, the geodesic $\Gamma^\dis_{q}$ coalesces with $\Gamma$ below $t=1$. Since the area of $S_i$ is $\frac{1}{3}\epsilon^{14/3}n^{5/3}$ and since the  $S_{i}$ are disjoint for distinct $i$, summing over different values of $i$ (at least $\kappa \epsilon_1^{-1}=\kappa L^{-1}\epsilon^{-1}$ many by hypothesis), we get the desired result.  
\end{proof}
The proof of Proposition \ref{t:vaexp} is now immediate. 
\begin{proof}[Proof of Proposition \ref{t:vaexp}]
  In view of Proposition \ref{p:va}, we set $\delta=\kappa L^{-1}\epsilon^{11/3}/3$, and thus the condition $\delta^{12/11}n\geq K'$ is equivalent to the condition $\epsilon^4n\geq K'(\kappa L^{-1}/3)^{-12/11}$. Now, we fix a large enough choice of $K'$ such that that $\epsilon,n$ are guaranteed to satisfy the conditions stated in Propositions \ref{p:density}, \ref{p:va}. Now, by using Proposition \ref{p:va} and combining the probability bounds from Proposition \ref{p:density}, Lemma \ref{l:bussebd}, Proposition \ref{p:tf} and Lemma \ref{l:tfglobal} by a union bound, the result follows.    
\end{proof}

\subsection{Proof of Proposition \ref{t:va}}

Finally, we use the convergence of exponential LPP to the directed landscape to complete the proof of Proposition \ref{t:va}.
\begin{proof}[Proof of Proposition \ref{t:va}]
  We work with the coupling from Proposition \ref{prop:6}. Analogously to Definition \ref{d:va}, for $p=(x,s)\in \RR^2$ and $\ell>0$, let $\wcV_R(\Gamma^n_p;\ell)$ denote the set of points $q=(y,t)$ with $t\in [s,s+\ell]$ which are to the right of $\Gamma^n_p$ in the sense that $y>\Gamma^n_p(t)$, with the additional property that the semi-infinite geodesic $\Gamma^n_q$ from $q$ coalesces with $\Gamma_0^n$ below $t=s+\ell$. Denote the Lebesgue area of $\wcV_R(\Gamma^n_p;\ell)$ by $\widetilde{V}_R(\Gamma^n_p;\ell)$. As a consequence of Proposition \ref{t:vaexp}, we know that for any fixed $\epsilon>0$ and for all $n$ large enough depending on $\epsilon$, we have $\PP(\widetilde{V}_R(\Gamma^n_0;1)\leq \epsilon)\leq C\exp(-c\epsilon^{-3/11})$, and this implies that
  \begin{equation}
    \label{eq:72}
    \PP(\limsup_{n\rightarrow\infty}\{\widetilde{V}_R(\Gamma^n_0;1)\geq \epsilon\})\geq 1-C\exp(-c\epsilon^{-3/11}),
  \end{equation}
  and we note that we are considering a $\limsup$ of events in the above equation as opposed to considering a $\limsup$ of the corresponding probabilities. In other words, there is an event $\cE_\epsilon$ with $\PP(\cE_\epsilon)\geq 1-C\exp(-c\epsilon^{-3/11})$ on which there is a subsequence $\{n_i\}$ for which we have $\widetilde{V}_R(\Gamma_0^{n_i};1)\geq \epsilon$ for all $i$. By using the geodesic convergence in Proposition \ref{prop:6}, it can be obtained that there is a random compact set $K$ such that a.s.\ for all $n$ large enough, $\wcV_R(\Gamma_0^n;1)\subseteq K$. We now work with this compact set $K$ and the subsequence $\{n_i\}$ from above. Let the set $\wcV$ be defined by
  \begin{equation}
    \label{eq:73}
    \wcV=
    \left\{
      q\in K:q\in \wcV(\Gamma^{n_i}_0;1)
     \textrm{ for infinitely many } i \right\}.
 \end{equation}
 We note that on the event $\cE_\epsilon$, since $\widetilde{V}_R(\Gamma_0^{n_i};1)\geq \epsilon$ for all $i$, we have $\leb(\wcV)\geq \epsilon$ as well, and this uses the compactness of the set $K$.

 Finally, we claim that on the event $\cE_\epsilon$, we have $\wcV\subseteq \wcV(\Gamma_0;1)$, and this can be seen as a consequence of the convergence of discrete infinite geodesics to the continuum ones, as stated in Proposition \ref{prop:6}. This completes the proof since it implies that on $\cE_\epsilon$, we have $\widetilde{V}_R(\Gamma_0;1)\geq \epsilon$ and we know that $\PP(\cE_\epsilon^c)\leq C\exp(-c\epsilon^{-3/11})$.
\end{proof}

\subsection{Some variants of the volume accumulation result}
\label{sec:variants}
Before closing this section, we now discuss a few variations of Proposition \ref{t:vaexp} that can be obtained by the same argument. We begin by noting that in Figure \ref{fig:volume}, the rectangles $S_i\subseteq \mathcal{V}_R(\Gamma;n)$ were at an $O(\epsilon^{5/3}n^{2/3})$ space distance from the geodesic $\Gamma$. That is, using $B_r(\Gamma)$ to denote lattice points at a spatial distance of at most $r$ from the semi-infinite geodesic $\Gamma$, for some positive constant $\beta$, the rectangles $S_i$ therein satisfied $S_i\subseteq B_{\beta \epsilon^{5/3} n^{2/3}}(\Gamma)$.  Using this observation, we can obtain the following slightly stronger version of Proposition \ref{t:vaexp}.
    \begin{proposition}
      \label{prop:9}
   There exist constants $C,c,K',\beta>0$ such that for all $\delta,n$ satisfying $\delta^{12/11} n\geq K'$, we have
    \begin{equation}
      \label{eq:13}
      \PP(
    \#(\mathcal{V}_R(\Gamma;n)\cap B_{\beta\delta^{5/11} n^{2/3}}(\Gamma)) \leq \delta n^{5/3}))\leq C\exp(-c\delta^{-3/11}),      
    \end{equation}  
    \end{proposition}
    To clarify the reasoning behind the appearance of the $\delta^{5/11}n^{2/3}$ term in the above, note that, just as in the completion of the proof of Proposition \ref{t:vaexp}, we set $\delta\sim \epsilon^{11/3}$ and thus $\epsilon^{5/3}n^{2/3}\sim \delta^{5/11}n^{2/3}$.

    One can also obtain a version of Proposition \ref{t:vaexp} for point-to-point geodesics instead of infinite geodesics. Indeed, if we use $\gamma^{\dis}_{0,\mathbf{n}}$ to denote the exponential LPP geodesic from $(0,0)$ to $\mathbf{n}=(n,n)$, then for any fixed $\alpha \in (0,2)$, we can define $\mathcal{V}_R(\gamma^{\dis}_{0,\mathbf{n}};\mathbb{L}_{\alpha n})$ to be the set of all lattice points $q$ such that $t(q)\geq 0$, $x(q)\geq \gamma^\dis_{0,\mathbf{n}}(t(q))$ and such that $\gamma^\dis_{0,\mathbf{n}}\cap \gamma^\dis_{q,\mathbf{n}}\cap \mathbb{L}_{\alpha n}\neq \emptyset$. One now has the following analogue of Proposition \ref{t:vaexp}.
      \begin{proposition}
        \label{prop:1}
 For any fixed $\alpha\in (0,2)$-- there exist constants $C,c,K',\beta>0$ such that for all $\delta,n$ satisfying $\delta^{12/11}n \geq K'$, we have
    \begin{equation}
      \label{eq:14}
      \PP(\#(\mathcal{V}_R(\gamma^\dis_{0,\mathbf{n}};\mathbb{L}_{\alpha n})\cap B_{\beta\delta^{5/11} n^{2/3}}(\gamma^\dis_{0,\mathbf{n}}))\leq \delta n^{5/3})\leq C\exp(-c\delta^{-2/33}).
    \end{equation}  
      \end{proposition}
      \begin{proof}[Proof sketch]
        Except for one difference, the proof of the above is exactly the same. The only difference is that the random walk $S$ defined by $S(x)=\cB^{\mathrm{dis}}( (x,n),(0,n))$ and appearing in Proposition \ref{l:contra} and Lemma \ref{l:bussebd} is now replaced by the function $S'(x)=T( (x,\alpha n), (n,n))$. That is, intuitively, Busemann functions or distances from infinity are simply replaced by distances to the point $\mathbf{n}$. Note that $S'$ is no longer a random walk; thus the argument in Lemma \ref{l:bussebd} used to obtain a stretched exponential decay estimate for the tails of its increments does not work-- however, such an estimate has already been shown indeed, the statement and proof of \cite[Theorem 3]{BG21} yield a version of \eqref{eq:1} with the term $e^{-cT^2}$ replaced by the weaker bound $e^{-cT^{4/9}}$. Thus, due to this change of exponent from $2$ to $4/9$, the $3/11$ exponent in Proposition \ref{t:vaexp} changes to the exponent $(3/11)\times (4/9)/2=2/33$. We note that the $4/9$ exponent from \cite[Theorem 3]{BG21} is not optimal, and thus the exponent in Proposition \ref{prop:1} is not optimal as well.
      \end{proof}

    Finally, though we work in the context of exponential LPP in this paper, the techniques used to obtain the above results can also be use to obtain corresponding volume accumulation results for semi-infinite and point-to-point geodesics in other integrable models of last passage percolation. We do not provide precise statements for these in this paper; for a statement for point-to-point geodesics in Brownian last passage percolation, we refer the reader to \cite[Proposition 63]{Bha24+}.

\section{Fractal geometry in the volume parametrization}
\label{sec:frac}
In this short section, we prove Theorem \ref{thm:4}, which discusses the way in which the fractal geometry of a set $\cA\subseteq \RR^2$ is reflected in that of the set $\eta^{-1}(\cA)$.

 \begin{proof}[Proof of Theorem \ref{thm:4}]
   To obtain the first statement, simply note that the map $\eta$ and thus $\eta\lvert_{\eta^{-1}(\cA)}$ is locally $1/5-$ H\"older continuous if we endow $\RR^2$ with the intrinsic metric. %
   Since $\eta(\eta^{-1}(\cA))=\cA$, a standard argument (see e.g.\ \cite[Lemma 4.3]{Dau23+}) shows $\dim (\eta^{-1}(\cA))\geq \dim_{\mathrm{in}}(\cA)/5$. %

   We now come to the proof of the second statement.  Since $\cA$ is measurable with respect to $\cL$, there exists a measurable function from the space in which $\cL$ takes its values to the space of closed subsets of $\RR^2$, both with their usual $\sigma$-algebras, %
   for which
 $f(\cL)=\cA$ holds almost surely. For $z\in \RR^2$, we use the notation $\cL^\mathrm{tr}_z$ to denote the translated landscape defined by $\cL^\mathrm{tr}_z(\cdot;\cdot)=\cL(z+\cdot;z+\cdot)$. Let $\Pi$ be a rate $1$ Poisson process on $\RR^2$ independent of $\cL$ and define the augmented fractal $\cA^\dagger=\bigcup_{z\in \Pi} f(\cL^\mathrm{tr}_z)$. We note that $\cA^\dagger$ has the property that for any $\epsilon>0$, the probability $\PP(B^\mathrm{in}_\epsilon(z)\cap \cA^\dagger\neq \emptyset)$ %
 does not depend on $z$. We also note that it suffices to show that $\dim \eta^{-1} (\cA^\dagger)\leq d/5$ a.s.\ since that would imply that $\dim \eta^{-1}(\cA)\leq d/5$ almost surely.

 Though we had only assumed that there exists a constant $C>0$ such that $\EE\leb(B^\mathrm{in}_\epsilon(\cA))\leq C\epsilon^{5-d}$, by using that $\cA\subseteq [-M,M]^2$, it can in fact be obtained that for some constant $C_1>0$ and for all $\epsilon\in (0,1)$, $\EE\leb(B^\mathrm{in}_\epsilon(\cA^\dagger)\cap [-M,M]^2)\leq C_1\epsilon^{5-d}$. %
To see this, we note that if we define $N=\# (\Pi\cap [-2M,2M]^2)$, then conditional on $\{N=n\}$, $X^\epsilon\coloneqq \leb(B^\mathrm{in}_\epsilon(\cA^\dagger)\cap [-M,M]^2)/\epsilon^{5-d}\leq X^\epsilon_1+\cdots+X^\epsilon_n$, where the $X^\epsilon_i$ are all marginally distributed as $\leb(B^\mathrm{in}_\epsilon(\cA)/\epsilon^{5-d}$. As a consequence, we obtain that
 \begin{equation}
   \label{eq:85}
   \EE\leb(B^\mathrm{in}_\epsilon(\cA^\dagger)\cap [-M,M]^2)\leq (\EE N )\EE\leb(B^\mathrm{in}_\epsilon(\cA))\leq C(\EE N) \epsilon^{5-d},
 \end{equation}
 and thus we have the desired estimate with $C_1=C\EE N$.
Now, by using that
   \begin{equation}
     \label{eq:52}
    M^2\PP(B^\mathrm{in}_\epsilon(0)\cap \cA^\dagger\neq \emptyset)= \int_{z\in [-M/2,M/2]^2}\PP(B^\mathrm{in}_\epsilon(z)\cap \cA^\dagger\neq \emptyset)dz\leq \EE\leb(B^\mathrm{in}_\epsilon(\cA^\dagger)\cap [-M,M]^2)
  \end{equation}
  holds for all $\epsilon<M/2$, along with the estimate obtained just above, we obtain that there exists a constant $C$ depending on $M$ such that we have $\PP(B^\mathrm{in}_\epsilon(0)\cap \cA^\dagger\neq \emptyset)\leq C\epsilon^{5-d}$ for all $\epsilon\in (0,1)$.

  Consider an interval $[v_0,v_0+\epsilon]\subseteq \RR$. By (4) in Theorem \ref{thm:1}, we know that $\cL^\mathrm{tr}_{\eta(v_0)}\stackrel{d}{=}\cL$ and this implies that for any fixed $\theta>0$, we have for some positive constants $C,c_1$,
  \begin{align}
    \label{eq:53}
    \PP([v_0,v_0+\epsilon]\cap \eta^{-1}(\cA^\dagger)\neq \emptyset)&= \PP(\eta([0,\epsilon])\cap \cA^\dagger\neq \emptyset)\nonumber\\
                                                                    &\leq \PP(B^\mathrm{in}_{\epsilon^{1/5-\theta}}(0)\cap \cA^\dagger\neq \emptyset)+\PP(\eta([0,\epsilon])\not\subseteq B^\mathrm{in}_{\epsilon^{1/5-\theta}}(0))\nonumber\\
    &\leq C\epsilon^{1-d/5-\theta(5-d)}+e^{-c_1\epsilon^{-60\theta/77}},
   \end{align}
   where the second term above is obtained by applying Proposition \ref{prop:tails}. Now, by a standard argument, we obtain $\dim \eta^{-1} (\cA^\dagger)\leq d/5+\theta(5-d)$. Since $\theta$ was arbitrary, we obtain $\dim \eta^{-1} (\cA^\dagger)\leq d/5$ and this completes the proof.
\end{proof}

\begin{remark}
  \label{rem:1}
  The usual way to upper bound the Hausdorff dimension of a random fractal $\cA\subseteq [-M,M]^2$ by $d$ is to instead bound the upper Minkowski dimension, which is a priori larger, by showing that for some constant $C$ and for all $z\in [-M,M]^2$ and $\epsilon$ small enough, we have $\PP(B^\mathrm{in}_\epsilon(z)\cap \cA\neq \emptyset)\leq C\epsilon^{5-d}$ and this in particular implies that $\EE\leb (B_\epsilon(\cA))\leq 4CM^2\epsilon^{5-d}$ for all small $\epsilon$. %
  As a result, though the second statement in Theorem \ref{thm:4} does not say that $\dim \eta^{-1}(\cA) \leq (\dim_\mathrm{in}\cA)/5$, it does yield the above in practice, since the notions of Hausdorff dimension and Minkowski dimension typically coincide for random fractals. We also note that in the typical case, when $\cA$ is not a bounded, we can apply Theorem \ref{thm:4} with $\cA\cap [-M,M]^2$ instead of $\cA$. As an example, if $\cA=\Gamma_0$, then it is not difficult to check that $\EE \leb(B^\mathrm{in}_\epsilon(\cA)\cap [-M,M]^2)/\epsilon^{5-d}$ is bounded for any fixed $M>0$, all $\epsilon\in (0,1)$ and any $d>3$. As a consequence, Theorem \ref{thm:4} would imply that $\dim \eta^{-1}(\cA\cap [-M,M]^2)\leq 3/5$ and thereby $\dim \eta^{-1}(\cA)\leq 3/5$ by a countable union argument. Similarly, for $\cA=\RR\times \{0\}$ and $\cA=\{0\}\times \RR$, one can obtain that $\dim\eta^{-1}(\cA)$ is a.s.\ equal to $2/5$ and $3/5$ respectively.%
\end{remark}
Before concluding this section, we give a short computation of the fractal dimension of the topological boundaries of segments of the Peano curve $\eta$. Though not difficult to prove, this result finds application in justifying the use of the Portmanteau lemma in the measure theoretic arguments used in the proof of Theorem \ref{thm:3}.
\begin{lemma}
   \label{lem:12}
   Almost surely, simultaneously for every $v<w$, the Euclidean Hausdorff dimension of $\partial \eta([v,w])$ as a subset of the plane is $4/3$.
 \end{lemma}
 \begin{proof}
   Begin by noting that $\partial\eta([v,w])=\partial \eta( (-\infty,w])\Delta \partial \eta((-\infty,v])$, where $\Delta$ denotes the symmetric difference of sets. Further, it is not difficult to observe that $\partial \eta ((-\infty,v])=\Gamma_{\eta(v)}\cup \Upsilon_{\eta(v)}$, for some choice of geodesic and interface emanating from $\eta(v)$. Thus, to complete the proof, it is sufficient to prove that almost surely, any segment of any upward semi-infinite geodesic or downward interface has Euclidean dimension $4/3$. Due to the duality between the geodesic tree and the interface portrait, it suffices to show the above just for infinite geodesics, but this is the content of Proposition \ref{prop:8}. %
\end{proof}

\section{The variation process}
\label{sec:var}
The aim of this section is to prove Theorem \ref{thm:3} and Theorem \ref{thm:5}. For both the proofs, it is important to obtain some independence between different parts of the Peano curve $\eta$, and the next result will be helpful (see Figure \ref{fig:1}) in this regard. We will work in the setting of Proposition \ref{prop:3} and will use the notation $\eta ^\mathrm{re} _{z_i}$ to denote the Peano curve corresponding to $\cL^\mathrm{re}_{z_i}$ normalized to satisfy $\eta^\mathrm{re}_{z_i}(0)=z_i$. For our setting, we will only use the above-mentioned proposition in two cases, when the value of $k$ therein is $1$ or $2$. In the following lemma, we use the notation $v_z$ to denote the a.s.\ unique (by (2) in Theorem \ref{thm:1}) number $\eta^{-1}(z)\in \RR$.
\begin{figure}
  \centering
  \includegraphics[width=0.75\textwidth]{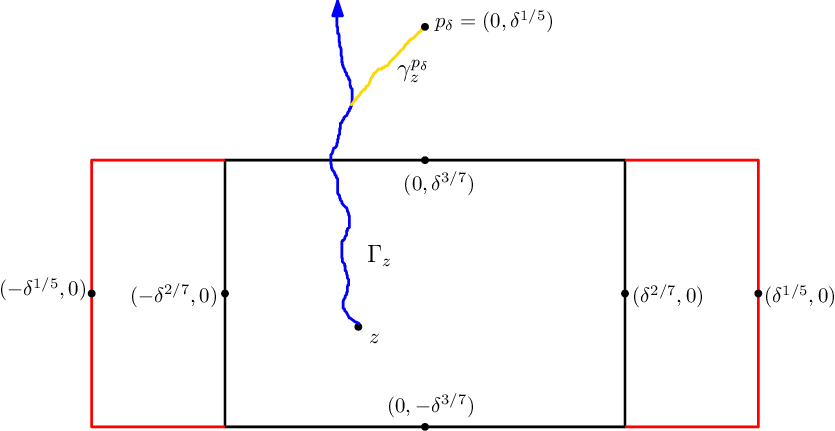}
  \caption{\textbf{The events in the proof of Lemma \ref{lem:19}:} The event $E_\delta^2$ asks that all geodesics $\Gamma_z,\gamma_z^{p_\delta}$ starting from points $z\in R_\delta$ (the black rectangle) do not exit the larger red rectangle. On the other hand, the event $E_\delta^1$ asks that the geodesics $\Gamma_z$ (the blue path) and $\gamma_z^{p_\delta}$ (a portion of the blue path concatenated with the golden path) only split after exiting the rectangle $R_\delta$.}
  \label{fig:local}
\end{figure}
\begin{lemma}
  \label{lem:19}
  Working in the coupling from Proposition \ref{prop:3} with $k=1$, for any fixed $z\in \RR^2$, there is an event $\cE^z_\delta\in \sigma(\cL)$ with probability going to $1$ as $\delta\rightarrow 0$ on which $\eta\lvert_{[v_z-\delta,v_z+\delta]}(v_z+\cdot)=\eta^\mathrm{re}_{z}\lvert_{[-\delta,\delta]}(\cdot)$, and further $\eta_z^\mathrm{re}([-\delta,\delta])\subseteq z+ [-\delta^{1/5},\delta^{1/5}]^2$. %
\end{lemma}
\begin{proof}
  First, by Lemma \ref{lem:5}, we can assume that $z=0$ without loss of generality, and we do so throughout the proof. Also, by using Proposition \ref{prop:tails}, it is clear that the probability of $\eta_0^\mathrm{re}([-\delta,\delta])\subseteq [-\delta^{1/5},\delta^{1/5}]^2$ goes to $1$ as $\delta\rightarrow 0$, and thus it suffices to focus on the first assertion in the lemma.

  It can be seen that it suffices to create an event $A_\delta$ with probability going to $1$ as $\delta\rightarrow 0$ on which $\eta\lvert_{[-\delta,\delta]}$ is determined by $\cL\lvert_{[-\delta^{1/5},\delta^{1/5}]^2}$. Indeed, this would be enough since then on an analogous event, $\eta^\mathrm{re}_{0}\lvert_{[-\delta,\delta]}$ %
  would be determined by $\cL^\mathrm{re}_0\lvert_{[-\delta^{1/5},\delta^{1/5}]^2}$, and then by applying Proposition \ref{prop:3} with $\epsilon=\delta^{1/5}$, we would obtain the required event $\cE_\delta$ on which $\eta\lvert_{[-\delta,\delta]}=\eta^\mathrm{re}_{0}\lvert_{[-\delta,\delta]}$. The focus of the rest of the proof is to construct the required event $A_\delta$.

  Let $H_\delta$ denote the event that $\sup_{v\in[-\delta,\delta]}|\eta_u(v)|\leq \delta^{3/7},\sup_{v\in[-\delta,\delta]}|\eta_h(v)|\leq \delta^{2/7}$. Let $R_\delta$ denote the rectangle $[-\delta^{2/7},\delta^{2/7}]\times [-\delta^{3/7}, \delta^{3/7}]$. We know from Propositions \ref{p:utail}, \ref{p:htail} that $\PP(H_\delta)\rightarrow 1$ as $\delta\rightarrow 0$. On the event $H_\delta$, we note that $\eta\lvert_{[-\delta,\delta]}$ is determined by $\cT_\uparrow\cap R_\delta$. We now construct an event $E_\delta$ satisfying $\PP(E_\delta)\rightarrow 1$ as $\delta\rightarrow 0$ on which $\cT_\uparrow\cap R_\delta$ is determined by $\cL\lvert_{[-\delta^{1/5},\delta^{1/5}]^2}$ and having done so, we can simply define $A_\delta=H_\delta\cap E_\delta$.

  Let $p_\delta$ denote the point $(0,\delta^{1/5})$ and define events $E_\delta^1,E_\delta^2$ by
  \begin{align}
E^1_\delta&=
\left\{\text{For each } z\in R_\delta \text{ and every } \Gamma_z, \text{ there is a } \gamma_z^{p_\delta} \text{ such that }
  \Gamma_z\cap R_\delta=\gamma_z^{p_\delta}\cap R_\delta \text{ and vice versa}
    \right\},\nonumber\\
    E_\delta^2&=
\left\{
  |\gamma_z^{p_\delta}(t)|,|\Gamma_z(t)|\leq \delta^{1/5}\text{ for all }z\in R_\delta \text{ and } t\in [-\delta^{3/7}, \delta^{3/7}] 
\right\},
\end{align}
and we define $E_\delta=E_\delta^1\cap E_\delta^2$. We refer the reader to Figure \ref{fig:local} for a depiction of these events. It is easy to see that on the event $E_\delta$, $\cT_\uparrow \cap R_\delta$ is determined by $\cL\lvert_{[-\delta^{1/5},\delta^{1/5}]^2}$. By the uniform transversal fluctuation estimates for finite geodesics in Proposition \ref{prop:5} along with the $n\rightarrow \infty$ limit of the prelimiting transversal fluctuation estimates for infinite geodesics in Proposition \ref{p:tfstrong}, we obtain that $\PP(E_\delta^2)\rightarrow 1$ as $\delta\rightarrow 0$, and it remains to show the same for $E_\delta^1$. By the KPZ scaling of the directed landscape, and locally defining $q(\delta)=(0,\delta^{-8/35})\in \RR^2$, we obtain that $\PP(E_\delta^1)$ is equal to
\begin{equation}
  \label{eq:30}
  \PP
  \left(
    \text{For each }z\in R_1\text{ and every } \Gamma_z, \text{ there is a } \gamma_z^{q(\delta)} \text{ such that }\Gamma_z \cap R_1=\gamma_z^{q(\delta)}\cap R_1 \text{ and vice versa}
  \right).
\end{equation}
We will now argue by contradiction. If it were that $\PP(E_\delta^1)$ did not go to $1$ as $\delta \rightarrow 0$, then one of two cases must occur with positive probability. In the first case, we would have a sequence $\{\delta_i\}_{i\in \NN}$ converging to $0$ and a random sequence of points $z_i=(x_i,s_i)\in R_1$ such that there is a geodesic $\gamma_{z_i}^{q(\delta_i)}$ for which $\gamma_{z_i}^{q(\delta_i)}\cap R_1$ is not equal to $\Gamma_{z_i}\cap R_1$ for any choice of the latter geodesic. In this case, we use the compactness property of geodesics from Proposition \ref{prop:7} to obtain that there exists a $z\in R_1$ and a geodesic $\Gamma_z$ for which $\gamma_{z_i}^{q(\delta_i)}$ converges to $\Gamma_{z}$ locally uniformly, and further, for all $i$ large enough and any fixed $t>1$, we have $\gamma_{z_i}^{q(\delta_i)}(t)=\Gamma_{z}(t)$. Let $\Gamma_{z_i}$ be a choice of an infinite geodesic from $z_i$ and by compactness again, we note that $\Gamma_{z_i}$ converges locally uniformly to a geodesic $\Gamma_z'$ which might possibly be different from $\Gamma_z$. Similarly, for all $i$ large enough and any fixed $t>1$, we have $\Gamma_{z_i}(t)=\Gamma_z'(t)$. However, since $\cT_\uparrow$ is one-ended, we can choose a $t$ large enough so that $\Gamma_z(t)=\Gamma_z'(t)$. This now leads to a contradiction as by concatenating $\gamma_{z_i}^{q(\delta_i)}\lvert_{[s_i,t]}$ with $\Gamma_z'\lvert_{[t,\infty)}$, we obtain a semi-infinite geodesic starting from $z_i$ which when intersected with $R_1$ yields $\gamma_{z_i}^{q(\delta_i)}\cap R_1$.

In the second case, we have a sequence $\{\delta_i\}_{i\in \NN}$ converging to $0$ and a random sequence of points $z_i\in R_1$ such that there is a geodesic $\Gamma_{z_i}$ for which $\Gamma_{z_i}\cap R_1$ is not equal to $\gamma_{z_i}^{q(\delta_i)}\cap R_1$ for any choice of the latter geodesic. We again use a similar argument in this case. By compactness (Proposition \ref{prop:7}), we find a $z\in R_1$ and geodesics $\Gamma_z,\Gamma'_z$ such that the $\Gamma_{z_i}$ converge to $\Gamma_z$ in the overlap sense and some geodesics $\gamma_{z_i}^{q(\delta_i)}$ converge to $\Gamma'_{z_i}$ in the overlap sense. Since there exists a $t$ such that $\Gamma_{z_i}(t)=\Gamma_{z_i}'(t)$, if we take $\delta_i$ small enough, we will have that $\Gamma_{z_i}(t)=\gamma_{z_i}^{q(\delta_i)}(t)$ and thus $\Gamma_{z_i}$ in particular contains the restriction to $R_1$ of a geodesic between $z_i$ and $q(\delta_i)$, and this contradicts our assumption.

The above two cases together show that $\PP(E_\delta^1)\rightarrow 1$ as $\delta\rightarrow 0$, thereby showing that $\PP(A_\delta)\rightarrow 1$ as $\delta \rightarrow 0$, and completing the proof. %
\end{proof}

\noindent With the above independence statement at hand, we now begin to work towards the proof of Theorem \ref{thm:3}.

\subsection{Proof of Theorem \ref{thm:3}}
\label{sec:proof1}

Let $\Pi_n^1$ denote a Poisson point process on $\RR$ of rate $n$ which is independent of $\cL$, and let $\Pi_n^2=\{\eta(v):v\in \Pi_n^1\}$. The following observation provides the additional symmetry which makes it beneficial for us to prove the variation result using points coming from a Poisson process instead of points with uniform deterministic spacing.
\begin{lemma}
  \label{lem:8}
  $\Pi_n^2$ is a Poisson point process on $\RR^2$ with rate $n$ and is independent of $\cL$.
\end{lemma}
\begin{proof}
To avoid confusion, we locally use $\leb_\RR,\leb_{\RR^2}$ to denote the Lebesgue measures on $\RR$, $\RR^2$ respectively. By definition, we know that for any measurable set $E\subseteq \RR^2$, $\#\{\Pi_n^2\cap E\}=\#\{\Pi_n^1\cap \eta^{-1}(E)\}$. Since $\eta$ is always parametrized according to its Lebesgue area, we have $\leb_\RR (\eta^{-1}(E))=\leb_{\RR^2}(E)$. Using this along with the fact that $\Pi_n^1$ is independent of $\cL$, we obtain that even conditional on $\cL$, $\#\{\Pi_n^1\cap \eta^{-1}(E)\}$ is distributed as a Poisson variable with rate $\leb_{\RR^2}(E)$. To finish the proof, it remains to show that $\Pi_n^2$ has the independence structure present in a Poisson process. To see this, note that for any disjoint measurable sets $E,F\subseteq \RR^2$, the sets $\eta^{-1}(E),\eta^{-1}(F)$ are disjoint as well. Thus by using the properties of the one dimensional Poisson process $\Pi_n^1$, we obtain that conditional on $\cL$, $\#\{\Pi_n^2\cap E \}=\#\{\Pi_n^1\cap \eta^{-1} (E) \}$ and $\#\{\Pi_n^2\cap F \}=\#\{\Pi_n^1\cap \eta^{-1} (F)\}$ are independent. This completes the proof.
\end{proof}
Given a locally finite point process $\psi$ on $\RR$, order its points in increasing order and denote them by $\{\dots,p_{-1},p_0,p_{1},\dots\}$. %
We define the atomic measure $\mu_{n,\psi}$ by
\begin{equation}
  \label{eq:39}
  \mu_{n,\psi}(\{p_i\})=|\eta_u(p_{i+1})-\eta_u(p_i)|^{5/3}+|\eta_h(p_{i+1})-\eta_h(p_i)|^{5/2}.
\end{equation}
Similarly, given a locally finite point process $\phi$ on $\RR^2$, define $\nu_{n,\phi}$ to be the measure $\eta^*(\mu_{n,\eta^{-1}(\phi)})$, the push-forward of $\mu_{n,\eta^{-1}(\phi)}$ by $\eta$. We will often work in the case when the point process $\psi= \Pi_n^1$, and in this case, we use $\mu_n,\nu_n$ to denote $\mu_{n,\Pi_n^1},\nu_{n,\Pi_n^2}$ respectively. Recall the notation for the variation increment $\In^5_{v}(f,\epsilon)$ from \eqref{eq:38}. We now have the following simple lemma regarding the moments of $\In^5_0(\eta,1)$.
\begin{lemma}
  \label{lem:22}
  For any $p>0$, we have $\EE \In^5_0(\eta,1)^p<\infty$. Further, if $X$ is a positive random variable independent of $\cL$, then we have $\EE\In^5_0(\eta,X)^p=\EE[ X^p]\EE[\In^5_0(\eta,1)^p]$.
\end{lemma}
\begin{proof}
  The finiteness of the moments is an immediate consequence of the tail estimates from Propositions \ref{p:utail}, \ref{p:htail}. For the latter statement, we first note that by the scaling relation from Lemma \ref{lem:3}, for any $v_0>0$, we have $\In^5_0(\eta,v_0)\stackrel{d}{=} v_0 \In^5_0(\eta,1)$, and thus by the independence between $X$ and $\cL$, we have $\In^5_0(\eta,X)\stackrel{d}{=} X \In^5_0(\eta,1)$, and this completes the proof.
\end{proof}
\noindent The following is the crucial probabilistic result that will be used in the proof of Theorem \ref{thm:3}.

\begin{proposition}
  \label{lem:10}
  For any fixed open or closed rectangle $R\subseteq \RR^2$, we have the convergence $\nu_n(R)\rightarrow  \EE\In^5_0(\eta,1)\leb(R)$ in probability.
\end{proposition}
We now go through a series of lemmas building up towards the proof of Proposition \ref{lem:10}. Throughout these lemmas, we will extensively use Palm measures, and for points $z_1,\dots,z_k\in \RR^2$, we use the notation $\Pi^2_{n,z_1,\dots,z_k}$ to denote the Palm measure corresponding to the Poisson process $\Pi^2_n$ conditioned to have points at $z_1,\dots ,z_k$. Since $\Pi^2_n$ is independent of $\cL$ by Lemma \ref{lem:8}, we will always couple the point process $\Pi^2_{n,z_1,\dots,z_k}$ and $\cL$ such that they are independent. %
We begin by computing $\EE[\nu_n(R)]$, and we note that in the lemmas that follow, $R$ will always denote a fixed open or closed rectangle in $\RR^2$.
\begin{lemma}
  \label{lem:23}
  For all $n\in \NN$, we have $\EE[\nu_n(R)]=\EE\In_0^5(\eta,1)\leb(R)$.
\end{lemma}
\begin{proof}
   By using Campbell's theorem (see for e.g.\ \cite[Theorem 9.1]{LP17}), we obtain that %
  \begin{equation}
    \label{eq:7}
    \EE[\nu_n(R)]=\EE[\sum_{p\in \Pi_n^2\lvert_R}\nu_n(\{p\})]= n\int_{z\in R} \EE\nu_{n,\Pi_{n,z}^2}(\{z\}) d\leb(z).
  \end{equation}
  With $X_n$ being an $\exp(n)$ variable independent of $\cL$, we note by Lemma \ref{lem:5} that $\EE\nu_{n,\Pi_{n,z}^2}(\{z\})=\EE\nu_{n,\Pi_{n,0}^2}(\{0\})=\EE\In^5_0(\eta,X_n)$ and this combined with \eqref{eq:7} yields $\EE[\nu_n(R)]=n\leb(R)\EE\In^5_0(\eta,X_n)$. %
  We now use Lemma \ref{lem:22} to obtain that $\EE\In^5_0(\eta,X_n)=n^{-1}\EE\In^5_0(\eta,1)$ and thereby $\EE[\nu_n(R)]= \EE\In^5_0(\eta,1)\leb(R)$.
\end{proof}
In view of the above lemma, the main estimate needed to obtain Proposition \ref{lem:10} is to show that $\Var(\nu_n(R))\rightarrow 0$ as $n\rightarrow \infty$. We know that
\begin{align}
    \label{eq:16}
    \EE[\nu_n(R)^2]&= \EE[\sum_{p\in \Pi_n^2\lvert_R}\nu_n(\{p\})^2] + \EE[\sum_{p,q\in \Pi_n^2\lvert_R,p\neq q}\nu_n(\{p\})\nu_n(\{q\})].
\end{align}
In the following lemma, we show that the important term in the above expression is the second one.
\begin{lemma}
  \label{lem:24}
  We have $\EE[\sum_{p\in \Pi_n^2\lvert_R}\nu_n(\{p\})^2]\rightarrow 0$ as $n\rightarrow \infty$.
\end{lemma}
\begin{proof}
  By a similar calculation to the one in Lemma \ref{lem:23}, we obtain
  \begin{equation}
    \label{eq:24}
    \EE[\sum_{p\in \Pi_n^2\lvert_R}\nu_n(\{p\})^2]=n\leb(R)\EE\In^5_0(\eta,X_n)^2=2n^{-1}\leb(R) \EE\In^5_0(\eta,1)^2,
  \end{equation}
  where the last equality follows by an application of Lemma \ref{lem:22}. The final expression above converges to $0$ as $n\rightarrow \infty$ and this completes the proof.
\end{proof}
  Intuitively, to control the second term in \eqref{eq:16}, we will need to argue that $\nu_n(\{p\})$ and $\nu_n(\{q\})$ are not correlated if $p$ and $q$ are far from each other. In fact, by using Palm measures, we will be able to reduce the estimation of the above-mentioned term to obtaining good estimates on the quantity
\begin{equation}
  \label{eq:79}
   \EE
    \left[
    \nu_{n,\Pi^2_{n,z_1,z_2}}(\{z_1\})\nu_{n,\Pi^2_{n,z_1,z_2}}(\{z_2\})
      \right]
    \end{equation}
    for fixed $z_1\neq z_2\in \RR^2$. To estimate the above, we will use Lemma \ref{lem:19}, with the intuition being that even if we resample the landscape in a neighbourhood of $z_2$, the behaviour of the Peano curve in a small neighbourhood of $z_1$ is left unchanged with high probability, and this will yield the desired independence between the Peano curve in a small neighbourhood of $z_1$ and a similarly small neighbourhood of a point $z_2\neq z_1$. %
  For fixed $z_1\neq z_2\in \RR^2$, we work in the coupling from Proposition \ref{prop:3} and use Lemma \ref{lem:19} to define the good event $A_\delta\in \sigma(\cL)$ by
  \begin{equation}
    \label{eq:69}
    A_\delta=\cE_\delta^{z_1}\cap \cE_\delta^{z_2}.
  \end{equation}
  By using Lemma \ref{lem:19} and the independence of $\cL$ and $\Pi_{n,z_1,z_2}^2$, we have the following lemma.
  \begin{lemma}
    \label{lem:25}
  We have $\PP(A_\delta)\rightarrow 1$ as $\delta\rightarrow 0$. Also, $A_\delta$ is independent of $\Pi_{n,z_1,z_2}^2$.
  \end{lemma}

  Recall the notation $v_z=\eta^{-1}(z)$ from Lemma \ref{lem:19}. Using the notation
  \begin{align}
    \label{eq:80}
    \Delta v_{z_1}&=\min\{v>v_{z_1}: \eta(v)\in \Pi^2_{n,z_1,z_2}\setminus \{z_1,z_2\}\}-v_{z_1},\nonumber\\
    \Delta v_{z_2}&=\min\{v>v_{z_2}: \eta(v)\in \Pi^2_{n,z_1,z_2}\setminus \{z_1,z_2\}\}-v_{z_2},  
  \end{align}
  we have the following result coming from basic properties of the one dimensional Poisson process $\eta^{-1}(\Pi^2_{n,z_1,z_2}\setminus \{z_1,z_2\})$ along with the independence of $\Pi^{2}_{n,z_1,z_2}$ and $\cL$.
  \begin{lemma}
    \label{lem:26}
    For any fixed $z_1\neq z_2\in \RR^2$, $\Delta v_{z_1},\Delta v_{z_2}$ are both marginally distributed as $\exp(n)$ even conditional on $\cL$. Further, $\Delta v_{z_1},\Delta v_{z_2}$ are conditionally independent given $\cL$ and given
  \begin{displaymath}
    (v_{z_1},v_{z_1}+\Delta v_{z_1})\cap (v_{z_2},v_{z_2}+\Delta v_{z_2})=\emptyset.
  \end{displaymath}
\end{lemma}

We now define the event
\begin{equation}
  \label{eq:81}
  F_\delta=\{\Delta v_{z_1}\leq \delta\} \cap \{\Delta v_{z_2}\leq \delta\} \cap A_\delta,
\end{equation}
and for this event, we have the following lemma.
\begin{lemma}
  \label{lem:27}
  For any fixed $z_1,z_2\in R$ with $|z_1-z_2|> \delta^{1/5}$, the variables $\Delta v_{z_1},\Delta v_{z_2}$ are conditionally independent given $F_\delta=\{\Delta v_{z_1}\leq \delta\} \cap \{\Delta v_{z_2}\leq \delta\} \cap A_\delta$. Further,
  \begin{equation}
    \label{eq:82}
   \PP(F_\delta^c)\leq 2\PP(\exp(n)\geq \delta)+ \PP(A_\delta^c).  
  \end{equation}
\end{lemma}
\begin{proof}
  We know that $A_\delta$ is independent of $\Pi^2_{n,z_1,z_2}$ and combining this with Lemma \ref{lem:26} and the definition of $A_\delta$, we obtain the first statement. The second statement is just a simple union bound using that $\Delta v_{z_1},\Delta v_{z_2}$ are marginally distributed as $\exp(n)$ variables.
\end{proof}
We note that in the expression in \eqref{eq:82}, the first term goes to $0$ as $n\rightarrow \infty$ for any fixed $\delta$, while the latter term goes to $0$ as $\delta\rightarrow 0$. We are now in position to estimate \eqref{eq:79}, and this is the heart of the proof.
\begin{figure}
  \centering
  \includegraphics[width=0.4\textwidth]{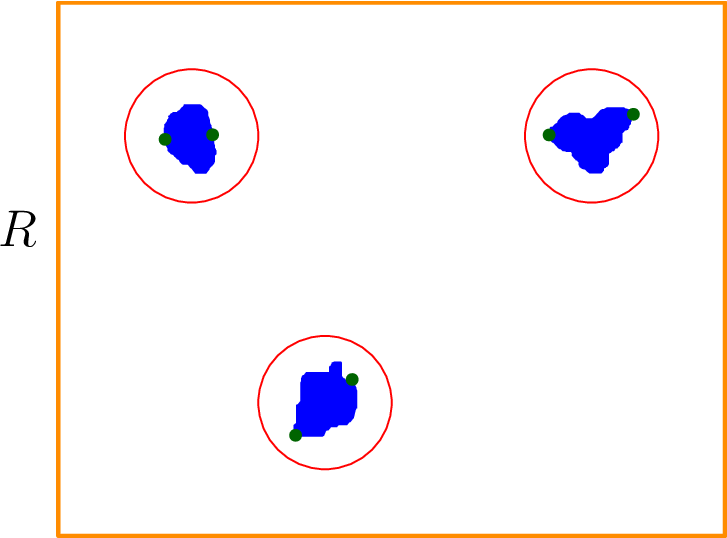}
  \caption{\textbf{The role of Lemma \ref{lem:19}:} Lemma \ref{lem:19} implies that $\eta\lvert_{[v_{z_i}-\delta, v_{z_i}+\delta]}$ are roughly independent for well-separated points $z_i$ and small enough $\delta$. Here, $z_1,z_2,z_3$ are centers of the three circles, and the blue regions are the graphs of  $\eta\lvert_{[v_{z_i}-\delta, v_{z_i}+\delta]}$, and these should be thought of as independent curves if the circles are small enough. This is used in the proofs of Lemmas \ref{lem:28} and \ref{lem:18}.}
  \label{fig:1}
\end{figure}
\begin{lemma}
  \label{lem:28}
  We have
  \begin{displaymath}
   \EE
    \left[
    \nu_{n,\Pi^2_{n,z_1,z_2}}(\{z_1\})\nu_{n,\Pi^2_{n,z_1,z_2}}(\{z_2\})
      \right]\leq  2n^{-2}\left(\EE
      \In^5_{0}(\eta,1)+\sqrt{2}\EE[\In^5_{0}(\eta,1)^2]^{1/2}\sqrt{\PP(A_\delta^c)}\right)^2+n^{-2}o(1),
    \end{displaymath}
  where we use $o(1)$ to denote a term that goes to $0$ when $n$ is sent to $\infty$ first, followed by sending $\delta$ to $0$.
\end{lemma}
\begin{proof}
  With $X_n^1,X_n^2$ denoting independent $\exp(n)$ random variables, for any $z_1,z_2$ with $|z_1-z_2|>\delta^{1/5}$, we obtain
  \begin{align}
    \label{eq:26}
    &\EE
    \left[
    \nu_{n,\Pi^2_{n,z_1,z_2}}(\{z_1\})\nu_{n,\Pi^2_{n,z_1,z_2}}(\{z_2\})
      \right]\nonumber\\
    &=\EE
    \left[
    \nu_{n,\Pi^2_{n,z_1,z_2}}(\{z_1\})\nu_{n,\Pi^2_{n,z_1,z_2}}(\{z_2\});F_\delta
               \right]+\EE
    \left[
    \nu_{n,\Pi^2_{n,z_1,z_2}}(\{z_1\})\nu_{n,\Pi^2_{n,z_1,z_2}}(\{z_2\});F_\delta^c
      \right]\nonumber\\
    &\leq \EE
    \left[
    \nu_{n,\Pi^2_{n,z_1,z_2}}(\{z_1\})\nu_{n,\Pi^2_{n,z_1,z_2}}(\{z_2\}); F_\delta
      \right]+\left(\PP(F_\delta^c)\EE \nu_{n,\Pi^2_{n,z_1}}(\{z_1\})^3\EE \nu_{n,\Pi^2_{n,z_2}}(\{z_2\})^3\right)^{1/3}\nonumber\\
    &=\EE
    \left[
   \In^5_{v_{z_1}}(\eta,X_n^1)\In^5_{v_{z_2}}(\eta,X_n^2)
      ; X_n^1,X_n^2\leq \delta,A_\delta
      \right]+\PP(F_\delta^c)^{1/3}(\EE \nu_{n,\Pi^2_{n,0}}(\{0\})^3)^{2/3}.
  \end{align}
  To obtain the first term in the last expression above, we use  (Lemma \ref{lem:26}) that $\Delta v_{z_1},\Delta v_{z_2}$ are marginally $\exp(n)$, even conditional on $\cL$, along with the fact that $\Delta v_{z_1}, \Delta v_{z_2}$ are conditionally independent given $F_\delta$ (Lemma \ref{lem:27}). On the other hand, the second term is obtained by using the translation invariance of the directed landscape (Lemma \ref{lem:6}). We now note that by Lemma \ref{lem:22}, $\EE \nu_{n,\Pi^2_{n,0}}(\{0\})^3=\EE \In^5_0(\eta, X_n)^3=6n^{-3}\EE\In^5_0(\eta,1)^3$. Using this along with the fact (Lemma \ref{lem:27}) that $\PP(F_\delta^c)\rightarrow 0$ when $n\rightarrow \infty$ followed by sending $\delta$ to $0$, we obtain that $\PP(F_\delta^c)^{1/3}(\EE \nu_{n,\Pi^2_{n,0}}(\{0\})^3)^{2/3}=n^{-2}o(1)$. With this in mind, we write
  \begin{align}
    \label{eq:40}
    \EE
    \left[
    \nu_{n,\Pi^2_{n,z_1,z_2}}(\{z_1\})\nu_{n,\Pi^2_{n,z_1,z_2}}(\{z_2\})
      \right]&\leq  \EE
      \left[
      \In^5_{v_{z_1}}(\eta,X_n^1)\In^5_{v_{z_2}}(\eta,X_n^2)
      ; X_n^1,X_n^2\leq \delta,A_\delta\right]+n^{-2}o(1)\nonumber\\
    &=   \EE
      \left[
       \In^5_{0}(\eta^\mathrm{re}_{z_1},X_n^1)\In^5_{0}(\eta^\mathrm{re}_{z_2},X_n^2)
      ; X_n^1,X_n^2\leq \delta,A_\delta\right]+n^{-2}o(1)\nonumber\\
     &\leq  \EE
      \left[
      \In^5_{0}(\eta^\mathrm{re}_{z_1},X_n^1)\In^5_{0}(\eta^\mathrm{re}_{z_2},X_n^2)
       \lvert X_n^1,X_n^2\leq \delta\right]+n^{-2}o(1),
  \end{align}
  where to obtain the second line, we use the definition of $A_\delta$ in \eqref{eq:69}. Using the independence of $\eta_{z_1}^\mathrm{re}$ and $\eta_{z_2}^\mathrm{re}$, we now obtain
  \begin{align}
    \label{eq:32}
     &\EE
      \left[
      \In^5_{0}(\eta^\mathrm{re}_{z_1},X_n^1)\In^5_{0}(\eta^\mathrm{re}_{z_2},X_n^2)
    \lvert X_n^1,X_n^2\leq \delta\right]\nonumber\\
    &= \EE
      \left[
       \In^5_{0}(\eta^\mathrm{re}_{z_1},X_n^1)
      \lvert X_n^1\leq \delta\right]\EE
      \left[
      \In^5_{0}(\eta^\mathrm{re}_{z_2},X_n^2)
      \lvert X_n^2\leq \delta\right]\nonumber\\
    &=  \EE
      \left[
      \In^5_{0}(\eta^\mathrm{re}_{z_1},X_n^1)
      \lvert X_n^1\leq \delta\right]^2\nonumber\\
    &= \EE
      \left[
       \In^5_{0}(\eta^\mathrm{re}_{z_1},X_n^1)
      ; X_n^1\leq \delta\right]^2\PP(X_n^1\leq \delta)^{-2}\nonumber\\
    &\leq  \left(\EE
      \left[
       \In^5_{0}(\eta^\mathrm{re}_{z_1},X_n^1)
      ; X_n^1\leq \delta, A_\delta\right]+\EE
      \left[
       \In^5_{0}(\eta^\mathrm{re}_{z_1},X_n^1)^2\right]^{1/2}\sqrt{\PP(A_\delta^c)}\right)^2\PP(X_n^1\leq \delta)^{-2}\nonumber\\
    &=\left(\EE
      \left[
      \In^5_{v_{z_1}}(\eta,X_n^1); X_n^1\leq \delta, A_\delta
      \right]+\EE
      \left[
       \In^5_{0}(\eta,X_n^1)^2\right]^{1/2}\sqrt{\PP(A_\delta^c)}\right)^2\PP(X_n^1\leq \delta)^{-2}\nonumber\\
    &\leq 2\left(\EE
      \left[
       \In^5_{v_{z_1}}(\eta,X_n^1) \right]+\EE[\In^5_{0}(\eta,X_n^1)^2]^{1/2}\sqrt{\PP(A_\delta^c)}\right)^2\nonumber\\
    &=2\left(\EE
      \left[
        \In^5_{0}(\eta,X_n^1)
      \right]+\EE[\In^5_{0}(\eta,X_n^1)^2]^{1/2}\sqrt{\PP(A_\delta^c)}\right)^2\nonumber\\
    &= 2n^{-2}\left(\EE
      \left[
        \In^5_{0}(\eta,1)
      \right]+\sqrt{2}\EE[\In^5_{0}(\eta,1)^2]^{1/2}\sqrt{\PP(A_\delta^c)}\right)^2
  \end{align}
  where the factor of $2$ is obtained by noting that $\PP(X_n^1\leq \delta)\geq 1/\sqrt{2}$ for all $n$ large enough compared to $\delta$. Moreover, the penultimate line uses Lemma \ref{lem:5}, while the last line uses Lemma \ref{lem:22}.
 Combining the above with \eqref{eq:40} completes the proof of the lemma. 
\end{proof}
We now use the above lemmas to finish the proof of Proposition \ref{lem:10}.
\begin{proof}[Proof of Proposition \ref{lem:10}]
 
  In view of Lemma \ref{lem:23}, it suffices to show that $\Var(\nu_n(R))\rightarrow 0$ as $n\rightarrow \infty$. This is equivalent to showing that
  \begin{equation}
    \label{eq:84}
    \limsup_{n\rightarrow \infty} \EE[\sum_{p,q\in \Pi_n^2\lvert_R,p\neq q}\nu_n(\{p\})\nu_n(\{q\})]\leq (\EE\In^5_0(\eta,1)\leb(R))^2,
  \end{equation}
  and this is what we shall show.

 By Lemma \ref{lem:22}, we know $\EE
        \In^5_{0}(\eta,X_n^1)=n^{-1} \EE  \In^5_{0}(\eta,1)$ and $\EE
        [\In^5_{0}(\eta,X_n^1)^2]^{1/2}=\sqrt{2}n^{-1} \EE  [\In^5_{0}(\eta,1)^2]^{1/2}$ and thus by using Lemma \ref{lem:28} along with Campbell's theorem, we can write for some constant $C$,
  \begin{align}
    \label{eq:23}
    &\EE[\sum_{p,q\in \Pi_n^2\lvert_R,p\neq q}\nu_n(\{p\})\nu_n(\{q\})]=n^2\int_{z_1,z_2\in R}\EE
    \left[
    \nu_{n,\Pi^2_{n,z_1,z_2}}(\{z_1\})\nu_{n,\Pi^2_{n,z_1,z_2}}(\{z_2\})
      \right]d\leb(z_1,z_2)\nonumber\\
    &\leq C\delta^{2/5} n^2\EE\In^5_{0}(\eta,X_n^1)^2+n^2\int_{z_1,z_2\in R,|z_1-z_2|\geq \delta^{1/5}} \EE
    \left[
    \nu_{n,\Pi^2_{n,z_1,z_2}}(\{z_1\})\nu_{n,\Pi^2_{n,z_1,z_2}}(\{z_2\})
      \right]d\leb(z_1,z_2)\nonumber\\
    &\leq C\delta^{2/5} \EE\In^5_{0}(\eta,1)^2+2\leb(R)^2\left(\EE \In^5_{0}(\eta,1)+ \sqrt{2}\EE[\In^5_{0}(\eta,1)^2]^{1/2}\sqrt{\PP(A_\delta^c)} \right)^2+o(1).
  \end{align}
The first term in the second line above is obtained by using the Cauchy-Schwarz inequality along with the fact that $\EE[ \nu_{n,\Pi^2_{n,z_1,z_2}}(\{z_i\})^2]=\EE\In^5_{0}(\eta,1)^2$ for $i\in \{1,2\}$. By Lemma \ref{lem:25}, we know that $\PP(A_\delta^c)\rightarrow 0$ as $\delta\rightarrow 0$ and since $\delta$ was arbitrary, we obtain %
      \eqref{eq:84}, and this completes the proof.
\end{proof}

We now use measure theoretic arguments to upgrade Proposition \ref{lem:10}, which proves convergence of $\nu_n(R)$ for a single rectangle $R$, into the weak convergence of the entire measure $\nu_n$. Recall that random measures $\psi_n$ on a Polish space are said to converge weakly in probability to a measure $\psi$ if for each bounded continuous function $f$, $\int fd\psi_n$ converges in probability to $\int f d\psi$, and we note that this is equivalent to the statement that $d(\psi_n,\psi)\rightarrow 0$ in probability as $n\rightarrow \infty$, where $d$, the L\'evy-Prokhorov metric, metrizes the weak topology on the space of measures (see for e.g.\ \cite[Section 6]{Bil13}). %
Also, recall that random variables $X_n$ converge to a random variable in $X$ in probability if and only if any deterministic sequence $n_i$ has a deterministic subsequence along which the $X_n$ converge almost surely to $X$. We are now ready to state the lemma.
\begin{lemma}
  \label{lem:11}
  For any fixed closed rectangle $K\subseteq \RR^2$, the measures $\nu_n\lvert_K$ converge weakly in probability to $\EE \In^5_0(\eta,1)$ times the Lebesgue measure on $K$.
\end{lemma}

\begin{proof}
  It suffices to show that any deterministic sequence $\{n_i\}$ has a further deterministic subsequence $\{n'_j\}$ along which the measures $\nu_n\lvert_K$ converge a.s.\ weakly to $\EE \In^5_0(\eta,1)$ times the Lebesgue measure on $K$. We now fix a subsequence $\{n_i\}$. By Proposition \ref{lem:10}, we have $\nu_n(K)\rightarrow \EE \In^5_0(\eta,1)\leb(K)$ in probability and this in particular implies that along a deterministic subsequence $\{n'_j\}\subseteq \{n_i\}$, the $\nu_n\lvert_K$ are a.s.\ tight and thus admit a weak limit which we call $\nu^K$. Since open rectangles $R\subseteq K$ along with the empty set create a $\pi$-system generating the Borel $\sigma$-algebra on $K$, it suffices to show that almost surely, for all open rectangles $R\subseteq K$, we have $\nu^K(R)=\EE \In^5_0(\eta,1)\leb(R)$; indeed, this would characterize $\nu^K$ and thereby imply that almost surely, $\nu_{n_j'}\lvert_K$ converges weakly to $\EE \In^5_0(\eta,1)$ times the Lebesgue measure on $K$, and this would complete the proof.

  Now, we choose a sequence of closed rational rectangles $\underline{R}_k$ such that $\underline{R}_k$ increase to $R$ and open rational rectangles $\overline{R}_k$ such that $\overline{R}_k$ decrease to $R$ as $k\rightarrow \infty$. As a consequence of Proposition \ref{lem:10} and a diagonal argument to replace $\{n'_j\}$ by a further deterministic subsequence $\{n''_j\}$, we know that almost surely, for all rational closed rectangles $R'\subseteq K$, we simultaneously have $\lim_{j\rightarrow \infty} \nu_{n''_j} (R')=\EE \In^5_0(\eta,1)\leb(R')$. Using the above along with the monotonicity of measures and the Portmanteau lemma, we obtain that almost surely,
  \begin{equation}
    \label{eq:2}
          \EE \In^5_0(\eta,1)\leb(\underline{R}_k)= \limsup_{j\rightarrow \infty} \nu_{n''_j}(\underline{R}_k) \leq \nu^K(\underline{R}_k)\leq \nu^K(R)\leq \nu^K(\overline{R}_k)\leq \liminf_{j\rightarrow \infty} \nu_{n''_j}(\overline{R}_k) =\EE \In^5_0(\eta,1)\leb(\overline{R}_k).
        \end{equation}
        Finally, we send $k\rightarrow \infty$ to obtain that $\nu^K(R)=\EE \In^5_0(\eta,1)\leb(R)$ for every open rectangle $R\subseteq K$. This completes the proof.
      \end{proof}

We can now employ another measure theoretic argument to conclude the weak convergence of the measures $\mu_n$ themselves. This is executed in the following proposition, which we in fact consider to be a more natural statement that Theorem \ref{thm:3} itself.
\begin{proposition}
  \label{prop:2}
 For any $M>0$, the measures $\mu_n\lvert_{[-M,M]}$ converge weakly in probability to $\EE \In^5_0(\eta,1)$ times the Lebesgue measure on $[-M,M]$. 
\end{proposition}

\begin{proof}
Again, it suffices to show that any deterministic sequence $\{n_i\}$ has a further deterministic subsequence $\{n'_j\}$ along which the measures $\mu_n\lvert_{[-M,M]}$ converge a.s.\ weakly to $\EE \In^5_0(\eta,1)$ times the Lebesgue measure on $[-M,M]$. We now fix a sequence $\{n_i\}$ and furnish the needed subsequence. By Lemma \ref{lem:12}, we know that almost surely, simultaneously for all intervals $I\subseteq \RR$, $\dim \partial \eta(I)=4/3$ and this in particular implies that $\leb (\partial \eta(I))=0$, and thus by using the Portmanteau lemma along with Lemma \ref{lem:11}, we obtain that for a deterministic subsequence $\{n'_j\}\subseteq \{n_i\}$, $\mu_{n'_j} (I)=\nu_{n'_j}( \eta(I) )\rightarrow  \EE \In^5_0(\eta,1) \leb(I)$ a.s.\ as $j\rightarrow \infty$. Thus we obtain that $\mu_{n'_j} ( [-M,M])$ a.s.\ converges, and as a consequence, we have that $\mu_{n'_j}\lvert_{[-M,M]}$ are a.s.\ tight and thus admit a weak limit which we call $\mu^M$. The rest of the argument is the same as the one in Lemma \ref{lem:11} and thus we only give a sketch. We know that open intervals $J\subseteq [-M,M]$ along with the empty set form a $\pi$-system generating the Borel $\sigma$-algebra and we approximate each such open interval by rational closed intervals from in the inside and rational open intervals from the outside. We then use the Portmanteau lemma long with the monotonicity of measures as in \eqref{eq:2} to obtain that $\mu^M(J)=\EE\In^5(\eta,1) \leb(J)$, thereby characterizing $\mu^M$ and completing the proof.
\end{proof}
\begin{proof}[Proof of Theorem \ref{thm:3}]
  Note that the sum in \eqref{eq:3} is just $\mu_n(I)$ and since $\leb(\partial I)=0$, the required result follows by using Proposition \ref{prop:2} along with the Portmanteau lemma.
\end{proof}
\noindent The aim of the remainder of the section is to show that $\eta$ is not H\"older $1/5$ continuous with respect to the intrinsic metric.

\subsection{Proof of Theorem \ref{thm:5}}
\label{sec:proof2}
The proof of Theorem \ref{thm:5} again crucially uses the independence coming from Lemma \ref{lem:19}. However, in order to be able to use this independence, we do need to argue that the Peano curve does have positive probability of taking unusually large values, and this is the content of the previously stated Proposition \ref{lem:9}, which we will prove by a barrier argument. %
Recall the notion of induced passage times from \eqref{eq:64}. To prepare for the upcoming barrier argument, we first prove a simple lemma regarding induced passage times.
\begin{lemma}
  \label{lem:13}
  Consider a rectangle $R=[x_1,y_1]\times [s_1,t_1]$. For $M>0$ and $n\in \NN$, use $\Across^n(M,R)$ to denote the event on which $\cL^n(x_1,s;y_1,t\vert R),\cL^n(y_1,s;x_1,t\vert R)\leq -M$ for all $s_1\leq s<t\leq t_1$. Then there exists a constant $C$ depending on $M,x_1,s_1,y_1,t_1$ such that for every $M$ large enough depending on $x_1,s_1,y_1,t_1$, and all $n$ large enough, we have $\PP(\Across^n(M,R))>C>0$. 
\end{lemma}

\begin{proof}
This follows from Lemma \ref{l:bstrong}. %
\end{proof}

We now come to the proof of Proposition \ref{lem:9}. The idea here is that if there are two interfaces $\Upsilon_p,\Upsilon_q$ which stay very close but do not meet each other for a long time, then by planarity, the upward semi-infinite geodesics starting from points in between the two interfaces can only have geodesics from other points merging into them if these other points also lie in between $\Upsilon_p,\Upsilon_q$. As a consequence, the Peano curve $\eta$ must cover a large vertical height relatively quickly when started from a point between $\Upsilon_p,\Upsilon_q$.
\begin{figure}
  \centering
  \includegraphics[width=0.8\textwidth]{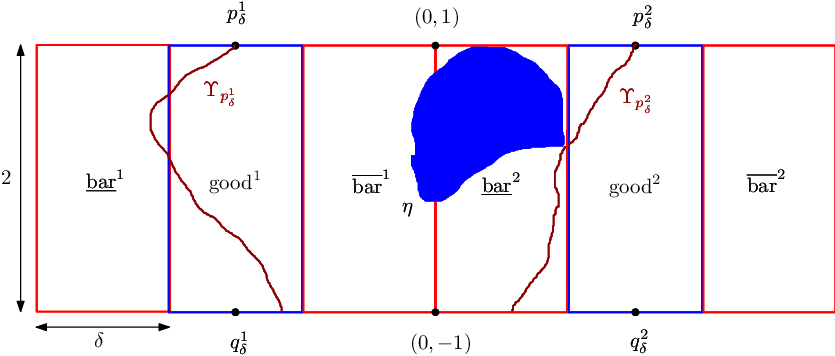}
  \caption{\textbf{Proof of Proposition \ref{lem:9}}: If the interfaces $\Upsilon_{p_{\delta}^1},\Upsilon_{p_\delta^2}$ stay in the regions $\corr^1=\ubarr^1\cup\good^1\cup\obarr^1$ and $\corr^2=\ubarr^2\cup \good^2\cup\obarr^2$ respectively, then the Peano curve $\eta$ (the blue blob) is forced to cover a large vertical height quickly since it would otherwise have to cross one of  $\Upsilon_{p_{\delta}^1},\Upsilon_{p_\delta^2}$.} %
  \label{fig:11}
\end{figure}

\begin{proof}[Proof of Proposition \ref{lem:9}]
  For some fixed $\delta>0$, define the regions $\corr^1,\corr^2$ by
  \begin{equation}
    \label{eq:4}
    \corr^1=[-3\delta,0]\times [-1,1], \corr^2=[0,3\delta]\times [-1,1].
  \end{equation}
   Let $p^1_\delta=(-3\delta/2,1)$ and $p^2_\delta=(3\delta/2,1)$ and define the event $\cE_\delta$ by
   \begin{equation}
     \label{eq:6}
     \cE_\delta=
     \left\{
       \Upsilon_{p^1_\delta}\lvert_{[-1,1]}\subseteq \corr^1, \Upsilon_{p^2_\delta}\lvert_{[-1,1]}\subseteq \corr^2
     \right\}.
   \end{equation}
   We first observe that on the event $\cE_\delta$, we have $\sup_{v\in [0,12\delta]}\eta_u(v)\geq 1$. To see this, we note that since the curve $\eta$ cannot cross the interfaces $\Upsilon_{p_\delta^1},\Upsilon_{p_\delta^2}$ and thus on the event $\cE_\delta$, we must have
   \begin{equation}
     \label{eq:88}
     \{\eta(v):v\in[ 0,\inf_{w\geq 0}\{|\eta_u(w)|=1\}]\}\subseteq [-3\delta,3\delta]\times [-1,1],
   \end{equation}
   which of course implies that $|\eta_u(v)|$ must be equal to $1$ for some $v\leq \leb([-3\delta,3\delta]\times [-1,1])=12\delta$. %

   We now claim that it is sufficient to prove that $\PP(\cE_\delta)>0$ for every $\delta>0$ to complete the proof. To see this, note that $\PP(\cE_\delta)>0$ would imply that $\PP(\sup_{v\in [0,12\delta]}|\eta_u(v)|\geq 1)>0$. By the scaling from (5) in Theorem \ref{thm:1}, we also note that
   \begin{equation}
     \label{eq:9}
     \PP(\sup_{v\in [0,12\delta]}|\eta_u(v)|\geq 1)=\PP(\sup_{v\in [0,1]}|\eta_u(v)|\geq 12^{-3/5}\delta^{-3/5}),
   \end{equation}
   and thus if $\PP(\cE_\delta)>0$ for every $\delta$, then the latter expression would have positive probability for every $\delta$. However, if there existed a constant $K$ such that $\eta_u(1)\leq K$ almost surely, by using that $\eta_{u}(\beta)\stackrel{d}{=}\beta^{3/5}\eta_{u}(1)$, we would have that $|\eta_u(v)|\leq K$ for all $v\in [0,1]\cap \QQ$ almost surely, and by continuity, this would imply that $\sup_{v\in [0,1]}|\eta_u(v)|\leq K$ almost surely. Now, we note that the above is contradicted by the fact that we can take $\delta$ small enough depending on $K$.

   Thus the task is reduced to showing that $\PP(\cE_\delta)>0$ for every $\delta>0$. We now use duality to change the event $\cE_\delta$ to a corresponding event $\widetilde{\cE}_\delta$ about infinite geodesics. With  $q_\delta^1,q_\delta^2$ denoting the points $(-3\delta/2,-1), (3\delta/2,-1)$, we define the event $\widetilde{\cE}_\delta$ by
   \begin{equation}
     \label{eq:10}
     \widetilde{\cE}_\delta= 
      \left\{
       \Gamma_{q^1_\delta}\lvert_{[-1,1]}\subseteq \corr^1, \Gamma_{q^2_\delta}\lvert_{[-1,1]}\subseteq \corr^2
     \right\}.
   \end{equation}
   By an application of the duality between the geodesic tree and the interface portrait (Proposition \ref{prop:4}), we obtain that $\PP(\cE_\delta)=\PP(\widetilde{\cE}_\delta)$. In the remainder of the proof, we use a barrier argument to show that $\PP(\widetilde{\cE}_\delta)>0$ for any $\delta>0$. A very similar barrier argument was used in \cite{BB21}; indeed, the latter work involved placing barriers to restrict a geodesic inside a thin strip, whereas the current setting involves placing barriers to ensure that the geodesics $\Gamma_{q_\delta^1},\Gamma_{q_\delta^2}$ stay inside their respective strips. While the one geodesic case does not directly imply the two geodesic case, the arguments proceed along similar lines and in fact the current setting is easier in the sense that we do not need to quantify the dependence of $\PP(\widetilde{\cE}_\delta)$ on $\delta$. We now move on to the barrier argument, but we note that we present it in a terse and condensed manner since such arguments have already been used earlier. %

 In order to place barriers, we will need to use the FKG inequality, and for this reason, we first work in the discrete-- we use the notation and coupling from Proposition \ref{prop:6}. We define $\widetilde{\cE}^n_\delta$, the pre-limiting analogue of $\widetilde{\cE}_\delta$, by
   \begin{equation}
     \label{eq:61}
     \widetilde{\cE}^n_\delta= 
      \left\{
       \Gamma^n_{q^1_\delta}\lvert_{[-1,1]}\subseteq \corr^1, \Gamma^n_{q^2_\delta}\lvert_{[-1,1]}\subseteq \corr^2
     \right\},
   \end{equation}
   and the goal now is to show that $\PP(\cE_\delta^n)>0$ uniformly for all large $n$, for every fixed $\delta>0$. Indeed, in combination with the geodesic convergence statement from Proposition \ref{prop:6}, the above would immediately yield $\PP(\widetilde{\cE}_\delta)>0$ and complete the proof.
   
   Define the barrier regions $\ubarr^1,\obarr^1,\ubarr^2,\obarr^2$ by
   \begin{equation}
     \label{eq:5}
     \ubarr^1=[-3\delta,-2\delta]\times [-1,1], \obarr^2=[2\delta,3\delta]\times [-1,1],
   \end{equation}
   \begin{equation}
     \label{eq:11}
     \obarr^1=[-\delta,0]\times [-1,1], \ubarr^2=[0,\delta]\times [-1,1].
   \end{equation}
 Define the good regions $\good^1,\good^2$ by
   \begin{equation}
     \label{eq:12}
     \good^1=[-2\delta,-\delta]\times [-1,1],\good^2=[\delta,2\delta]\times [-1,1].
   \end{equation}
  We refer the reader to Figure \ref{fig:11} for a depiction of the above newly defined regions. %
  \begin{figure}
  \centering
  \includegraphics[width=0.8\textwidth]{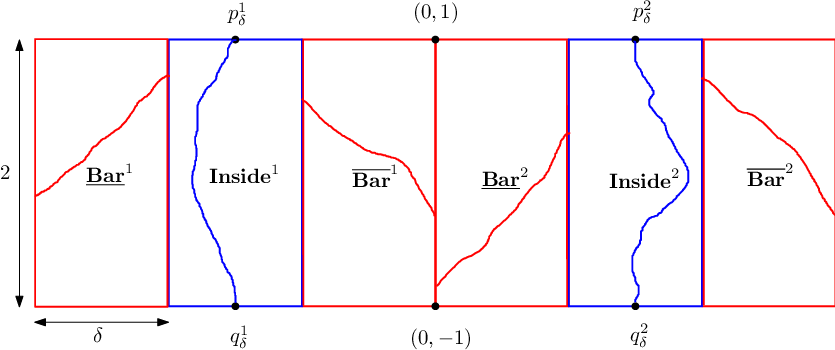}
  \caption{\textbf{The events in the proof of Proposition \ref{lem:9}}: The two $\Inside$ events ensure that there are paths between $q_\delta^1,p_\delta^1$ and $q_\delta^2,p_\delta^2$ staying inside the blue rectangles (the blue paths) which have good lengths. On the other hands, the $\Barr$ events are barriers which ensure that any paths crossing the red regions are highly uncompetitive. The role of the regularity event $\Reg$ (not shown above) is to make sure that paths which do cross a barrier do not make up for the incurred loss in the remaining time when they are not crossing barriers. All the above-mentioned events together ensure that the geodesics $\Gamma_{q_\delta^1}$ and $\Gamma_{q_\delta^2}$ stay inside their respective blue regions.}
  \label{fig:1.1}
\end{figure}
 For a constant $\alpha>0$, define events $\Inside^1,\Inside^2$ by
   \begin{equation}
     \label{eq:31}
     \Inside^1=
     \left\{
       \cL^n(q_\delta^1;p_\delta^1\vert \good^1)\geq \alpha
     \right\}; \Inside^2=
     \left\{
       \cL^n(q_\delta^2;p_\delta^2\vert \good^2)\geq \alpha
     \right\}.
   \end{equation}
   and set $\Inside=\Inside^1\cap \Inside^2$. By translation invariance, we know that $\PP(\Inside^1)=\PP(\Inside^2)$, and we now fix an $\alpha$ for which $\PP(\Inside)=\PP(\Inside^1)^2>c>0$ for all large $n$. In fact, the above holds for any choice of the constant $\alpha$ (see the proof of \cite[Lemma 4.1]{BB21}), but we will not require this fact.
   
   Recall the definition of the scaled discrete Busemann function %
   $\cB^n(\cdot,\cdot)$ from \eqref{eq:66}. %
   For a constant $\beta>0$, we define the events
   \begin{align}
     \label{eq:41}
     \underline{\Reg}^1&=\left\{
     \sup_{-1\leq t\leq 1} 
     \cL^n(q_\delta^1;-2\delta,t)\leq \beta  
     \right\}\cap 
     \left\{
     \sup_{t\in (-1,1)}\cB^n ( (-3\delta,t),p_\delta^1)\leq \beta 
                         \right\},\nonumber\\
     \overline{\Reg}^1&=
     \left\{
     \sup_{-1\leq t\leq 1} 
     \cL^n(q_\delta^1;-\delta,t)\leq \beta  
     \right\}\cap 
     \left\{
     \sup_{t\in (-1,1)}\cB^n ( (0,t),p_\delta^1)\leq \beta 
                        \right\},\nonumber\\
      \underline{\Reg}^2&=
     \left\{
     \sup_{-1\leq t\leq 1} 
     \cL^n(q_\delta^2;\delta,t)\leq \beta  
     \right\}\cap 
     \left\{
  \sup_{t\in (-1,1)}\cB^n ( (0,t),p_\delta^2)\leq \beta 
     \right\},\nonumber\\
     \overline{\Reg}^2&=
     \left\{
     \sup_{-1\leq t\leq 1} 
     \cL^n(q_\delta^2;2\delta,t)\leq \beta  
     \right\}\cap 
     \left\{
 \sup_{t\in (-1,1)}\cB^n ( (3\delta,t),p_\delta^2)\leq \beta 
                         \right\}
   \end{align}
   respectively,
 and we define $\Reg$ to be the intersection of the above four events. We fix $\beta$ large enough such that
   \begin{equation}
     \label{eq:42}
     \PP(\Reg\lvert \Inside)>0
   \end{equation}
for all large $n$. Indeed, such a $\beta$ does exist since both $\cL^n,\cB^n$ are tight in $n$ (see Proposition \ref{prop:6} and \cite[Proposition 35]{Bha23}) and moreover, we can choose $\beta$ large enough such that the analogue of the event $\Reg$ for the landscape has probability at least $1-c/2$, where we recall that $\PP(\Inside)\geq c$ for all $n$ large enough. Thus for the above choice of $\beta$, we would have $\PP(\Reg\cap \Inside)>c/2>0$ for all large $n$. 

We now define the barrier event $\Barr=\underline{\Barr}^1\cap \overline{\Barr}^1\cap \underline{\Barr}^2\cap \overline{\Barr}^2$, where the individual events are defined in terms of the notation from Lemma \ref{lem:13} by
   \begin{align}
     \label{eq:43}
     \underline{\Barr}^1&= \Across^n(M,\ubarr^1); \overline{\Barr}^1= \Across^n(M,\obarr^1),\nonumber\\
     \underline{\Barr}^2&= \Across^n(M,\ubarr^2); \overline{\Barr}^2= \Across^n(M,\obarr^2).
   \end{align}
 To gain a quick understanding of the above defined events, the reader could refer to Figure \ref{fig:1.1}. We fix $M$ to be a large constant satisfying $2\beta-M< \alpha$ and note that for this choice of $M$, we have $\PP(\Barr)>0$ as a consequence of Lemma \ref{lem:13} and the FKG inequality. We now define the favourable event $\Fav=\Barr \cap \Reg\cap \Inside$ and write
   \begin{equation}
     \label{eq:44}
     \PP(\Fav)=\PP(\Inside)\PP(\Barr\cap \Reg\vert \Inside).
   \end{equation}
   Now note that $\Inside$ is measurable with respect to $\sigma(\cL(\cdot;\cdot\vert \good^1),\cL(\cdot;\cdot\lvert \good^2))$. Also, conditional on the same $\sigma$-algebra, $\Barr$ and $\Reg$ are both decreasing events in the remainder of the discrete configuration $X^n$. Thus by the FKG inequality, we obtain
   \begin{align}
     \label{eq:45}
     \PP(\Barr\cap \Reg\vert \Inside)&\geq \PP(\Barr\vert \Inside)\PP(\Reg\vert \Inside)\nonumber\\
     &=\PP(\Barr)\PP(\Reg\vert \Inside)>0,
   \end{align}
uniformly for large $n$, where we used that $\Barr$ is independent of $\Inside$. We now show that $\widetilde{\cE}^n_\delta\supseteq \Fav$. By symmetry, it suffices to show that $\Gamma^n_{q^1_\delta}\lvert_{[-1,1]}\subseteq \corr^1$ on the latter event. In fact, we only show that $\inf_{s\in[-1,1]}\Gamma^n_{q_\delta^1}(s)\geq -3\delta$ on the above event. %
 An almost identical proof %
 shows that $\sup_{s\in [-1,1]}\Gamma^n_{q_\delta^1}(s)\leq 0$ on the event $\Fav$. 

 With the eventual goal of detecting a contradiction, assume that $\inf_{s\in[-1,1]}\Gamma^n_{q_\delta^1}(s)< -3\delta$ is possible on the event $\Fav$ and let $t_1$ be the smallest number for which $\Gamma^n_{q^1_\delta}(t_1)=-3\delta$ and let $t_2=\sup\{t<t_1:\Gamma^n_{q_\delta^1}(t)=-2\delta\}$. %
 It follows that the restriction of the geodesic $\Gamma^n_{q_\delta^1}$
 on $[t_1,t_2]$ is contained in $\ubarr^1$. Therefore, on the event $\mathbf{Fav}$, we can write
   \begin{align}
     \label{eq:46}
    \cB^n( q_\delta^1,p_\delta^1) &= \cL(q_\delta^1;\Gamma_{q_\delta^1}^n(t_2),t_2)+\cL^n(\Gamma^n_{q^1_\delta}(t_2),t_2;\Gamma^n_{q^1_\delta}(t_1),t_1)+\cB^n( (\Gamma_{q_\delta^1}^n(t_1),t_1),p_\delta^1)\nonumber\\   
     &<2\beta-M<\alpha\leq \cL^n(q_\delta^1;p_\delta^1\vert \good^1)\leq \cL^n(q_\delta^1;p_\delta^1),
   \end{align}
  and this contradicts the reverse triangle inequality. %
   This shows that the inequality $\inf_{s\in[-1,1]}\Gamma^n_{q_\delta^1}(s)\geq -3\delta$ holds on $\Fav$, and arguing the other cases similarly we get $\widetilde{\cE}^n_\delta\supseteq \Fav$. This implies $\PP(\widetilde{\cE}_\delta^n)>0$, thereby completing the proof. %
 \end{proof}

\noindent We now combine the above with the independence coming from Lemma \ref{lem:19} and then go on to finally complete the proof of Theorem \ref{thm:5}.
 \begin{lemma}
   \label{lem:18}
   Fix an open set $U\subseteq \RR^2$. For $M>0$, let $A_M$ denote the event that there almost surely exists a sequence of points $z'_m\in U$ for which $|\eta_u(v_{z'_m}+m^{-1})-z'_m|^{5/3}\geq Mm^{-1}$. Then $A_M$ occurs almost surely for any fixed $M>0$. 
 \end{lemma}

 \begin{proof}
   Let $K\in \NN$ be a parameter which will be sent to $\infty$ at the end of the argument. Choose and fix $K$ distinct points $\{z_1,\dots,z_{K}\}\subseteq U$. Now, using $v_{z_i}$ to denote the a.s.\ unique point $\eta^{-1}(z_i)$ as in Lemma \ref{lem:19}, we note that there is an event $E_\delta$, on which we have %
   \begin{equation}
     \label{eq:48}
     \eta\lvert_{[v_{z_i}-\delta,v_{z_i}+\delta]}(v_{z_1}+\cdot)=\eta^\mathrm{re}_{z_i}\lvert_{[-\delta,\delta]}(\cdot),
   \end{equation}
   and the event $E_\delta$ satisfies $\PP(E_\delta)\rightarrow 1$ as $\delta\rightarrow 0$ for any fixed $K$. Note that the latter curves in \eqref{eq:48} are independent of each other and thus we can write 
   \begin{align}
     \label{eq:49}
     \PP( |\eta^\mathrm{re}_{z_i,u}(\delta)-z_i|^{5/3}\geq M\delta \text{ for some } i )& = 1 - \PP( |\eta_u(\delta)|^{5/3} \leq M\delta)^K\nonumber\\
     &= 1 - \PP( |\eta_u(1)|^{5/3} \leq M)^K,
   \end{align}
   where $\eta_{z_i,u}^\mathrm{re}$ denotes the second coordinate of $\eta_{z_i}^{\mathrm{re}}$, and we note that we have used the scaling property of $\eta$ from Theorem \ref{thm:1} (5). Thus we obtain that for any fixed $K$ and any $N\in \NN$,
   \begin{align}
     \PP(A_M)&\geq \liminf_{N\rightarrow \infty}
     \PP( |\eta_u(v_z+m^{-1})-v_z|^{5/3}\geq Mm^{-1} \text{ for some } m> N \text{ and } z\in U)\nonumber\\
     &\geq \liminf_{N\to \infty}\left( \PP( |\eta^\mathrm{re}_{z_i,u}(N^{-1})-z_{i}|^{5/3}\geq MN^{-1} \text{ for some } i)-\PP(E^{c}_{1/N})\right)\nonumber\\
     &\geq 1 - \PP( |\eta_u(1)|^{5/3} \leq M)^K,
    \label{eq:50}
   \end{align}
   where in the second inequality we have used \eqref{eq:48} and in the third inequality we have used \eqref{eq:49} together with the fact the $\PP(E^{c}_{1/N})\to 0$. Finally, we note that as a consequence of Proposition \ref{lem:9}, the right-hand side above converges to $0$ as $K\rightarrow \infty$, and this completes the proof.  
 \end{proof}

 \begin{proof}[Proof of Theorem \ref{thm:5}]
By using Theorem \ref{thm:1} (4) and (2), for any fixed $v_0\in \RR$, almost surely, there is no $w\neq v_0$ satisfying $\eta(w)=\eta(v_0)$. By using the above along with Fubini's theorem, the set of $v\in \inte I$ such that $\eta(v)$ is visited by $\eta$ exactly once form a full measure subset of $I$. In particular, this implies that there is an Euclidean ball $B\subseteq \RR^2$ with rational center and rational radius %
such that $\eta(v)\in B\subseteq \eta(I)$. We now apply Lemma \ref{lem:18} with the open sets $U$ being balls with rational centers and radii to obtain that $\eta_u\lvert_I$ is not $3/5$ H\"older continuous with respect to the Euclidean metric. As a consequence, we obtain that $\eta\lvert_I$ is not $1/5$ H\"older continuous with respect to the intrinsic metric $d_\mathrm{in}$ defined in \eqref{eq:54}.
\end{proof}
\begin{remark}
  \label{rem:unb}
  In the above proof, to show that $\eta$ is not $1/5$ locally H\"older continuous with respect to the intrinsic metric, we showed that $\eta_u$ is not locally $3/5$ H\"older continuous anywhere. Though we do not pursue this, it can be shown by an analogous argument that $\eta_h$ is not locally $2/5$ H\"older continuous anywhere. Indeed, the main task is to establish the unboundedness of the law of $\eta_h(1)$ and this can be done by a similar argument as in the proof of Proposition \ref{lem:9}, with the modification that all the rectangles in Figures \ref{fig:1}, \ref{fig:1.1} are tilted instead of lying vertically upwards. This would indeed force $\eta$ to cover a large horizontal distance since the $\Upsilon_{p_{\delta}^1},\Upsilon_{p_\delta^2}$ would themselves do so, and $\eta$ would be sandwiched between these.
\end{remark}

\printbibliography
\end{document}